\newif\ifpreprint
\preprintfalse
\preprinttrue % Uncomment for preprint mode
\ifpreprint
\documentclass{article}
\usepackage[margin=1.2in]{geometry}
\usepackage{natbib}
\setcitestyle{authoryear, aysep={ }, comma}
\usepackage{authblk}
\else
\documentclass[aos]{imsart}
\usepackage[authoryear]{natbib}
\fi

\usepackage[T1]{fontenc}
\usepackage{graphicx}
\usepackage{amsmath,amsfonts,amssymb,stmaryrd,amsthm,thmtools}
\usepackage{thm-restate}
\usepackage[inline]{enumitem}
\usepackage[acronym,abbreviations,prefix]{glossaries-extra}

\declaretheorem[numberwithin=section]{theorem}
\declaretheorem[numberwithin=section,sibling=theorem]{lemma}
\declaretheorem[numberwithin=section,sibling=theorem]{corollary}

\declaretheorem[style=definition,numberwithin=section,sibling=theorem]{definition}
\declaretheorem[style=definition,numberwithin=section,sibling=theorem]{remark}

\ifpreprint
\usepackage{color, xcolor}
\definecolor{darkgreen}{rgb}{0,0.4,0}
\usepackage[
    colorlinks,
    breaklinks,
    linkcolor=darkgreen,
    citecolor=darkgreen,
    urlcolor=darkgreen,
    linktocpage
]{hyperref}
\else
\usepackage[colorlinks,breaklinks,citecolor=blue,urlcolor=blue]{hyperref}%% uncomment this for coloring bibliography citations and linked URLs
\fi
\usepackage[noabbrev,capitalise,nosort,nameinlink]{cleveref}
\crefname{equation}{}{}  % do not write Equation when referencing an equation

\newcommand{\MM}[1]{\bgroup\color{olive}MM:~#1\egroup}
\newcommand{\pfm}[1]{\bgroup\color{teal}PFM:~#1\egroup}
\newcommand{\seb}[1]{\bgroup\color{blue}Seb:~#1\egroup}

\newcommand{\sigAlg}{\mathcal A}
\newcommand{\filtration}{\mathcal F}
\newcommand{\E}{\mathbb E}
\newcommand{\probaMeasures}{\mathcal{M}_1^+}
\newcommand{\subsetProbaMeasures}{\mathcal{P}}
\newcommand{\R}{\mathbb{R}}
\newcommand{\Rnn}{\R_{\geq 0}}
\newcommand{\Rp}{\R_{>0}}
\newcommand{\N}{\mathbb{N}}
\newcommand{\Np}{\N_{>0}}
\newcommand{\dd}{\mathrm{d}}
\newcommand{\rBoundedStoppingTimes}{\mathcal{T}}
\newcommand{\norm}[1]{\lVert#1\rVert}
\newcommand{\integers}[1]{\llbracket#1\rrbracket}

\newcommand{\pwMax}{\vee}
\newcommand{\pwMin}{\wedge}

\newcommand{\indicator}{\mathbf{1}}
\newcommand{\ess}{\mathrm{ess}}
\newcommand{\frm}{\mathrm{fin}}

\ifpreprint
\newcommand\blfootnote[1]{%
  \begingroup
  \renewcommand\thefootnote{}\footnote{#1}%
  \addtocounter{footnote}{-1}%
  \endgroup
}
\newcommand{\keywords}[1]{%
  \par\vspace{0.5em}
  {\small\noindent\textbf{Keywords:} #1\par}
}
\newcommand{\msc}[1]{%
  \par\vspace{0.25em}
  {\small\noindent\textbf{2020 Mathematics Subject Classification:} #1\par}
}
\fi
\newcommand{\preprintdot}{\ifpreprint.\fi}

\setabbreviationstyle[acronym]{long-short}
\glsdisablehyper
\newacronym{as}{a.s.}{almost surely}
\newacronym{asp}{ASP}{asymptotic supermartingale property}
\newabbreviation[prefixfirst={a\space},prefix={an\space}]{asm}{ASM}{asymptotic supermartingale}
\newacronym{rhs}{RHS}{right-hand side}
\newacronym{savi}{SAVI}{safe anytime-valid inference}
\newacronym{saep}{SAEP}{strongly $r$-asymptotic e-process}
\newacronym{aep}{AEP}{$r$-asymptotic e-process}
\newacronym[longplural={strongly $r$-asymptotic p-processes}]{sapp}{SAPP}{strongly $r$-asymptotic p-process}
\newacronym[longplural={$r$-asymptotic p-processes}]{app}{APP}{$r$-asymptotic p-process}

\begin{document}
\ifpreprint\else\begin{frontmatter}\fi
\title{Asymptotic e-processes}
\ifpreprint
    \author{Pierre-Fran\c{c}ois Massiani}
    \author{Sebastian Schulze}
    \author{Mattes Mollenhauer}
    \date{\today}
    \affil{
        \small Merantix Momentum Research
    }
    \maketitle
    \blfootnote{Correspondence: \texttt{\{pierrefrancois.massiani,sebastian.schulze,mattes.mollenhauer\}@merantix-momentum.com}}
\else
    \runtitle{Asymptotic e-processes}
    \begin{aug}
    \author[A]{\fnms{Pierre-Fran\c{c}ois}~\snm{Massiani}\ead[label=e1]{pierrefrancois.massiani@merantix-momentum.com}}
    \author[A]{\fnms{Sebastian}~\snm{Schulze}\ead[label=e2]{sebastian.schulze@merantix-momentum.com}}
    \author[A]{\fnms{Mattes}~\snm{Mollenhauer}\ead[label=e3]{mattes.mollenhauer@merantix-momentum.com}}
    \address[A]{Merantix Momentum Research\printead[presep={,\ }]{e1,e2,e3}}
    \end{aug}
\fi

\begin{abstract}
We investigate the concept of an \emph{asymptotic e-process},
which is a doubly-indexed stochastic process
$(E_{m,n})_{m,n\in\N}$ that possesses, asymptotically for an approximation index $m\to\infty$, the properties of an
e-process along a monitoring time index $n$.
This constitutes the first in-depth study of this recently introduced concept, which is relevant in \emph{asymptotic sequential anytime-valid inference}.
Our theory is motivated by practical applications
in sequential hypothesis testing, 
in which e-variables and e-processes 
can only be constructed approximately 
from observations 
due to model misspecification or estimation errors.
Technically, asymptotic e-processes satisfy an asymptotic version of
Ville's inequality, which
bounds excursion probabilities of $(E_{m,n})_{m,n\in\N}$
uniformly over $n$ up to
a \emph{monitoring time horizon} $r_m$.
We show the necessity of allowing for \emph{finite} values of $r_m$,
recovering truly anytime-valid guarantees asymptotically if $r_m\to\infty$.
We derive various properties of asymptotic e-processes, and study
their connections to asymptotic
supermartingales.
We also investigate general methods for their construction such as calibration, the cumulative product of asymptotic e-variables, and the monitoring an of an e-process that depends on an estimated parameter. The latter construction
constitutes a generalization of a recent approach
within the context of asymptotic post-hoc inference.
\end{abstract}

\ifpreprint
    \msc{62L10, 60G07, 60G42}
    {
    \small 
    \setcounter{tocdepth}{2}
    \tableofcontents
    }
\else
    
    \end{frontmatter}
\fi

\section{Introduction}
\emph{Sequential testing} addresses the problem of statistical decision-making in settings where observations are collected continuously over time, requiring repeated assessment of whether accumulated evidence warrants rejection of the null hypothesis.
Such tests arise naturally in a wide range of settings where data are collected sequentially, from monitoring patients in clinical trials and online A/B user studies to scientific experimentation.
Starting with the work of \cite{waldSequentialMethodSampling1945}, sequential tests have been characterized through \emph{stopping times}, that is, data-dependent (possibly infinite) times at which the null is rejected.
A recurring source of criticism, however, is that guarantees are tied to a stopping rule that must be fully specified before any data are observed.
In practice, this separation is routinely violated, invalidating the guarantees.
As \citet[Section 7.10]{ramdasHypothesisTestingEvalues2025} illustrate,
even well-meaning practitioners can fall into this trap.
This phenomenon is arguably a central cause of irregularities in the reporting of statistical evidence \citep{shaferLanguageBettingStrategy2019}.

\paragraph*{SAVI and e-processes\preprintdot}
The recent framework of \emph{\gls{savi}} addresses this issue by introducing sequential statistical objects (such as measures of evidence
and corresponding confidence sequences) that provide appropriate guarantees at any point in time, regardless of the stopping rule applied.
A central class of such objects is that of \emph{e-processes}
\citep{howardTimeUniformNonparametricNonasymptotic2021,ramdasAdmissibleAnytimevalidSequential2022,ramdas2022exchangeability,grunwaldSafeTesting2024}, which can be constructed based on (sequential) \emph{e-variables} \citep{vovk2021evalues}.
Intuitively, e-processes accumulate evidence against the null hypothesis; they should remain small uniformly across time if the null holds, and grow if it does not.
Formally, an e-process $(E_n)_{n \in \N}$ is a nonnegative stochastic
process satisfying $\E_P[E_\tau] \leq 1$ for every stopping
time $\tau$ and distribution $P$ in the null.
Their connection to supermartingales makes e-processes theoretically and practically tractable, and \emph{Ville's inequality}, 
which bounds their excursion probability past any fixed threshold, provides the central instrument for translating them into practical tests: for all $\alpha > 0$ and all distributions $P$ in
the null, e-processes satisfy
\begin{equation*}
    P \left[\sup_{n \in \N} E_n \geq \frac{1}{\alpha} \right]
    \leq \alpha.
\end{equation*}
A game-theoretic view established in the \emph{testing-by-betting} framework \citep{shafer2019foundations,shafer2021betting, shekharNonparametricTwoSampleTesting2025} provides another tool for interpretation and analysis of e-processes as an adversarial game.
E-processes, \gls{savi}, and the testing-by-betting framework have seen rapid development in recent years; we refer to \cite{ramdasHypothesisTestingEvalues2025} and the references therein (particularly, Section 1.8) for a comprehensive review.

\paragraph*{Asymptotic SAVI\preprintdot}
In practice, however, the construction of e-processes often requires precise knowledge of quantities that are unavailable and have to be estimated from data instead.
Establishing finite-sample guarantees based on these estimates is challenging and often requires additional constraints such as restricting the null hypothesis \citep{massianiKernelConditionalTests2025}.
This mirrors classical hypothesis testing, where finite-sample tests, particularly under large composite nulls or in nonparametric settings, are often intractable, and one often resorts instead to \emph{asymptotic tests}. 
Intuitively, such tests are obtained by substituting consistent estimators for unknown quantities and recover validity in the limit of infinite data.
Yet, while asymptotic testing has a rich and well-developed theory in the classical setting, the corresponding framework for \gls{savi} is less mature and has only been partially investigated in recent works.
\cite{ignatiadisAsymptoticCompoundEvalues2024} introduce approximate and asymptotic e- and p-variables. While this facilitates an approximate treatment of the building blocks of \gls{savi}, it stops short of addressing asymptotic anytime-validity directly. 
\cite{waudby-smithDistributionuniformAnytimevalidSequential2026} propose distribution-uniform anytime p-values satisfying anytime-validity asymptotically with respect to an approximation index. 
\cite{choGAAVIGlobalAsymptotic2026} propose an implementation that achieves a similar practical goal but without uniformity over stopping rules, making it less directly comparable to the \gls{savi} framework.
\citet{chuggPostHocLargeSampleStatistical2026} are the first (to our knowledge)
to propose the notion and an example of asymptotic e-processes, building upon
\citet{ignatiadisAsymptoticCompoundEvalues2024} and
intuitively combining the defining properties
of e-processes and of asymptotic e-variables. 
Nevertheless, asymptotic e-processes still lack a thorough and systematic study of their properties and construction.

\begin{remark}[Update of this manuscript and connection
to \citealt{chuggPostHocLargeSampleStatistical2026}]
    Shortly after the release of the first version of this manuscript,
    we became aware that the term \emph{asymptotic e-processes} had
    been proposed several weeks earlier by \citet{chuggPostHocLargeSampleStatistical2026}
    for the first time in a different context in terms
    of the property \cref{eq:asymptotic e-process intro} discussed below. 
    While the work by \cite{chuggPostHocLargeSampleStatistical2026} is 
    predominantly focused on asymptotic post-hoc inference and explores 
    connections to relevant objects in this context, our work 
    focuses on asymptotic e-processes as a primary object of study
    and investigates their properties and general construction methods,
    with the goal of 
    reflecting existing theory of the nonasymptotic case.
    Crucially, this goal calls for a definition more general than that of \eqref{eq:asymptotic e-process intro} and \cite{chuggPostHocLargeSampleStatistical2026}, as we illustrate in the next paragraph.
    Independently of this fact,
    we want to emphasize that the work of \cite{chuggPostHocLargeSampleStatistical2026} inspired us to investigate non-strongly asymptotic e-processes more thoroughly (cf.\ \cref{def:asymptotic e-process}), as well as a general method for constructing asymptotic e-processes via burn-in and domination, which we added as \cref{subsec:event partitioning}.
\end{remark}

\paragraph*{Unbounded monitoring horizons are too restrictive for finite approximation indices\preprintdot}
Recent definitions in the context of asymptotic \gls{savi} are stated as limiting conditions that should hold asymptotically in an approximation index.
Such conditions involve arbitrary stopping times \citep{chuggPostHocLargeSampleStatistical2026}, or a time-uniform quantifier \citep{waudby-smithDistributionuniformAnytimevalidSequential2026}.
Despite being asymptotic, these conditions have pre-asymptotic consequences: for them to hold asymptotically, the underlying object must satisfy nontrivial conditions for finite approximation indices.
For instance, a sequence of expectations must first become finite to be less than $1$ in limit superior.
While such constraints can be imposed axiomatically and are satisfied in practice by some objects, they are overly restrictive for many constructions.
We summarize this concern as follows:
\begin{quote}
    For asymptotic e-processes and related objects, properties related to monitoring times that may be unbounded or infinite should only be required to hold in the limit of the approximation index.
    Requiring such properties to hold at any finite approximation index, even implicitly as a consequence of an asymptotic condition, is too strong for an axiomatic definition.
\end{quote}

To illustrate this point and ground the discussion, we now consider the definition of
an asymptotic e-process by \citet[Definition 4.2]{chuggPostHocLargeSampleStatistical2026}; a similar argument could be made for the definition of anytime p-values in \cite{waudby-smithDistributionuniformAnytimevalidSequential2026}.
The definition by \citet{chuggPostHocLargeSampleStatistical2026}
involves a doubly-indexed nonnegative random process $E = (E_{m,n})_{m,n\in I}$ with \emph{monitoring time index} ${n \in \N}$ and \emph{approximation index} ${m \in \N}$. 
Here, $I = \{(m,n)\in\N^2\mid m\leq n\}$ is the index set of upper-triangular arrays; this structure does not play a role in this argument, but is the one considered in the reference.
Once reformulated in the vocabulary of \cite{ignatiadisAsymptoticCompoundEvalues2024}, the definition imposes that, for any stopping time\footnote{Note that this implicitly assumes that there is a single filtration $\filtration = (\filtration_{n})_{n\in\N}$ such that $E_{m,\bullet}$ is adapted to $\filtration$ for all $m\in\N$, which we assume in this paragraph but not in the main body of this work.} $\tau$, the variable $(E_{m,\tau\pwMax m})_{m\in\N}$ is a uniformly strongly asymptotic e-variable; that is, \begin{equation}\label{eq:asymptotic e-process intro}
    \limsup_{m\to\infty}\E_P[E_{m,\tau\pwMax m}]\leq 1.
\end{equation}
This definition allows unbounded and possibly infinite stopping times.
In particular, for every such $\tau$, there must exist $m_0\in\N$ such that the expectation in \cref{eq:asymptotic e-process intro} is finite for all $m\geq m_0$.
For instance, this is the case for the stopping time $\tau$ constant equal to $\infty$,
illustrating that \cref{eq:asymptotic e-process intro} indeed imposes a property over infinite monitoring times for finite approximation indices.
In fact, a consequence of one of our results is that processes that satisfy \cref{eq:asymptotic e-process intro} can be \emph{renormalized} into true e-processes; there exists $C\in\Rnn$ (independent of $\tau$) such that $C\cdot E$ is an e-process (cf.\ \cref{rmk:infinity-asymptotic e-processes}).
This intuition already highlights that \cref{eq:asymptotic e-process intro} is a very strong requirement and excludes important and natural constructions that require a genuinely asymptotic notion, for example the cumulative product of sufficiently regular asymptotic e-variables.
\par
Let us illustrate this in more detail by investigating the cumulative product construction, which is arguably the cornerstone of e-process theory within
the context of supermartingales and
the testing-by-betting framework.
Assume that we have a collection of independent asymptotic e-variables $(e_{m,n})_{m,n\in\N}$
such that $\E_P[e_{m,n}] = 1 + \varepsilon_{m}$ for all $m,n \in \N$, with vanishing approximation error $\varepsilon_{m} > 0$ and $\varepsilon_m \to 0$ as $m \to \infty$.
We define the cumulative product
\begin{equation*}
    E_{m,n} :=  
    \prod_{i=0}^n
    e_{m,i}.
\end{equation*}
We see that, for every $m,n \in \N$, we have
$
    \E_P[E_{m,n}] = (1 + \varepsilon_m)^{n+1},
$
and hence for any almost surely finite but unbounded stopping time $\tau$
that is independent of $E$, we get
\begin{equation}\label{eq:diverging series}
    \E_P[E_{m,m\pwMax\tau}] = 
    \sum_{n =0}^\infty
    P[m\pwMax\tau = n] \,
    (1 + \varepsilon_m)^{n+1} = \sum_{n = m}^\infty
    P[\tau = n] \,
    (1 + \varepsilon_m)^{n+1},
\end{equation}
where the \gls{rhs} diverges for every fixed $m$ if the tail $P[\tau = n]$ does not decay fast enough as $n\to\infty$,
clearly conflicting with \cref{eq:asymptotic e-process intro}.
\par
The cause for this is clear: the excess $\varepsilon_m$ in the expectation of each factor compounds over time, and for any fixed $m$ this accumulation eventually dominates. Condition \cref{eq:asymptotic e-process intro} offers no way to account for this compounding effect, making it incompatible with the cumulative product construction for any finite $m$.
The key insight motivating our definition is that this compounding can be controlled by introducing a monitoring horizon $r_m \in \N$ that grows with $m$ but slowly enough that the accumulated error remains negligible. 
Specifically, we show that if $r_m \cdot \varepsilon_m \rightarrow 0$ as $m \rightarrow \infty$, the diverging series is effectively truncated before the error accumulates, and an asymptotic analogue of Ville's inequality can be recovered up to time $r_m$ under appropriate conditional assumptions (\cref{thm:cumulative product}). 
In the context of \cref{eq:asymptotic e-process intro}, this corresponds to imposing an \emph{upper bound} on the stopping time $\tau$, which evolves with the approximation index $m\in\N$.
As $m \rightarrow \infty$ and the approximation improves, $r_m \rightarrow \infty$ as well, so the guarantees become truly anytime-valid, but only in the limit, consistently with the concern raised at the beginning of this paragraph.
This is the central idea behind our definition of an asymptotic e-process, which we develop formally in \cref{sec:asymptotic-e-processes}.

\paragraph*{This work: $r$-asymptotic e-processes\preprintdot}

We propose the concept of a \emph{(uniformly) \gls{saep}}
as a doubly-indexed nonnegative process $(E_{m,n})_{m,n \in \N}$
that possesses the defining property of an e-process in a suitable limiting
sense whenever $m \to \infty$: informally,
we require the condition
\begin{equation*}
    \limsup_{m \to \infty} 
    \E_P[E_{m,\tau_m}]
    \leq 1.
\end{equation*}
\emph{over sequences $(\tau_m)_{m \in \N}$ 
of stopping times that are pointwise bounded by
the values of the extended integer horizon sequence 
$r = (r_m)_{m \in \N}
\subset \N \cup \{\infty\} $ for every $m \in \N$}.
The terminology is based on the work of \cite{ignatiadisAsymptoticCompoundEvalues2024}, as the definition involves uniformly strongly asymptotic e-variables.
We also introduce
the weaker notion of \emph{(uniformly) \glspl{aep}}, for which the defining property only holds for the thresholded process $(E_{m,n}\pwMin t)_{m,n\in\N}$, where $t\in\Rnn$ is arbitrary.
Clearly, the definition of \citet[Definition 4.2]{chuggPostHocLargeSampleStatistical2026}
is recovered for the constant sequence $r_m \equiv \infty$ for all $m \in \N$, and the example they construct in Proposition 4.4 is a uniformly $r$-asymptotic e-process with $r\equiv\infty$.
In our work, however, we reserve the generic term ``asymptotic e-process'' to designate informally $r$-\glspl{aep} or $r$-\glspl{saep} when neither the sequence $r$ nor the strength of the notion are relevant to the discussion.
We show that this general
definition satisfies two constraints: \begin{enumerate*}[label=(\roman*)]
    \item compatibility with the \gls{savi} framework, in the sense that \gls{savi} is possible asymptotically; and
    \item processes of this type are easily constructed from sufficiently regular asymptotic e-variables, as this is arguably one of the key strengths of (non-asymptotic) e-processes.
\end{enumerate*}
This second requirement is critical for practical relevance.
\par
As motivated in the previous section, the monitoring horizon sequence
$r = (r_m)_{m \in \N}$ typically grows with the approximation index $m\in\N$, capturing the time horizon up to which the process enjoys anytime-valid properties asymptotically.
In particular, our definition implies an asymptotic version of Ville's inequality
(\cref{thm:asymptotic-ville}), which is stated for $r$-\glspl{saep} as:
\begin{equation}
    \label{eq:asymptotic-ville-intro}
        \limsup_{m\to\infty} 
        P\left[\sup_{n\in \{1,\dots,r_m\}} E_{m,n} \geq \frac1\alpha\right]\leq\alpha.
\end{equation}
When $r_m\to\infty$ as $m\to\infty$, the guarantees are truly anytime-valid in the limit; this is certainly the most interesting and motivating case, but we allow general choices. 
The inequality \cref{eq:asymptotic-ville-intro} confirms that for asymptotic e-processes,
Ville's inequality holds for increasingly large time horizons as the approximation quality improves.
We visualize this phenomenon in \cref{fig:intro-experiment}.
\begin{figure}[t]
    \centering
    \includegraphics[width=0.99\linewidth]{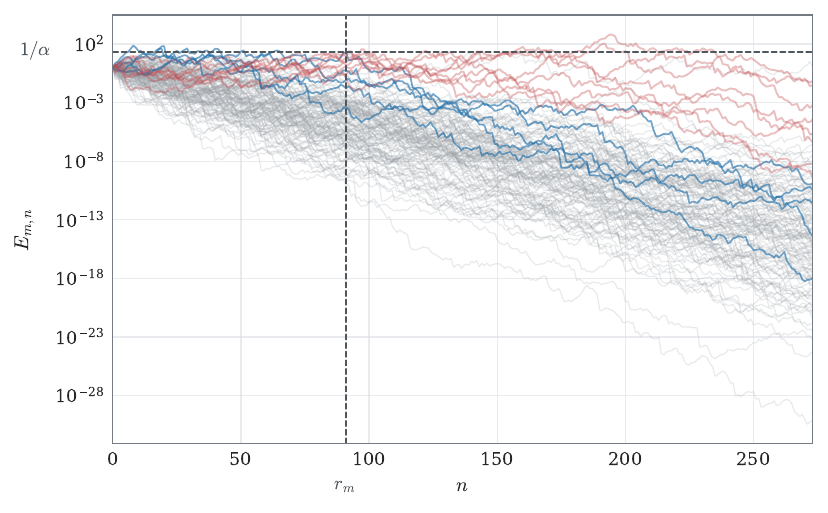}
    \caption{Multiple realizations
    of an asymptotic e-process $(E_{m,n})_{m,n \in \N}$
    over the monitoring time index $n \in \N$
    for a fixed approximation level $m \in \N$.
    The blue trajectories cross the threshold $1/\alpha$
    at some monitoring time $n \leq r_m$, where
    the time horizon $r_m$ depends on
    the quality of approximation determined by $m$.
    Our asymptotic version of Ville's
    inequality in \eqref{eq:asymptotic-ville-intro}
    shows that asymptotically in $m$, the probability
    of this event is at most $\alpha$.
    In contrast to notions that impose $r_m=\infty$ in their version
    of Ville's inequality, the red trajectories
    crossing the threshold for $n > r_m$ 
    cannot be accounted for in general due to 
    a compounding approximation error effect.
    We refer the reader to \cref{sec:numerical-experiments}
    for a detailed construction of the visualized process
    and the experimental setup.
    \label{fig:intro-experiment}}
\end{figure}
\par
A central focus of this paper is deriving general methods for the \emph{construction} of asymptotic e-processes.
We study two main methods.
The first one involves the cumulative product of asymptotic e-variables.
In this context, we investigate the close relationship between $r$-\glspl{saep} and \emph{asymptotic supermartingales}.
The second one is the one leveraged in \cite{chuggPostHocLargeSampleStatistical2026} and, to some extent, \cite{waudby-smithDistributionuniformAnytimevalidSequential2026}.
Namely, it consists of inserting a (modified) estimator of an unknown parameter into the expression defining an e-process in order to obtain asymptotic properties.
We show that this procedure works if the estimator is consistent under a mild monotonicity condition, e.g., when the original e-process is a self-normalized process \citep{penaSelfnormalizedProcessesExponential2004,howardTimeuniformChernoffBounds2020,howardTimeUniformNonparametricNonasymptotic2021}.
We also investigate two other constructions; namely, the calibration of the distribution-uniform anytime-valid p-values of \cite{waudby-smithDistributionuniformAnytimevalidSequential2026}, which we connect to asymptotic p-processes, and a time-mixture construction from asymptotic e-values extending that of \citet[Section 7.9]{ramdasHypothesisTestingEvalues2025}.

\subsection{Summary of contributions}

For the reader's convenience, we give a more detailed
high-level overview of the results in this paper.
In \cref{def:asymptotic e-process}, we introduce 
$r$-\glspl{aep} and $r$-\glspl{saep} (which we
simply refer to as ``asymptotic e-processes'' in 
a general context) as doubly-indexed nonnegative processes $E = (E_{m,n})_{m,n\in\N}$ satisfying an asymptotic e-variable condition for all stopping times up to a horizon  $r_m\in\N\cup\{\infty\}$.
Here, the extended integer sequence $r = (r_m)_{m\in\N}$ is a parameter of the notion quantifying the maximum monitoring time at approximation index $m\in\N$.
Our first main result, \cref{prop:horizon of uniformly asymptotic e-processes}, shows via a diagonal argument that the monitoring horizon $r$ can generally be assumed to diverge to $\infty$ under mild conditions.
\cref{thm:characterizations r-SAEP} provides several equivalent characterizations identifying $r$-\glspl{saep} as processes that are upper bounded for all $m\in\N$ by a process that is a supermartingale up to time $r_m$ and satisfies an asymptotic calibration property.
\cref{clry:characterizations r-AEP} extends one of the characterizations to $r$-\glspl{aep}.
Those results can be seen as direct generalizations of \citet[Lemma 6]{ramdasAdmissibleAnytimevalidSequential2022} to the asymptotic case, and enable a clear interpretation of the sequence $r$.
\cref{thm:characterizations r-SAEP} also formally justifies the interest of introducing the sequence $r$ in the definition, since it shows that $r$-asymptotic e-processes with $r\equiv\infty$ can be renormalized to be e-processes.
Finally, we establish in \cref{thm:asymptotic e-processes converge to e-processes} that, for processes that converge in $L_1$ along the $m$-axis, the notions of $r$-\glspl{aep} and of $r$-\glspl{saep} coincide when $r\to\infty$, and are equivalent to the limit itself being an e-process.
This further grounds the intuition that asymptotic e-processes are approximations of e-processes that are only valid asymptotically, and that no asymptotic properties are lost by allowing finite values of $r_m$ as long as $r$ diverges.
\par
We then investigate the connection between  $r$-\glspl{saep} and supermartingales. We 
introduce in \cref{def:asymptotic supermartingale} \emph{\glspl{asm}} as processes 
converging in $L_1$ to a supermartingale, and in \cref{def:asp} the \emph{\gls{asp}} as a weaker condition pointwise in $n \in \N$ on excess expectations capturing the idea of behaving similarly to a supermartingale asymptotically without requiring $L_1$ convergence. 
\cref{thm:asp implies asymptotic e-process} establishes that calibrated processes with the ASP are $r$-\glspl{saep} for any sequence $r$ satisfying an explicit condition that 
formalizes the intuition of truncating the series \cref{eq:diverging series}.
This provides a simple and practical sufficient criterion for verifying that a process is an $r$-\gls{saep}, and yields an explicit method for constructing the corresponding sequence $r$. 
We leverage this in \cref{subsec:cumulative-product} to show that the cumulative product of sequentially asymptotic e-variables yields an $r$-\gls{saep}.
\cref{subsec:asymptotic e-processes without the asp} shows that not all $r$-\glspl{saep} satisfy the ASP, mirroring the fact that not all e-processes are supermartingales.
We construct such examples in \cref{subsec:time mixture} by adapting the time-mixture construction of \citet[Section 7.9]{ramdasHypothesisTestingEvalues2025}.
\par
\cref{thm:asymptotic-ville} is an asymptotic version of Ville's inequality for asymptotic e-processes. This result ties asymptotic e-processes directly to asymptotic hypothesis testing and establishes $r$ as the time horizon up to which the process can be monitored with asymptotic guarantees.
\par
In \cref{sec:construction} we present various ways to construct asymptotic e-processes, supporting our claim of practical constructability.
We focus first on the cumulative product already introduced, as well as an analogue of the time-mixture construction for e-processes \cite[Section 7.9]{ramdasHypothesisTestingEvalues2025}, thereby exhibiting asymptotic e-processes that do not have the \gls{asp}.
The cumulative product is illustrated with simple numerical experiments in \cref{sec:numerical-experiments}.
We investigate the domination method; a process upper-bounded by an $r$-\gls{aep} with asymptotically full probability is itself an $r$-\gls{aep} (\cref{thm:domination}).
This enables constructing $r$-\glspl{aep} by inserting (over-)estimators into expressions of true e-processes under a monotonicity condition, generalizing the construction of \cite{chuggPostHocLargeSampleStatistical2026} in \cref{clry:domination via monotonicity,thm:r-aep from burn in of single data stream}.
We also examine the calibration \cite[Section 2.3]{ramdasHypothesisTestingEvalues2025} of the anytime p-values introduced in \cite{waudby-smithDistributionuniformAnytimevalidSequential2026}.
More systematically, we introduce asymptotic p-processes in \cref{def:asymptotic p-process}, and identify anytime p-values as uniformly $\infty$-asymptotic p-processes in \cref{thm:anytime p-values are infty-APPs}.
The general result is that calibrating asymptotic p-processes yields asymptotic e-processes (\cref{thm:calibration}), mirroring the results available in the non-asymptotic case.
\par
We conclude this section by clarifying the relations of the present work to the previous studies of \cite{ignatiadisAsymptoticCompoundEvalues2024}, \cite{waudby-smithDistributionuniformAnytimevalidSequential2026}, and \cite{chuggPostHocLargeSampleStatistical2026}.
All three studies are central to our contributions.
In particular, they all introduce indexing by the quality of approximations, leading to our central object being a bi-indexed process, and all provide a clear framework for handling distribution-uniform guarantees.
This uniformity is essential for ensuring that the sequence $r$ does not depend on the choice of distribution $P$ in the null.
Furthermore, \cite{ignatiadisAsymptoticCompoundEvalues2024} provide the foundational definitions upon which we build various notions of asymptotic e-processes, just as e-processes are defined in terms of e-values.
Nevertheless, an idea central to our work and not explored in those prior studies is that time uniformity should not be imposed before the asymptotic regime is attained.
It is precisely this insight that enables constructing asymptotic e-processes under relatively weak assumptions while retaining strong guarantees in the limit.

\subsection{Structure of this paper}
In \cref{sec:preliminaries}, we introduce
our notation and the necessary background of
e-variables and e-processes.
We present the formal definition of an asymptotic e-process
in \cref{sec:asymptotic-e-processes},
as well as the main theoretical results pertaining to their characterization and convergence.
\cref{sec:asymptotic-supermartingales} connects them to the notion of asymptotic supermartingales, culminating in an explicit sufficient condition for $r$-asymptotic e-processes.
\cref{sec:asymp-ville} contains 
an asymptotic version of Ville's inequality for 
asymptotic e-processes.
In \cref{sec:construction},
we construct specific examples of asymptotic e-processes
from asymptotic e-variables, formalize the domination method and its relation to plugging-in estimators, and study the calibration of anytime p-values.
\Cref{sec:numerical-experiments} showcases on simple examples how the behavior described by the asymptotic Ville's inequality can be observed empirically.
All proofs can be found in the supplementary materials.
\section{Preliminaries}
\label{sec:preliminaries}

We now introduce the necessary mathematical background
and the concepts of e-variables and e-processes.

\paragraph*{Notation\preprintdot}
Throughout this article, $\R$ is the set of real numbers, and $\Rnn$ and $\Rp$ are respectively the sets of nonnegative and of positive real numbers.
The set of nonnegative integers is $\N$, and $\Np$ is the set of positive integers.
When considered as measurable spaces, they are understood as equipped with their Borel $\sigma$-algebras, and so is the extended real line $\R\cup\{-\infty,\infty\}$ and the corresponding extensions of $\Rnn$, $\Rp$, and $\N$.
A \emph{sequence of extended integers} is a sequence indexed by $\N$ with values in $\N\cup\{\infty\}$.
Finally, for $n\in\N\cup\{\infty\}$, the notation $\integers n$ denotes the subset of $\N$ composed of elements that are at most $n$, with the convention that $\integers\infty = \N$.
We also define the standard notation $m\pwMax n = \max(m,n)$ and $m\pwMin n = \min(m,n)$, for all $m,n\in\N\cup\{\infty\}$.
\par
We introduce $(\Omega,\sigAlg)$ a measurable space on which we define all random elements, and $\probaMeasures(\Omega)$ is the set of probability measures on $(\Omega,\sigAlg)$.
We also introduce a subset $\subsetProbaMeasures\subset\probaMeasures(\Omega)$.
Real random variables are considered in the extended sense by default: a real (resp., nonnegative, positive) random variable takes values in $\R\cup\{-\infty,\infty\}$ (resp., $\Rnn\cup\{\infty\}$, $\Rp\cup\{\infty\}$).
When random variables are assumed to only take finite values, they are explicitly specified to be \emph{finite}, or $P$-\gls{as} finite if they are finite on a set of full measure for a measure $P\in\probaMeasures(\Omega)$.
For any real random variable $X$, we define its positive part as $X^+ := \max\{0,X\}$ and its negative part as $X^- = (-X)^+$.
If it is $P$-integrable for some $P\in\probaMeasures(\Omega)$, we introduce its expectation on the probability space $(\Omega,\sigAlg,P)$ as \begin{equation*}
    \E_P[X] := \int_\Omega X(\omega) \dd P(\omega).
\end{equation*}
In particular, integrability implies that $X$ takes values in $\R$, $P$-\gls{as} (i.e., is \gls{as} finite).
The Banach space of (equivalence classes of) $P$-integrable real random variables is denoted as $L_1(\Omega,\sigAlg,P)$, and its norm is $\norm{\cdot}_P$.
If $X$ is $P$-integrable and $\filtration\subset\sigAlg$ is a sub-$\sigma$-algebra, $\E_P[X\mid\filtration]$ denotes any (measurable) choice of the conditional expectation of $X$ given $\filtration$.
The choice of conditional expectations does not play a role in this work.
If $X\geq a$ for some constant $a\in\R$, $P$-\gls{as}, we also define $\E_P[X]$ if $X$ is not $P$-integrable as $\E_P[X] := \infty$.
Throughout, a real process (resp. nonnegative process, positive process) is a stochastic process of real (resp., nonnegative, positive) variables indexed by $\N$, $\N^2$, or a subset of one of those sets.
The exact subset is always clear from context.
Furthermore, if $X = (X_n)_{n\in\N}$ is a nonnegative process (whose trajectories do not necessarily converge, even only \gls{as} for some measure), we define \begin{equation}\label{eq:convention evaluation at infinity}
    \forall \omega\in\Omega,~X_\infty(\omega) := \limsup_{n\to\infty} X_n(\omega).
\end{equation}
It follows that $X_\infty$ is a nonnegative random variable.
If the trajectory $X(\omega)$ converges for some $\omega\in\Omega$, this notation is consistent since then $X_\infty(\omega) = \lim_{n\to\infty} X_n(\omega)$.
Finally, recall that a filtration $\filtration$ on $(\Omega,\sigAlg)$ is a nondecreasing sequence of sub-$\sigma$-algebras of $\sigAlg$.
Given a filtration $\filtration = (\filtration_n)_{n\in\N}$, we consistently define the notation $\filtration_\infty$ to denote the limit $\sigma$-algebra \begin{equation*}
    \filtration_\infty = \sigma\left(\bigcup_{n\in\N}\filtration_n\right).
\end{equation*}
Since our results involve bi-indexed processes, we use the following notation: if $a = (a_{m,n})_{m,n\in I}$ is a bi-indexed sequence indexed by a set $I\subset\N^2$, then $a_{\bullet,n}$ and $a_{m,\bullet}$ denote $(a_{j,n})_{j\in\{k\mid(k,n)\in I\}}$ and $(a_{m,i})_{i\in\{k\mid(m,k)\in I\}}$, respectively and for all $m,n\in\N$.
This enables the following definitions.
\begin{definition}[Filtration sequences and arrays]
    A \emph{filtration sequence} (along the first index, on $(\Omega,\sigAlg)$) is a family $\filtration = (\filtration_{m,n})_{m,n\in\N}$ of sub-$\sigma$-algebras of $\sigAlg$ such that $\filtration_{m,\bullet}$ is a filtration for all $m\in\N$.
    In this case, we define $ \filtration_{m,\infty} = \sigma\left(\bigcup_{i\vphantom{j}\in\N}\filtration_{m,i}\right)$ for all $m\in\N$.
    Furthermore, $\filtration$ is a \emph{filtration array} (on $(\Omega,\sigAlg)$) if both $\filtration_{m,\bullet}$ and $\filtration_{\bullet,n}$ are filtrations, for all $m,n\in\N$.
    We then also define $\filtration_{\infty,n} = \sigma\left(\bigcup_{j\in\N}\filtration_{j,n}\right)$ for all $n\in\N$.
    A process $E = (E_{m,n})_{m,n\in\N}$ is said to be \emph{adapted} to a filtration sequence or array $\filtration$ if, for all $m,n\in\N$, $E_{m,n}$ is $\filtration_{m,n}$-measurable; that is, for all $m\in\N$, $E_{m,\bullet}$ is adapted to $\filtration_{m,\bullet}$.
    Finally, the filtration sequence and array generated by a process $E = (E_{m,n})_{m,n\in\N}$ are respectively defined as \begin{equation*}
        \left(\sigma\left(\left\{E_{m,i}\mid i\in\integers n\right\}\right)\right)_{m,n\in\N},\quad\text{and}\quad\left(\sigma\left(\left\{E_{j,i}\mid j\in\integers m\land i\in\integers n\right\}\right)\right)_{m,n\in\N}.
    \end{equation*}
\end{definition}
One can immediately verify that $\filtration$ is a filtration array if, and only if, $\filtration_{m,n}\subset\filtration_{m+1,n}\cap\filtration_{m,n+1}$, and that, in this case, $\filtration_{\infty,\bullet}$ and $\filtration_{\bullet,\infty}$ are also filtrations on $(\Omega,\sigAlg)$.
All filtrations, filtration sequences, and filtration arrays we consider in this work are on $(\Omega,\sigAlg)$; hence, we omit specifying it from now on.

\paragraph*{E-variables and related notions\preprintdot}
We now introduce e-variables, e-processes
and test supermartingales.
We refer to \cite{ramdasHypothesisTestingEvalues2025} and \cite{ignatiadisAsymptoticCompoundEvalues2024} for more details.
\begin{definition}[e-variable]
\label{def:evalue}
An \emph{e-variable} (for $\subsetProbaMeasures$) is a nonnegative random variable $e$ such that $e$ is $P$-integrable and $\E_P[e]\leq 1$, for all $P\in\subsetProbaMeasures$. 
\end{definition}
\par
We briefly revisit the definition of stopping times and e-processes.
It is known that e-processes are equivalently defined irrespective of whether or not one allows stopping times to take the value $\infty$ \citep[Lemma 6]{ramdasAdmissibleAnytimevalidSequential2022}.
We will see that the same holds for asymptotic e-processes, and thus distinguish finite and possibly infinite stopping times.
\begin{definition}[Stopping time]\label{def:stopping time}
    Let $\filtration = (\filtration_n)_{n\in\N}$ be a filtration.
    An $\filtration$-stopping time is a measurable map $\tau:\Omega\to\N \cup \{\infty\}$ 
    such that, for all $n\in\N$, 
    \begin{equation*}
        \{\tau \leq n\} := \{\omega\in\Omega\mid \tau(\omega) \leq n\} \in\filtration_n,
    \end{equation*}
    and an $\filtration$-stopping time is said to be \emph{finite} if $\{\tau=\infty\} = \emptyset$.
    For all $\rho\in\N\cup\{\infty\}$, we denote as $\rBoundedStoppingTimes(\rho,\filtration,\subsetProbaMeasures)$ the set of 
    $\filtration$-stopping times $\tau$ such that \begin{equation*}
        \forall P\in\subsetProbaMeasures,~P[\tau\leq \rho] = 1,
    \end{equation*}
    with the convention that $\rBoundedStoppingTimes(\infty,\filtration,\subsetProbaMeasures)$ is the set of all $\filtration$-stopping times (possibly taking the value $\infty$).
    We also introduce $\rBoundedStoppingTimes_\frm(\rho,\filtration,\subsetProbaMeasures)$ the subset of $\rBoundedStoppingTimes(\rho,\filtration,\subsetProbaMeasures)$ consisting of finite stopping times.
\end{definition}
It follows trivially that $\rBoundedStoppingTimes_\frm(\rho,\filtration,\subsetProbaMeasures)$ and $\rBoundedStoppingTimes(\rho,\filtration,\subsetProbaMeasures)$ only differ when $\rho=\infty$.
We often omit the filtration $\filtration$ in the terminology of stopping times, as it is always clear from context.
Finally, we emphasize that our results are unchanged when one defines $\rBoundedStoppingTimes_\frm(\rho,\filtration,\subsetProbaMeasures)$ as the subset of $\rBoundedStoppingTimes(\rho,\filtration,\subsetProbaMeasures)$ consisting of stopping times that are only $P$-\gls{as} finite for all $P\in\subsetProbaMeasures$ (as opposed to finite pointwise in $\Omega$); we choose to state all results for finite stopping times only, the $P$-\gls{as} finite case being entirely analogous.
\begin{definition}[e-process]\label{def:e-process}
An \emph{e-process} (for $\subsetProbaMeasures$,
with respect to a filtration
$\filtration$)
is a nonnegative process $E = (E_n)_{n \in \N}$
adapted to $\filtration$ such that, for any $\filtration$-stopping time $\tau$, $E_{\tau}$ is an e-variable for $\subsetProbaMeasures$; that is, $\forall P\in\subsetProbaMeasures, \E_P[E_\tau]\leq 1$.
 
\end{definition}
As mentioned, the definition is unaffected by whether the stopping times are enforced to be finite or not \cite[Lemma 6]{ramdasAdmissibleAnytimevalidSequential2022}.
\begin{definition}[Supermartingale]
A \emph{supermartingale} (for $\subsetProbaMeasures$, with respect to a filtration
$\filtration$)
is a process $S = (S_n)_{n \in \N}$ adapted to $\filtration$ such that, for all $P\in\subsetProbaMeasures$ and $n\in\N$, $S_n$ is $P$-integrable and \begin{equation*}
    \E_P[S_{n+1}\mid\filtration_n]\leq S_n,\quad P\text{-\gls{as}}
\end{equation*}
It is said to be a \emph{test supermartingale} (for $\subsetProbaMeasures$, with respect to a filtration $\filtration$) if, additionally, it is nonnegative and satisfies $\E_P[S_0]\leq 1$ for all $P\in\subsetProbaMeasures$.
\end{definition}
A test supermartingale $S$ is always an e-process.
This follows from the optional sampling theorem for nonnegative supermartingales \citep[see e.g.][Theorem 10.11]{klenke2020probability}.
Finally, we recall the notion of asymptotic e-variables as introduced in \cite{ignatiadisAsymptoticCompoundEvalues2024}.
An asymptotic e-variable is a nonnegative process whose index $m\in\N$ represents the quality of approximation of a true e-variable, in some sense.
\begin{definition}[Asymptotic e-variable, \citealp{ignatiadisAsymptoticCompoundEvalues2024}]
    A nonnegative process $e = (e_m)_{m\in\N}$ is said to be a \emph{uniformly asymptotic e-variable for $\subsetProbaMeasures$} if \begin{equation}\label{eq:uniformly AEV}
        \forall t\in\Rnn,\quad\limsup_{m\to\infty} \sup_{P\in\subsetProbaMeasures}\E_P[e_m\pwMin t]\leq 1.
    \end{equation}
    Furthermore, it is said to be a \emph{uniformly strongly asymptotic e-variable for $\subsetProbaMeasures$} if \begin{equation}
        \limsup_{m\to\infty} \sup_{P\in\subsetProbaMeasures}\E_P[e_m]\leq 1.
    \end{equation}
\end{definition}
We emphasize that, in the original work of \cite{ignatiadisAsymptoticCompoundEvalues2024}, they are called \emph{sequences} of uniformly (strongly) asymptotic e-variables.
We instead favor the terminology ``asymptotic e-variable'' for the full sequence.
\par
We also emphasize that \cref{eq:uniformly AEV} is equivalent to the existence of a sequence $(t_m)_{m\in\N}\subset\Rnn$ such that \begin{equation}
    \lim_{m\to\infty}t_m = \infty,\quad\text{and}\quad \limsup_{m\to\infty} \sup_{P\in\subsetProbaMeasures}\E_P[e_m\pwMin t_m]\leq 1.
\end{equation}
This characterization follows from \cref{clry:rate of uniform boundedness 2}, and proves useful in the statement of later results.

\paragraph*{Convergence of bi-indexed processes in $L_1$\preprintdot}
We briefly address the convergence of a bi-indexed real process $E = (E_{m,n})_{m,n\in\N}$ to another real process $F = (F_n)_{n\in\N}$.
The convergence we require is pointwise in $n\in\N$, and should occur in $L_1$, uniformly over $P\in\subsetProbaMeasures$.
We formalize it in the following definitions to facilitate the statement of later results.
To that end, we remind the reader that a real process $(X_\lambda)_{\lambda\in\Lambda}$ indexed by an arbitrary set $\Lambda$ is said to be $P$-integrable if $X_\lambda$ is $P$-integrable for all $\lambda\in\Lambda$, where $P\in\probaMeasures(\Omega)$, and it is $\subsetProbaMeasures$-integrable if it is $P$-integrable for all $P\in\subsetProbaMeasures$.
\begin{definition}\label{def:l1 convergence of bi-indexed process}
    Let $E = (E_{m,n})_{m,n\in\N}$ and $F = (F_n)_{n\in\N}$ be nonnegative processes.
    We say that \emph{$E$ converges to $F$ in $L_1$ uniformly in $\subsetProbaMeasures$ (along the first index)} if the following conditions hold:
    \begin{enumerate}[label=(\roman*)]
        \item \label{item:l1 convergence integrability} $E$ and $F$ are $\subsetProbaMeasures$-integrable;
        \item for all $n\in\N$, $E_{m,n}\to F_n$ as $m\to\infty$, where convergence is in $L_1(\Omega,\sigAlg,P)$ uniformly in $P\in\subsetProbaMeasures$; that is, \begin{equation*}
            \lim_{m\to\infty}\sup_{P\in\subsetProbaMeasures}\E_P\left[\lvert E_{m,n} - F_n\rvert\right] = 0.
        \end{equation*}
    \end{enumerate}
\end{definition}
Finally, some of our results prompt us to take sums and differences between processes taking values in the extended real line.
This is problematic when they take the value $\pm\infty$, as this leads to expressions of the form $\infty-\infty$, which are undefined.
We address this preemptively by saying that such sums are only taken in the context of processes that converge in $L_1$ in the sense of \cref{def:l1 convergence of bi-indexed process}.
In particular, the variables involved are integrable, and thus \gls{as} finite, which enables understanding these sums \gls{as} with respect to the measure at hand.
Such expressions should then always be understood in this \gls{as} sense.
\par
We also point out that the integrability condition in \cref{def:l1 convergence of bi-indexed process}\labelcref{item:l1 convergence integrability} may be imposed only for all $m\geq m_0$, with $m_0\in\N$ independent of $P\in\subsetProbaMeasures$, without affecting our results involving such convergence in $L_1$.
Nevertheless, this comes with a notational burden since the difference between processes is then only unambiguously defined for $m\geq m_0$ as per the remark right above.
For this reason, we stick to the simpler condition in \cref{def:l1 convergence of bi-indexed process}\labelcref{item:l1 convergence integrability}.

\section{Asymptotic e-processes}
\label{sec:asymptotic-e-processes}
We now formally define two notions of asymptotic
e-processes.
To that end, we overload the symbols $\rBoundedStoppingTimes$ and $\rBoundedStoppingTimes_\frm$ introduced in \cref{def:stopping time} to accept as their first two arguments sequences $r = (r_m)_{m\in\N}\subset\N\cup\{\infty\}$ and filtration sequences $\filtration = (\filtration_{m,n})_{m,n\in\N}$, and define $\rBoundedStoppingTimes(r,\filtration,\subsetProbaMeasures)$ (resp.\ $\rBoundedStoppingTimes_\frm(r,\filtration,\subsetProbaMeasures)$) as the set of sequences $\tau = (\tau_m)_{m\in\N}$ such that, for all $m\in\N$, $\tau_m \in\rBoundedStoppingTimes(r_m,\filtration_{m,\bullet},\subsetProbaMeasures)$ (resp.\ $\tau_m \in\rBoundedStoppingTimes_\frm(r_m,\filtration_{m,\bullet},\subsetProbaMeasures)$).
As announced, we first define asymptotic e-processes using $\rBoundedStoppingTimes$.
\begin{definition}[Uniformly $r$-asymptotic e-process]
    \label{def:asymptotic e-process}
    Let $E = (E_{m,n})_{m,n\in\N}$ be a nonnegative process adapted to a filtration sequence $\filtration$ and $r$ be a sequence of extended integers.
    We say that $E$ is a \begin{enumerate}[label=(\roman*)]
        \item \emph{uniformly $r$-asymptotic e-process ($r$-\glsps{aep}\glsunset{aep})} (for $\subsetProbaMeasures$ and for $\filtration$) if for all $\tau=(\tau_m)_{m\in\N}\in\rBoundedStoppingTimes(r,\filtration,\subsetProbaMeasures)$, the process $(E_{m,\tau_m})_{m\in\N}$ is a uniformly asymptotic e-variable for $\subsetProbaMeasures$; that is, \begin{equation*}
        \forall t\in\Rnn,~\forall \tau=(\tau_m)_{m\in\N}\in\rBoundedStoppingTimes(r,\filtration,\subsetProbaMeasures), \quad\limsup_{m\to\infty}\sup_{P\in\subsetProbaMeasures}\E_P[E_{m,\tau_m}\pwMin t]\leq 1,
    \end{equation*}
        \item \emph{uniformly strongly $r$-asymptotic e-process ($r$-\glsps{saep}\glsunset{saep})} (for $\subsetProbaMeasures$ and for $\filtration$) if for all $\tau=(\tau_m)_{m\in\N}\in\rBoundedStoppingTimes(r,\filtration,\subsetProbaMeasures)$, the process $(E_{m,\tau_m})_{m\in\N}$ is a uniformly strongly asymptotic e-variable for $\subsetProbaMeasures$; that is, \begin{equation*}
        \forall \tau=(\tau_m)_{m\in\N}\in\rBoundedStoppingTimes(r,\filtration,\subsetProbaMeasures), \quad\limsup_{m\to\infty}\sup_{P\in\subsetProbaMeasures}\E_P[E_{m,\tau_m}]\leq 1.
    \end{equation*}
    \end{enumerate}
\end{definition}
This definition warrants two remarks.

\begin{remark}[Terminology]
The qualifier ``uniformly'' above refers to uniformity in $P\in\subsetProbaMeasures$.
Previous works have made clear the importance of such uniformity in the context of hypothesis testing; see mainly \citet{waudby-smithDistributionuniformAnytimevalidSequential2026}.
For that reason, we always consider such uniformity in this paper, and thus omit it from the acronyms $r$-\gls{aep} and $r$-\gls{saep} as well as in the discussion.
Furthermore, we generally use the expression ``asymptotic e-process'' when referring to either an $r$-\gls{aep} or an $r$-\gls{saep} when the specific sequence $r$ and which notion exactly is considered are not relevant to the argument at hand.
In the interest of rigor, however, formal statements always use the expressions ``uniformly $r$-\gls{aep}'' or ``uniformly $r$-\gls{saep}''.
\end{remark}

\begin{remark}[Relationship between $r$-\glspl{aep} and $r$-\glspl{saep}]
    An $r$-\gls{saep} is always an $r$-\gls{aep}.
    However, the $r$-\gls{saep} property 
    is more restrictive than the $r$-\gls{aep} property.
    To see this, assume that we have an $r$-\gls{saep} $E$, and introduce a sequence $(A_m)_{m\in\N}\subset\sigAlg$ such that $\inf_{P\in\subsetProbaMeasures} P[A_m] \to 1$ as $m\to\infty$ but $P[A_m] < 1$ for all $P\in\subsetProbaMeasures$ and $m\in\N$, assuming existence.
    One straightforwardly verifies that the process $\bar E$ defined for all $m,n\in\N$ as \begin{equation}
    \bar E_{m,n} = E_{m,n}\cdot\indicator_{A_m} + \infty\cdot\indicator_{A_m^\complement}
    \end{equation}
    is an $r$-\gls{aep}, but is not an $r$-\gls{saep}; in particular, it is never integrable for any value of $m$ and $n$ in $\N$.
    This construction of $r$-\glspl{aep} is generalized in \cref{subsec:event partitioning}.
\end{remark}
\par
We now interpret \cref{def:asymptotic e-process} intuitively.
Comparing the condition for $r$-\glspl{saep} to that for e-processes (\cref{def:e-process}), it becomes immediately clear that the former is a relaxation of the latter imposing the condition only asymptotically in the approximation index $m$.
A key difference, however, is the introduction of the sequence $r$; $r$-\glspl{saep} only control the behavior of stopping times bounded by $r_m$ at index $m\in\N$.
Intuitively, for a given approximation index $m\in\N$, the statistician is ``allowed'' to monitor the process $E_{m,\bullet}$ arbitrarily, but only up to time $r_m$.
It is thus desirable that $r_m\to\infty$ as $m\to\infty$, as quickly as possible.
As we show on concrete examples in \cref{subsec:cumulative-product}, the divergence speed of $r$ is directly linked to how quickly the process approximates a true e-process.
Specifically, when constructing an $r$-\gls{saep} from asymptotic e-variables, the faster these approximate true e-variables, the faster the sequence $r$ is allowed to increase.
The difference between $r$-\glspl{saep} and $r$-\glspl{aep} is that the latter require an additional (arbitrary) truncation to satisfy the asymptotic property.
This enables handling for instance processes that would be $r$-\glspl{saep} if it were not for a sequence of events with uniformly vanishing probability on which they are ill-behaved (e.g., grow unbounded, preventing integrability), as we illustrate in \cref{subsec:event partitioning}.
Intuitively, the relation between $r$-\glspl{aep} and $r$-\glspl{saep} mirrors that between uniformly asymptotic e-variables and their strong version; we refer to \citet{ignatiadisAsymptoticCompoundEvalues2024} for more details on this topic.
\par
In order to discuss the relation between \cref{def:asymptotic e-process} and the recent studies of \cite{waudby-smithDistributionuniformAnytimevalidSequential2026} and \citet{chuggPostHocLargeSampleStatistical2026}, the following characterizations are helpful.
\begin{restatable}{proposition}{characterizationUniformStoppingTime}\label{prop:characterization uniform in stopping time}
    Let $E = (E_{m,n})_{m,n\in\N}$ be a nonnegative process adapted to a filtration sequence $\filtration$ and $r = (r_m)_{m\in\N}$ be a sequence of extended integers.
    Then, $E$ is a uniformly $r$-\gls{saep} for $\subsetProbaMeasures$ and $\filtration$ if, and only if, \begin{equation*}
        \limsup_{m\to\infty}\sup_{\tau\in\rBoundedStoppingTimes(r_m,\filtration_{m,\bullet},\subsetProbaMeasures)}\sup_{P\in\subsetProbaMeasures}\E_P[E_{m,\tau}]\leq 1.
    \end{equation*}
    Furthermore, $E$ is a uniformly $r$-\gls{aep} for $\subsetProbaMeasures$ and $\filtration$ if, and only if, there exists a sequence $(t_m)_{m\in\N}\subset\Rnn$ such that $t_m\to\infty$ as $m\to\infty$ and \begin{equation}\label{eq:characterization r-AEP}
        \limsup_{m\to\infty}\sup_{\tau\in\rBoundedStoppingTimes(r_m,\filtration_{m,\bullet},\subsetProbaMeasures)}\sup_{P\in\subsetProbaMeasures}\E_P[E_{m,\tau}\pwMin t_m]\leq 1,
    \end{equation}
    that is, if, and only if, $(E_{m,n}\pwMin t_m)_{m,n\in\N}$ is a uniformly $r$-\gls{saep} for $\subsetProbaMeasures$ and $\filtration$.
    Such a sequence $t$ is called a \emph{truncation sequence} of $E$.
\end{restatable}
We are now equipped to discuss how \cref{def:asymptotic e-process} relates to the recent works of \citet{waudby-smithDistributionuniformAnytimevalidSequential2026} and \citet{chuggPostHocLargeSampleStatistical2026}, which introduce related concepts.
Specifically, ``asymptotic e-processes'' as defined in \citet[Definition 4.2]{chuggPostHocLargeSampleStatistical2026} correspond to what \cref{def:asymptotic e-process} calls $r$-\glspl{saep} with $r_m=\infty$ for all $m\in\N$, up to a reindexing convention on the second index and the fact that we allow a dependency of the stopping time on $m$, which is necessary in our more general setting where the filtration itself depends on $m$.
Furthermore, the $r$-\gls{saep} that \citet{chuggPostHocLargeSampleStatistical2026} construct in Proposition 4.4 is based on an $r$-\gls{aep} thresholded by an appropriate sequence $(t_m)_{m\in\N}$.
\Cref{prop:characterization uniform in stopping time} supports the generality of this construction.
Similarly, \citet{waudby-smithDistributionuniformAnytimevalidSequential2026} introduce $\subsetProbaMeasures$-uniform anytime p-values; we see in \cref{subsec:anytime p-values} that they relate to $r$-\glspl{aep} with $r_m=\infty$ for all $m\in\N$ via p-to-e calibration.
Summarizing, \cref{def:asymptotic e-process} is a strict generalization of the concepts already introduced in the literature, in two aspects: it allows general sequences $r$ instead of imposing $r_m=\infty$ for all $m\in\N$ in the definition, and it introduces the distinction between the different strengths of the asymptotic regime based on the analysis of \citet{ignatiadisAsymptoticCompoundEvalues2024}.
The contribution of this work then resides in analyzing the properties and construction methods of processes satisfying \cref{def:asymptotic e-process}.
In particular, one of our main findings is that allowing $r_m<\infty$ for all $m\in\N$ does not result in weaker \emph{asymptotic} guarantees as long as $r_m\to\infty$ as $m\to\infty$.
In fact, \cref{sec:limiting behavior} shows that this is the necessary and sufficient property capturing ``convergence to an e-process'' in the $m$-index.
\par
This discussion highlights the special role played by $r$-\glspl{aep} and $r$-\glspl{saep} with $r_m=\infty$ for all $m\in\N$.
For that reason, by abuse of notation, we denote them as $\infty$-\glspl{aep} and $\infty$-\glspl{saep}, where it is understood that the symbol $\infty$ stands in lieu of the sequence $r$ constant equal to $\infty$.
Furthermore, given the announced relevance of the case $r_m\to\infty$ as $m\to\infty$ for asymptotic properties, we provide the following simple characterization.
\begin{restatable}{proposition}{horizon}
    \label{prop:horizon of uniformly asymptotic e-processes}
    Let $E = (E_{m,n})_{m,n\in\N}$ be a nonnegative process adapted to a filtration sequence $\filtration = (\filtration_{m,n})_{m,n\in\N}$.
    The following statements are equivalent:
    \begin{enumerate}[label=(\roman*)]
        \item \label{item:horizon of uniformly asymptotic e-processes:bounded horizon} For any bounded integer sequence $r$, $E$ is a uniformly $r$-\gls{aep} (resp., $r$-\gls{saep}) for $\subsetProbaMeasures$ and $\filtration$.
        \item \label{item:horizon of uniformly asymptotic e-processes:diverging horizon} There exists an integer sequence $r = (r_m)_{m\in\N}$ such that $r_m\to\infty$ as $m\to\infty$ and $E$ is a uniformly $r$-\gls{aep} (resp., $r$-\gls{saep}) for $\subsetProbaMeasures$ and $\filtration$.
    \end{enumerate}
\end{restatable}

\subsection{Characterizations with finite stopping times or nonnegative partial supermartingales}

Two important theoretical results in the theory of e-processes are (i) that it suffices to verify the defining inequality over finite stopping times (this is equivalent to allowing stopping times that may take the value $\infty$), and (ii) that e-processes are precisely the nonnegative processes that are dominated by a suitably normalized nonnegative supermartingale \citep[Lemma 6]{ramdasAdmissibleAnytimevalidSequential2022}.
In this section, we show in what sense these useful characterizations extend to the asymptotic setting, beginning with the case of $r$-\glspl{saep}.
The proof of this result is similar to that of the nonasymptotic case \cite[Lemma 6]{ramdasAdmissibleAnytimevalidSequential2022}.
\begin{restatable}{theorem}{characterizationSAEP}\label{thm:characterizations r-SAEP}
    Let $\filtration$ be a filtration sequence, $E = (E_{m,n})_{m,n\in\N}$ be a nonnegative process adapted to $\filtration$, and $r$ be an extended integer sequence.
    The following statements are equivalent:\begin{enumerate}[label=(\roman*)]
        \item \label{item:characterizations r-SAEP:i} $E$ is a uniformly $r$-\gls{saep} for $\subsetProbaMeasures$ and $\filtration$;
        \item \label{item:characterizations r-SAEP:ii} For any $\tau = (\tau_m)_{m\in\N}\in\rBoundedStoppingTimes_\frm(r,\filtration,\subsetProbaMeasures)$, the process $(E_{m,\tau_m})_{m\in\N}$ is a uniformly strongly asymptotic e-variable;
        \item \label{item:characterizations r-SAEP:iii} There exists a family $L=(L_{m,n}^P)_{(m,n,P)\in\N\times\N\times\subsetProbaMeasures}$ of nonnegative variables with the following properties:
        \begin{enumerate}
            \item \label{item:characterizations r-SAEP:npsms upper bound}for all $m,n\in\N$ and $P\in\subsetProbaMeasures$, $L_{m,n}^P\geq E_{m,n}$, $P$-\gls{as},
            \item \label{item:characterizations r-SAEP:npsms integrability} there exists $m_0\in\N$ such that $L_{m,n}^P$ is $P$-integrable for all $m\geq m_0$, $n\in\integers{r_{m}}$, and $P\in\subsetProbaMeasures$,
            \item \label{item:characterizations r-SAEP:npsms supermartingale} for all $m\geq m_0$, $L_{m,\bullet}^P$ is a supermartingale until index $r_m-1$; that is, \begin{equation*}
                \forall n\in\integers{r_m-1},\quad\E_P[L_{m,n+1}^P\mid\filtration_{m,n}] \leq L_{m,n}^P,\quad P\text{-\gls{as}}
            \end{equation*}
            \item \label{item:characterizations r-SAEP:npsms calibration} $\limsup_{m\to\infty}\sup_{P\in\subsetProbaMeasures}\E_{P}[L_{m,0}^P]\leq 1$.
        \end{enumerate}
    \end{enumerate}
\end{restatable}
\begin{remark}
    In the nonasymptotic case, processes similar to $L$ as in \cref{thm:characterizations r-SAEP}\labelcref{item:characterizations r-SAEP:iii} are introduced in \cite{ramdasAdmissibleAnytimevalidSequential2022} as \emph{$Q$-nonnegative supermartingales}, with $Q\in\subsetProbaMeasures$, though with stronger properties that relate them to e-processes. Processes satisfying \cref{thm:characterizations r-SAEP}\labelcref{item:characterizations r-SAEP:iii} are their counterparts for asymptotic e-processes, and could thus be named ``$(Q,r)$-nonnegative partial supermartingale sequences'', with $Q\in\subsetProbaMeasures$.
\end{remark}
It follows immediately from the definition that the characterization of $r$-\glspl{saep} by finite stopping times also holds for $r$-\glspl{aep} by applying \cref{thm:characterizations r-SAEP} to the $r$-\gls{saep} $E\pwMin t := (E_{m,n}\pwMin t)_{m,n\in\N}$ for all fixed $t\in\Rnn$.
This results in the associated family $(L_{m,n}^P)_{(m,n,P)\in\N\times\N\times\subsetProbaMeasures}$ depending on the value of $t\in\Rnn$, making it less useful than in the case of $r$-\glspl{saep}.
We summarize this in the following result.
\begin{restatable}{corollary}{characterizationAEP}\label{clry:characterizations r-AEP}
    Let $\filtration$ be a filtration sequence, $E = (E_{m,n})_{m,n\in\N}$ be a nonnegative process adapted to $\filtration$, and $r$ be an extended integer sequence.
    The following statements are equivalent:\begin{enumerate}[label=(\roman*)]
        \item \label{item:characterizations r-AEP:i} $E$ is a uniformly $r$-\gls{aep} for $\subsetProbaMeasures$ and $\filtration$;
        \item \label{item:characterizations r-AEP:ii} For any $\tau = (\tau_m)_{m\in\N}\in\rBoundedStoppingTimes_\frm(r,\filtration,\subsetProbaMeasures)$, the process $(E_{m,\tau_m})_{m\in\N}$ is a uniformly asymptotic e-variable.
    \end{enumerate}
\end{restatable}
\begin{remark}[$\infty$-\glspl{saep}]\label{rmk:infinity-asymptotic e-processes}
An interesting consequence of \cref{thm:characterizations r-SAEP} is that $\infty$-\glspl{saep} are precisely those processes that are upper-bounded by nonnegative supermartingales that satisfy the asymptotic calibration condition \cref{thm:characterizations r-SAEP}\labelcref{item:characterizations r-SAEP:npsms calibration}.
In other words, leveraging \citet[Lemma 6]{ramdasAdmissibleAnytimevalidSequential2022}, such processes $E$ can be re-normalized to be true e-processes for all $m\in\N$ sufficiently large for the integrability of \cref{thm:characterizations r-SAEP}\labelcref{item:characterizations r-SAEP:npsms integrability} to occur.
This strongly supports the relevance of allowing general sequences $r$ in \cref{def:asymptotic e-process} for a genuinely asymptotic notion, since $\infty$-\glspl{saep} are e-processes up to normalization. 
\end{remark}

\subsection{Limiting behavior}
\label{sec:limiting behavior}

We conclude this section by showing that \cref{def:asymptotic e-process} generalizes the idea of ``converging to an e-process'' in the following sense:
for bi-indexed processes that converge in $L_1$ along the first index, the notions of $r$-\gls{saep} and of $r$-\gls{aep} coincide, and are equivalent to the limit being an e-process.
This further justifies the relevance of allowing for general sequences $r$ other than $r\equiv\infty$ in \cref{def:asymptotic e-process}.
\begin{restatable}{theorem}{aepWithConvergence}\label{thm:asymptotic e-processes converge to e-processes}
    Let $\filtration$ be a filtration array, $E = (E_{m,n})_{m,n\in\N}$ be a nonnegative process adapted to $\filtration$, and $F = (F_n)_{n\in\N}$ be a process adapted to $\filtration_{\infty,\bullet}$.
    Assume that $E$ converges to $F$ in $L_1$ uniformly in $\subsetProbaMeasures$ along the first index.
    The following statements are equivalent: \begin{enumerate}[label=(\roman*)]
        \item \label{item:asymptotic e-processes converge to e-processes:r-AEP} there exists an extended integer sequence $r=(r_m)_{m\in\N}$ with $r_m\to\infty$ as $m\to\infty$ such that $E$ is a uniformly $r$-\gls{aep} for $\filtration$ and $\subsetProbaMeasures$;
        \item \label{item:asymptotic e-processes converge to e-processes:r-SAEP} there exists an extended integer sequence $r=(r_m)_{m\in\N}$ with $r_m\to\infty$ as $m\to\infty$ such that $E$ is a uniformly $r$-\gls{saep} for $\filtration$ and $\subsetProbaMeasures$;
        \item \label{item:asymptotic e-processes converge to e-processes:e-process} $F$ is an e-process for $\filtration_{\infty,\bullet}$ and $\subsetProbaMeasures$.
    \end{enumerate}
    If one of these statements holds (and thus, all hold), then $E$ is a uniformly $r$-\gls{saep} (and, thus, also a uniformly $r$-\gls{aep}) for any extended integer sequence $r=(r_m)_{m\in\N}\subset\N\cup\{\infty\}$ that satisfies \begin{equation}\label{eq:r of converging asymptotic e-process}
        \lim_{m\to\infty}\sum_{n=0}^{r_m}\sup_{P\in\subsetProbaMeasures}\E_P[\lvert E_{m,n}-F_{n}\rvert] = 0,
    \end{equation}
    and there exists at least one such sequence with $r_m\to\infty$ as $m\to\infty$.
\end{restatable}

\section{Asymptotic supermartingales}
\label{sec:asymptotic-supermartingales}
We now proceed to generalize the
theory of test supermartingales in an asymptotic
context and relate it to \cref{def:asymptotic e-process}.
The general goal is to find sufficient criteria for $r$-\glspl{saep} that can be verified from simple assumptions and used to find suitable sequences $r$, similarly to the fact that, in the nonasymptotic case, test supermartingales are e-processes.
In particular, we are after conditions that are easier to verify than \cref{thm:characterizations r-SAEP}\labelcref{item:characterizations r-SAEP:iii}.
We focus on \glspl{saep}, and leave the extension to \glspl{aep} for future work.
\begin{definition}[Asymptotic supermartingale in $L_1$]\label{def:asymptotic supermartingale}
    Let $\filtration = (\filtration_{m,n})_{m,n\in\N}$ be a filtration array and $E = (E_{m,n})_{m,n\in\N}$ be a nonnegative process adapted to $\filtration$.
    We say that $E$ is an \emph{\glsfirst{asm}} (in $L_1$, for $\filtration$, uniformly in $\subsetProbaMeasures$) if there exists a nonnegative supermartingale $(S_n)_{n\in\N}$ for $\subsetProbaMeasures$ with respect to $\filtration_{\infty,\bullet}$ such that $E$ converges to $S$ in $L_1$ uniformly in $\subsetProbaMeasures$.
\end{definition}
We briefly discuss the relation between $S$, $E_{\infty,\bullet}$, and the question of whether this last process is a supermartingale.
To simplify, we reason for a fixed $P\in\subsetProbaMeasures$ in the discussion.
\par
The convergence required in \cref{def:asymptotic supermartingale} is in $L_1$, and not \gls{as}, and there is no relation in general between $S$ and $E_{\infty,\bullet}$, which we defined pointwise for $\omega\in\Omega$ as 
\begin{equation*}
    E_{\infty,n}(\omega) = \limsup_{m\to\infty}E_{m,n}(\omega).
\end{equation*}
We may have in general $E_{\infty,n}\neq S_n$ on events of positive measure.
The additional assumption that $E_{\infty,n}$ is an actual limit suffices to enforce $E_{\infty,n} = S_n$ \gls{as}, however.
Indeed, this follows from the fact that both \gls{as} and $L_1$ convergences imply convergence in probability to their respective limits, and thus the limits must coincide \gls{as}
In what follows, we do not make the assumption that $E_{\infty,n}$ is a true
\gls{as} limit, and thus stick to the distinct notation 
$S$ to denote the limit in $L_1$.
\par
Conversely, it does not suffice that $E_{\infty,\bullet}$ is a true limit and is a supermartingale for $E$ to be \pgls{asm} with $S = E_{\infty,\bullet}$, as the convergence may fail in $L_1$.
In fact, convergence in $L_1$ of $E_{\bullet,n}$, $n\in\N$ fixed, is equivalent to its \emph{uniform integrability}, since it converges in probability by assumption \citep[Theorem 4.6.3]{durrett2019probability}.
In other words, convergence in $L_1$ is a strong condition that may be tedious to verify.
Fortunately, it is also not necessary for our later results; all that is required is that $E$ behaves ``like'' a supermartingale as $m\to\infty$.
This is handled in the next definition.
\begin{definition}[Asymptotic supermartingale property]
\label{def:asp}
    Let $\filtration = (\filtration_{m,n})_{m,n\in\N}$ be a filtration sequence and $E = (E_{m,n})_{m,n\in\N}$ be a $\subsetProbaMeasures$-integrable\footnote{In consistency with the discussion after \cref{def:l1 convergence of bi-indexed process}, the results of this section generalize to the case where integrability holds only for $m\geq m_0$, with $m_0\in\N$ independent of $P\in\subsetProbaMeasures$.} nonnegative process adapted to $\filtration$.
    For all $m,n\in\N$ and $P\in\subsetProbaMeasures$, define \begin{equation*}
        \eta_{m,n} = \E_P[E_{m,n+1}\mid\filtration_{m,n}] - E_{m,n}
    \end{equation*}
    where the dependency of $\eta_{m,n}$ on $P$ is omitted to 
    simplify notation.
    We say that $E$ has the \emph{\glsfirst{asp}} (in $L_1$, for $\filtration$, uniformly in $\subsetProbaMeasures$) if  \begin{equation*}
        \forall n\in\N,~\lim_{m\to\infty}\sup_{P\in\subsetProbaMeasures}\E_P[\eta_{m,n}^+] = 0.
    \end{equation*}
\end{definition}
In the discussion, we omit for conciseness the qualifiers ``in $L_1$'', ``for $\filtration$'', and ``uniformly in $\subsetProbaMeasures$'' when talking about \glspl{asm} or the \gls{asp}.
Naturally, \glspl{asm} have the \gls{asp}.
\begin{restatable}{theorem}{asmImpliesAsp}\label{thm:asm implies asp}
    Let $E = (E_{m,n})_{m,n\in\N}$ be a nonnegative process and $\filtration = (\filtration_{m,n})_{m,n\in\N}$ be a filtration array.
    Assume that $E$ is \pgls{asm} in $L_1$ for $\filtration$ uniformly in $\subsetProbaMeasures$.
    Then, $E$ has the \gls{asp} in $L_1$ for $\filtration$ uniformly in $\subsetProbaMeasures$.
\end{restatable}
The converse implication does not hold; that is, not every process with the \gls{asp} is an asymptotic supermartingale.
The reason is that the \gls{asp} only constrains the drift part of the Doob decomposition.
More specifically, fix $P\in\subsetProbaMeasures$ and introduce $(M,A)$ the Doob decomposition of $E$ along the $n$-index; that is, for all $m\in\N$ and $n\in\N$,\begin{equation*}
    E_{m,n} = M_{m,n} + A_{m,n},
\end{equation*}with $M_{m,\bullet}$ an $\filtration_{m,\bullet}$-martingale and $A_{m,\bullet}$ being $\filtration_{m,\bullet}$-predictable \citep[Theorem 12.11]{williams1991martingales}.
It holds that $A_{m,n} = \sum_{i=0}^{n} \eta_{m,i}$ for $m,n\in\N$
and $\eta_{m,n} = \E_P[E_{m,n+1}\mid\filtration_{m,n}] - E_{m,n}$.
Introduce now $M^\prime$ another bi-indexed process such that $M^\prime_{m,\bullet}$ is a martingale and such that $E^\prime_{m,n} := M^\prime_{m,n} + A_{m,n}$ is nonnegative, $m,n\in\N$.
It is immediate to see that we also have $\eta_{m,n} = \E_P[E^\prime_{m,n+1}\mid\filtration_{m,n}] - E^\prime_{m,n}$, and thus $E$ has the \gls{asp} if, and only if, $E^\prime$ does.
This enables constructing processes with the \gls{asp} that do not converge along the $m$-axis to integrable processes.
For instance, assuming that $E$ is such that $A$ is bounded \gls{as} by some constant $B\in\Rp$, taking $M_{m,n}^\prime = B+m$ above certainly defines a martingale along the $n$-index such that $E^\prime_{m,n}$ is nonnegative, and $E^\prime_{m,n}\to\infty$ as $m\to\infty$, \gls{as}, and thus $E_{\bullet,n}^\prime$ does not converge in $L_1(\Omega,\sigAlg,P)$.
Another example is the process $E_{m,n} = 2+\sin(m)$, which is such that $E_{m,\bullet}$ is a nonnegative martingale for all $m\in\N$ (and, thus, has the \gls{asp}, since $\eta_{m,n} = 0$ for all $m,n\in\N$), but $E_{\bullet,n}$ does not converge in $L_1$ for any $n\in\N$.

\subsection{A sufficient condition for uniformly strongly \texorpdfstring{$r$}{r}-asymptotic e-processes}
We recall the standard 
result that test supermartingales are e-processes.
A similar result holds in the asymptotic case: the \gls{asp} is sufficient for being an $r$-\gls{saep} for some $r=(r_m)_{m\in\N}$ with $r_m\to\infty$, under an additional calibration requirement.
The proof begins with the following observation.
\begin{restatable}{lemma}{rSummableIncrements}\label{lemma:r-summable increments}
    Let $\filtration = (\filtration_{m,n})_{m,n\in\N}$ be a filtration sequence and $E = (E_{m,n})_{m,n\in\N}$ be a nonnegative process adapted to $\filtration$.
    Assume that $E$ has the \gls{asp} in $L_1$ for $\filtration$ uniformly in $\subsetProbaMeasures$.
    Then, there exists $r=(r_m)_{m\in\N}\subset\N$ with $r_m\to\infty$ as $m\to\infty$ such that \begin{equation}\label{eq:r-summable increments}
        \lim_{m\to\infty}\sup_{P\in\subsetProbaMeasures}\E_P\left[\sum_{n=0}^{r_m-1}\eta_{m,n}^+\right] = 0,
    \end{equation}
    where $\eta_{m,n}$ is given in \cref{def:asp}.
\end{restatable}
We are now equipped to state the sufficient condition on the sequence $r$.
We begin by naming the calibration condition we require.
\begin{definition}[Asymptotic calibration]\label{def:asymptotic calibration}
    A nonnegative process $E=(E_{m,n})_{m,n\in\N}$ is said to be \emph{asymptotically calibrated} for $\subsetProbaMeasures$ if \begin{equation}\label{eq:calibration}
        \limsup_{m \to \infty}
        \sup_{P\in\subsetProbaMeasures}\E_P[E_{m,0}]\leq 1.
    \end{equation}
\end{definition}
\begin{restatable}{theorem}{aspImpliesSaep}\label{thm:asp implies asymptotic e-process}
    Let $\filtration$ be a filtration sequence and $E = (E_{m,n})_{m,n\in\N}$ be a process adapted to $\filtration$ that has the \gls{asp} in $L_1$ for $\filtration$ uniformly over $\subsetProbaMeasures$.
    If $E$ is asymptotically calibrated for $\subsetProbaMeasures$, then $E$ is a uniformly $r$-\gls{saep} for $\subsetProbaMeasures$ and $\filtration$ for any sequence $r\subset\N\cup\{\infty\}$ that satisfies \cref{eq:r-summable increments}.
\end{restatable}
\paragraph*{Discussion\preprintdot} The criterion \cref{eq:r-summable increments} is essentially a formalization of the idea mentioned in the introduction to truncate the diverging series \cref{eq:diverging series}.
Specifically, the quantity \begin{equation}\label{eq:summable increments}
    \E_P\left[\sum_{n=0}^\infty \eta_{m,n}^+\right]
\end{equation}
may be infinite for every fixed $m\in\N$ for some (or all) $P\in\subsetProbaMeasures$.
This directly results from the fact that each $E_{m,\bullet}$ may not be a supermartingale for any finite $m\in\N$; if it were, every summand would be \gls{as} $0$.
Requiring this quantity to be finite is also a strong property, however.
Indeed, it is easy to see that if \cref{eq:summable increments} is finite for some $m\in\N$ and $E_{m,0}$ is integrable, then $E_{m,\bullet}$ is an \emph{almost} supermartingale, in the sense that it satisfies the assumptions of \citet[Theorem 1]{robbinsConvergenceTheoremNonnegative1971}.
In contrast, allowing a sequence $r$ of (finite) numbers provides flexibility: $r$ needs to increase slowly enough for the sum in \cref{eq:r-summable increments} to vanish, and it is always possible to find such a sequence under the condition that each summand $\eta^+_{m,n}$ vanishes as $m\to \infty$, which is precisely the \gls{asp}.
This gives an intuitive interpretation of the sequence $r$: it should increase slowly enough so that the errors incurred by the fact that $E_{m,\bullet}$ is not a supermartingale do not accumulate over the considered horizon.

\subsection{Uniformly strongly \texorpdfstring{$r$}{r}-asymptotic e-processes without the ASP}\label{subsec:asymptotic e-processes without the asp}

The \gls{asp}, in combination with calibration, is only a sufficient condition for $r$-\glspl{saep};
it is not necessary.
This mirrors the fact that, in the nonasymptotic regime, there are e-processes that are not supermartingales.
This observation gives a general recipe to construct $r$-\glspl{saep} that do not have the \gls{asp}: it suffices that they converge in $L_1$ to an e-process that is not a supermartingale.
This is based on the following result, which can be seen as a converse result to \cref{thm:asm implies asp} and guarantees that the limit in $L_1$ of a process that has the \gls{asp} must be a supermartingale when it exists.
\begin{restatable}{theorem}{aspAndConvergenceImplyAsm}\label{thm:asp and convergence implies asm}
    Let $\filtration$ be a filtration array, $E = (E_{m,n})_{m,n\in\N}$ be a nonnegative process adapted to $\filtration$, and $S = (S_n)_{n\in\N}$ be a process adapted to $\filtration_{\infty,\bullet}$.
    If $E$ has the \gls{asp} in $L_1$ for $\filtration$ uniformly in $\subsetProbaMeasures$ and converges to $S$ in $L_1$ uniformly in $\subsetProbaMeasures$, then $S$ is a supermartingale for $\subsetProbaMeasures$ with respect to $\filtration_{\infty,\bullet}$.
\end{restatable}
\Cref{thm:asp and convergence implies asm} enables constructing $r$-\glspl{saep} that do not have the \gls{asp}.
Indeed, for processes that converge in $L_1$ along the first axis, having the \gls{asp} is equivalent to the limit process being a supermartingale, by \cref{thm:asm implies asp,thm:asp and convergence implies asm}.
Therefore, if one has at hand an e-process that is \emph{not} a supermartingale, any process that converges to it in $L_1$ along the first axis is an $r$-\gls{saep} for a sequence $r$ chosen appropriately as a function of the speed of convergence in $L_1$, by \cref{thm:asymptotic e-processes converge to e-processes}, and it does not have the \gls{asp}.
Relevant examples are processes that converge in $L_1$ along the first axis to \emph{time mixture e-processes} \cite[Section 7.9]{ramdasHypothesisTestingEvalues2025}.
We discuss this in more detail in \cref{subsec:time mixture}.

\section{Asymptotic Ville's inequality}
\label{sec:asymp-ville}
We now present an essential result of this work that relates asymptotic e-processes and asymptotic \gls{savi} to asymptotic testing:
an asymptotic analogue of Ville's inequality for $r$-\glspl{aep} and $r$-\glspl{saep}.
It reflects the core idea of \cref{def:asymptotic e-process} by providing an asymptotic bound for the excursion probabilities of $E_{m,\bullet}$.
Crucially, the bound only concerns a horizon up to $r_m$, establishing $r$ as the time horizon that evolves with the approximation index $m$ and until which $E$ can be monitored with asymptotic guarantees.
Intuitively, this ties $r$ to how close $E$ is to being a true e-process for testing purposes.
\begin{restatable}{theorem}{ville}
    \label{thm:asymptotic-ville}
    Let $E = (E_{m,n})_{m,n\in\N}$ be a nonnegative process adapted to a filtration sequence $\filtration$, and $r = (r_m)_{m\in\N}\subset\N\cup\{\infty\}$ be a sequence of extended integers.
    If $E$ is a uniformly $r$-\gls{saep}, it holds for all $\alpha\in(0,1)$ that 
    \begin{equation*}
        \limsup_{m\to\infty} \sup_{P\in\subsetProbaMeasures}P\left[\sup_{n\in\integers{r_m}} E_{m,n} \geq \frac1\alpha\right]\leq\alpha.
    \end{equation*}
    If, instead, $E$ is a uniformly $r$-\gls{aep}, it holds for all $\alpha\in(0,1)$ that \begin{equation*}
        \limsup_{m\to\infty} \sup_{P\in\subsetProbaMeasures}P\left[\sup_{n\in\integers{r_m}} E_{m,n}\pwMin t_m \geq \frac1\alpha\right]\leq\alpha,
    \end{equation*}
    where $(t_m)_{m\in\N}\subset\Rnn$ is any sequence that satisfies \cref{eq:characterization r-AEP}.
    In particular, it holds with $t_m := t$ for all $m\in\N$, where $t\in\Rnn$ is fixed.
\end{restatable}
\section{Construction}
\label{sec:construction}

In this section, we investigate general methods for the construction of asymptotic e-processes.
We begin with the cumulative product, which is the cornerstone construction in the nonasymptotic case and immediately runs into difficulties if one tries to avoid the double-indexing we introduce (\cref{subsec:counterexample}).
We show in \cref{subsec:cumulative-product} that the cumulative product of uniformly strongly asymptotic conditional e-variables yields an $r$-\gls{saep}, with an explicit and simple criterion on the sequence $r$ hinging on \cref{thm:asp implies asymptotic e-process}.
We also discuss in \cref{subsec:time mixture} how the time-mixture construction of \citet[Section 7.9]{ramdasHypothesisTestingEvalues2025} carries to the asymptotic case.
After that, we investigate in \cref{subsec:event partitioning} how $r$-\glspl{aep} appear by following the \emph{event partitioning} method introduced in \cite{chuggPostHocLargeSampleStatistical2026}.
Finally, we identify in \cref{subsec:anytime p-values} how the distribution-uniform anytime p-values introduced in \cite{waudby-smithDistributionuniformAnytimevalidSequential2026} are precisely \emph{uniformly $\infty$-asymptotic p-processes}, up to reindexing, which leads us to generalize them with uniformly (strongly) $r$-asymptotic p-processes.
Crucially, calibrating such processes yields $r$-\glspl{aep} and $r$-\glspl{saep}, extending the existing duality between e- and p-variables to the asymptotic setting.
In the context of the previous works of \cite{waudby-smithDistributionuniformAnytimevalidSequential2026} and \cite{chuggPostHocLargeSampleStatistical2026}, our work introduces a layer of abstraction that simplifies
the overall analysis and generalizes constructions.
\par
In what follows, we consider a nonnegative process $(e_{m,n})_{m,n\in\N}$, which represents the asymptotic e-variables used in constructing 
corresponding asymptotic e-processes.
We also define $\filtration = (\filtration_{m,n})_{m,n\in\N}$, the filtration array generated by $e$.
Each subsection imposes individual assumptions, and formal statements redefine all symbols in a self-contained way.

\subsection{Counterexample: diagonal cumulative product}\label{subsec:counterexample}
The first example we investigate is the one that arises naturally when trying the usual construction of e-processes from sequential e-variables, which are e-variables $(\bar e_n)_{n\in\N}$ such that $\E_P[\bar e_{n+1}\mid \bar e_0,\dots, \bar e_n]\leq 1$, \gls{as} and for all $n\in\N$ and $P\in\subsetProbaMeasures$ \citep[Definition 7.19]{ramdasHypothesisTestingEvalues2025}.
Given such a sequence, the cumulative product \begin{equation*}
    \bar E_n = \prod_{i=0}^n \bar e_i
\end{equation*}
is a test supermartingale and, thus, is an e-process.
\par
In the case where only asymptotic e-variables $e = (e_{m,n})_{m,n\in\N}$ are available, there are many options to account for the second indexing dimension.
We consider first the following option: \begin{equation*}
    E_{m,n} = \prod_{i=0}^{m+n}e_{i,i}.
\end{equation*}
The choice of $m+n$ as the limit to the cumulative product is not central and other options are possible; we make this choice to ground the discussion.
For an example illustrating the relevance of this general form for $E$, one may consider the standard case where an underlying data set $(X_d)_{d\in\N}$ is revealed sequentially.
Given $d+1\in\N$ observed samples, one reserves the most recently revealed sample and computes an estimator based on the other $d$ observations.
Then, one computes a statistic denoted as $f_{d+1}$ involving the reserved sample and the estimator.
Mirroring the nonasymptotic construction $\bar E$, one is lead to consider the cumulative product $F_n = \prod_{d=0}^n f_d$ as a candidate signal to monitor.
This is for instance one of the examples in \cite{berrettConditionalPermutationTest2020}, though in a context without e-processes.
The construction of $E$ above simply frames this in the double-indexed framework considered in this work, by defining $e_{i,j} := f_{i+j}$ and $E_{m,n} = F_{m+n}$.
\par
In general, the process $E$ is not an asymptotic e-process without strong assumptions on the variables $e$.
For instance, consider the case where for all $i\in\N$ and $P\in\subsetProbaMeasures$, $\E_P[e_{i+1,i+1}\mid \filtration_{i,i}] \geq 1 + \varepsilon_i$, with $\varepsilon_i>0$ a non-random sequence (possibly vanishing asymptotically).
Then, for all $m,n\in\N$, \begin{equation*}
    \E_P[E_{m,n}] \geq \E_P[e_{0,0}]\cdot\prod_{i=0}^{m+n -1}(1+\varepsilon_i).
\end{equation*}
Assuming $\E_P[e_{0,0}] = 1$, this shows that $\E_P[E_{m,n}]$ is certainly not less than $1$ in limit superior for any $n\in\N$, from which it follows that $E$ is not an $r$-\gls{saep} for any sequence $r$.
A similar reasoning shows that $E$ is not an $r$-\gls{aep} for any sequence $r$ if for all $i\in\N$, $P\in\subsetProbaMeasures$, and $t\in\R$, $\E_P[e_{i+1,i+1}\pwMin t\mid \filtration_{i,i}] \geq 1 + \varphi_i(t)$, where $\varphi_i:\Rnn\to\Rnn$ is nonzero.
\par
An intuitive explanation for this negative result is the observation that the expectations $\E_P[E_{m,n}]$ and $\E_P[E_{m,n}\pwMin t]$ are heavily influenced by the properties of variables $e_{j,i}$ for low values of $i+j$.
These variables may be poor approximations of true e-variables and their (conditional) expectations may exceed $1$; such past excesses are never ``forgotten'' due to the cumulative product construction.
This hints that the issue lies in the cumulation of the approximation errors
$\varepsilon_m$, 
and that a possible fix is to allow a form of forgetting along that index.
There are many options for such forgetting; the next section explores a simple one.

\subsection{Cumulative product}
\label{subsec:cumulative-product}
We consider the process given by 
\begin{equation*}
    E_{m,n} = \prod_{i = 0}^n e_{m,i}.
\end{equation*}
In light of the observations of \cref{subsec:counterexample}, this corresponds to a total forgetting; $E_{m,n}$ only involves the variables $(e_{m,i})_{i\in\N}$.
Under the assumption that $e$ satisfies a conditional version of the defining property of uniformly strongly asymptotic e-variables, then $E$ is an $r$-\gls{saep} for a suitable sequence $r$.
\begin{restatable}{theorem}{cumulativeProduct}\label{thm:cumulative product}
    Let $(e_{m,n})_{m,n\in\N}$ be a nonnegative process adapted to a filtration sequence $\filtration = (\filtration_{m,n})_{m,n\in\N}$, and define \begin{equation}
        \quad E_{m,n} = \prod_{i=0}^ne_{m,i}.
    \end{equation}
    Assume that, for all $m\in\N$, there exists $\varepsilon_m\in\Rnn$ with \begin{equation}\label{eq:conditionally uniformly SAEV}
        \forall P\in\subsetProbaMeasures,\forall n\in\N,\quad\E_P[e_{m,n+1}\mid\filtration_{m,n}]\leq 1 + \varepsilon_m,\quad P\text{-\gls{as}}
    \end{equation}
    If $\sup_{P\in\subsetProbaMeasures}\E_{P}[e_{\bullet,0}]$ is bounded and $\varepsilon_m\to0$ as $m\to\infty$, then $E$ has the \gls{asp} for $\filtration$ uniformly in $\subsetProbaMeasures$.
    If, additionally, $e_{\bullet,0}$ is a uniformly strongly asymptotic e-variable, then so is $e_{\bullet,n}$ for all $n\in\N$, and $E$ is asymptotically calibrated for $\subsetProbaMeasures$.
    In particular, $E$ is a uniformly $r$-\gls{saep} for $\subsetProbaMeasures$ and $\filtration$ for any sequence $r = (r_m)_{m\in\N}\subset\N$ such that $r_m\cdot \varepsilon_m \to 0$ as $m\to\infty$.
\end{restatable}

\paragraph*{Discussion of other forms of forgetting\preprintdot}
The example presented in this section can be seen as an ``extreme'' form of forgetting along the $m$-axis: the value of the candidate e-process at index $(m,n)\in\N^2$ only involves the variables $(e_{m,i})_{i\in\N}$.
One may consider milder options, such as \begin{align*}
    E_{m+1,n} = E_{m,n}^{\lambda_m}\cdot\prod_{i=0}^{n} e_{m,i},\quad\text{or}\quad
    E_{m+1,n} = C_{m}\left[\lambda_m E_{m,n} + \prod_{i=0}^{n} e_{m,i}\right],
\end{align*}
where $\lambda = (\lambda_m)_{m\in\N}\subset[0,1)$ is a sequence of forgetting parameters, and $C = (C_m)_{m\in\N}$ is a sequence of normalization constants that depends only on $\lambda$ (ensuring for instance that $E_{m,\bullet}$ is indeed an e-process for all $m\in\N$ if $e$ consists of mutually independent, exact e-variables).
We expect that introducing such a recurrence leads to more constraining requirements on $e$ for such processes to have the \gls{asp} compared to the condition that $\varepsilon_m\to 0$ in \cref{thm:cumulative product}; indeed, the conditions should probably impose that $d$ decays faster than a speed parameterized by $\lambda$.
In turn, such asymptotic e-processes may have desirable properties in the specific application at hand (such as power in statistical testing, or improved stability upon increasing the index $m$).
In this work, we limit ourselves to illustrating simple ways of constructing asymptotic e-processes, and leave such explorations for future work.

\subsection{Weighted average}\label{subsec:time mixture}
We extend the construction of \citet[Section 7.9]{ramdasHypothesisTestingEvalues2025} to the asymptotic case.
Assume that $e$ consists of uniformly strongly asymptotic
e-variables and let, for all $P \in \subsetProbaMeasures$ and $m,n \in \N$,
$\E_P[ e_{m,n}] \leq 1 + \varepsilon_{m,n}^P$, with $\varepsilon_{m,n}^P \geq 0 $ and
$\varepsilon^P_{m,n} \to 0$ as $m \to \infty$ for every $n \in \N$.
For a collection of deterministic
nonnegative weights $(w_{m,n})_{m,n\in\N}$
such that $\limsup_{m\to\infty}\sum_{i =0}^\infty w_{m,i} \leq 1,$
we consider the weighted average process
(also called \emph{time mixture process},
\citealp{ramdasHypothesisTestingEvalues2025}) given by
\begin{equation*}
    E_{m,n} := \sum_{i=0}^n w_{m,i}\, e_{m,i}.
\end{equation*}
For every integer sequence $r=(r_m)_{m\in\N}$ and stopping time $\tau
\in \rBoundedStoppingTimes(r,\filtration,\subsetProbaMeasures)$, we have
\begin{align*}
\mathbb{E}_P[E_{m,\tau}]
&= \mathbb{E}_P\!\left[\sum_{t=0}^\infty \indicator_{\{\tau=t\}}
   \sum_{i=0}^t w_{m,i} e_{m,i}\right] \\
&= \mathbb{E}_P\!\left[\sum_{i=0}^\infty w_{m,i} e_{m,i}
   \sum_{t=i}^\infty \indicator_{\{\tau=t\}}\right] \\
&= \mathbb{E}_P\!\left[\sum_{i=0}^\infty w_{m,i} e_{m,i}\indicator_{\{\tau\ge i\}}\right] \\
&\le \sum_{i=0}^{r_m} w_{m,i}\,\mathbb{E}_P[e_{m,i}] \\
&\le \sum_{i=0}^{r_m} w_{m,i}(1+\varepsilon_{m,i}^P) \\
&\le 1+\sum_{i=0}^{r_m} w_{m,i}\varepsilon_{m,i}^P.
\end{align*}
Hence, 
this process is an $r$-\gls{saep}
whenever 
\begin{equation*}\lim_{m \to \infty}  \sup_{P \in \subsetProbaMeasures}
\sum_{i=0}^{r_m} w_{m,i} \varepsilon^P_{m,i} = 0.\end{equation*}
\par
Additionally we may assume that, for every $n \in \N$,
each asymptotic e-variable $(e_{m,n})_{m\in\N}$ converges in $L_1$ to an e-variable $e_n$ uniformly over $\subsetProbaMeasures$, 
in the sense that \begin{equation*}
    \lim_{m\to\infty}\sup_{P\in\subsetProbaMeasures}\E_P[\lvert e_{m,n} - e_n\rvert] = 0.
\end{equation*}
Also assume that $w$ converges to a probability mass function on $\N$, in the sense that $w_n = \lim_{m\to\infty} w_{m,n}$ exists for all $n\in\N$, and $\sum_{n=0}^\infty w_n = 1$.
Then, one can define the limit process \begin{equation*}
    S_n = \sum_{i=0}^n w_{i}e_i, \forall n\in\N,
\end{equation*}
which is an e-process by the same reasoning as above
\citep[see also][Proposition 7.23]{ramdasHypothesisTestingEvalues2025}.
Yet, it is also nondecreasing, and is not a supermartingale without additional assumptions.
This makes $E$ an $r$-\gls{saep} that does not have the \gls{asp}, illustrating the discussion in \cref{subsec:asymptotic e-processes without the asp}.
\par
The overall reasoning carries over to $r$-\glspl{aep}.
If $e$ consists only of uniformly asymptotic e-variables, define for all $P \in \subsetProbaMeasures$ and $m,n \in \N$ constants $\delta_{m,n}^P\in\Rnn$ and $\varepsilon_{m,n}^P\in\Rnn$ such that, for all $t\in\Rnn$,
$\E_P[e_{m,n}\pwMin t] \leq 1 + \varepsilon_{m,n}^P + \delta_{m,n}^P\cdot t$, with $\varepsilon_{m,n}^P \geq 0 $, $\delta_{m,n}^P \geq 0$, and
$\varepsilon^P_{m,n} \to 0$ and $\delta^P_{m,n} \to 0$ as $m \to \infty$ for every $n \in \N$.
Taking a collection of weights $w$ as above, similar computations show that for every integer sequence $r=(r_m)_{m\in\N}$, $\tau\in\rBoundedStoppingTimes(r,\filtration,\subsetProbaMeasures)$, and $t\in\Rnn$, \begin{align*}
    \E_{P}[E_{m,\tau}\pwMin t] 
        &\leq \E_P\left[\left(\sum_{i=0}^{r_m}w_{m,i}\cdot e_{m,i}\right)\pwMin t\right]\\
        &\leq \sum_{i=0}^{r_m}\E_P\left[w_{m,i}\cdot e_{m,i}\pwMin t\right]\\
        &\leq \sum_{i=0}^{r_m}w_{m,i} (1+\varepsilon_{m,i}^P + \delta_{m,i}^P\cdot t),
\end{align*}
where the second inequality follows from $(a+b)\pwMin t\leq a\pwMin t + b\pwMin t$ for all $a,b\geq 0$.
As a result, it suffices that \begin{equation*}
\lim_{m \to \infty}  \sup_{P \in \subsetProbaMeasures}
\sum_{i=0}^{r_m} w_{m,i} \varepsilon^P_{m,i} = 0,\quad\text{and}\quad\lim_{m \to \infty}  \sup_{P \in \subsetProbaMeasures}
\sum_{i=0}^{r_m} w_{m,i} \delta^P_{m,i} = 0,
\end{equation*}
for this process to be an $r$-\gls{aep}.

\subsection{Domination, event partitioning, and burn-in}\label{subsec:event partitioning}

To motivate the next construction and understand its guarantees, let us introduce a motivating example.
Consider the case where a (single-indexed) data stream $X = (X_n)_{n\in\N}$ adapted to a filtration $\mathcal G = (\mathcal G_n)_{n\in\N}$ is available, and one has identified a measurable function $\mathcal E$ such that, for some parameter $\theta^\ast\in\R$, the process $E^\ast := (\mathcal E(X_{0:n},\theta^\ast))_{n\in\N}$ is an e-process (for $\mathcal G$ and $\subsetProbaMeasures$).
Yet, the parameter $\theta^\ast$ is assumed to be unknown, such that $E^\ast$ cannot be computed.
Such a setup explicitly appears in \citet[Section 3.4]{chuggPostHocLargeSampleStatistical2026}.
\par
A natural idea is to investigate what happens when one plugs in an \emph{estimator} of $\theta^\ast$ in $\mathcal E$, and   whether this results in an approximate or asymptotic e-process, in some sense.
In other words, one constructs the process $E := (\mathcal E(X_{0:n},\hat\theta_m))_{m,n\in\N}$, where $\hat\theta := (\Theta(X_{0:m}))_{m\in\N}$ is an estimator of $\theta^\ast$, for some measurable function $\Theta$.
This procedure can be reasonably described as a ``burn-in'' method, since monitoring $E_{m,\bullet}$ involves first collecting $m$ samples to evaluate $\hat \theta_m$; the process $E_{m,\bullet}$ is adapted to $\mathcal G_{m\pwMax\bullet}$.
As a result, the question is whether $E$ is an asymptotic e-process for $\subsetProbaMeasures$ and the filtration array defined as $\filtration_{m,n} = \mathcal G_{m\pwMax n}$, for $m,n\in\N$.
A consequence of the results of this section is that, under a monotonicity condition on $\mathcal E$, the process $\hat \theta$ can be modified such that $E$ is an $\infty$-\gls{aep} as long as $\hat \theta$ converges in probability.
The proof method involves altering $\hat\theta$ to construct, for every $m\in\N$, an \emph{overestimation} of $\theta^\ast$, and then leveraging monotonicity to upper bound the modified $E$ with $E^\ast$ on high-probability sets.
\par
Before giving the result, we discuss an immediate objection to the above construction, which is the reason why it is \emph{not} used in \citet[Proposition 4.4]{chuggPostHocLargeSampleStatistical2026} despite the reference motivating asymptotic e-processes from a burn-in perspective.
Indeed, at index $(m,n)\in\N^2$, with $m\leq n$ so that $E_{m,n}$ can be computed from $X_{0:n}$, the variable $E_{m,n}$ defined above leverages the estimator $\hat\theta_m$, but could in principle rely on $\hat\theta_n$ instead, since all of the data required for its computation is available (it is $\filtration_{m,n} = \mathcal G_n$-measurable).
That is, one is restricting the amount of data used in the estimation of $\theta^\ast$ and, hereby, the quality of the approximation.
Using instead all available data for the estimator leads to monitoring the process $F = (\mathcal E(X_{0:n},\hat\theta_n))_{n\in\N}$.
However, it is not directly an e-process in general, and one may wonder in what sense $F$ is an asymptotic e-process.
In particular, asymptotic e-processes are an inherently bi-indexed notion as per the preceding developments, and yet $F$ is single-indexed.
The key resides in understanding the dual role of the single index $n\in\N$; it controls both how well $\hat\theta$ approximates $\theta^\ast$ and the time along which the operator monitors.
The question asked by the statistician is then: ``when monitored from index $n\in\N$ on, how far is $F$ from being an e-process?''
Denoting the index at which monitoring starts as $m\in\N$, the previous question can be rephrased as whether the bi-indexed process $(F_{m+n})_{m,n\in\N}$ is an $\infty$-\gls{aep} or an $\infty$-\gls{saep}.
The results of this section answer this question positively in the case where $\mathcal E$ satisfies again the monotonicity condition.
It involves, however, modifying $\hat\theta$ such that the following stronger property holds: for all $m\in\N$, $\hat\theta_n$ overestimates $\theta^\ast$ for all $n\geq m$ (as opposed to only for $n=m$ previously), with high probability.
This leads to requiring the \emph{almost sure} convergence of $\hat\theta$ to $\theta^\ast$ instead of its convergence in probability.
\par
Let us summarize the above construction and identify its key arguments to formalize it.
The first core idea is that of an \emph{asymptotically almost-sure domination uniformly in $\subsetProbaMeasures$}: a process $E = (E_{m,n})_{m,n\in\N}$ that is upper-bounded by an $r$-\gls{aep} $E^\ast=(E^\ast_{m,n})_{m,n\in\N}$ on sets of asymptotically full probability (uniformly in $P\in\subsetProbaMeasures$) is itself an $r$-\gls{aep}.
The construction above uses for the upper-bound a true e-process independent of $m$, which is why its conclusions involve $\infty$-\glspl{aep}, but the generalization need not be restricted to those choices.
The second main argument is an immediate consequence: if one is able to construct an $\filtration_{m,n}$-measurable overapproximation of a parameter involved in an $r$-\gls{aep} and the function defining the $r$-\gls{aep} is nonincreasing in that parameter, the result is again an $r$-\gls{aep}.
The last argument is based on identifying when such overapproximations are possible in the presence of estimators that converge, in probability or almost surely.
\par
The first idea is formalized in the next result.
It can be seen as a generalization of \citet[Observation 3.7]{chuggPostHocLargeSampleStatistical2026} and of \citet[Proposition 3.2]{ignatiadisAsymptoticCompoundEvalues2024} to $r$-\glspl{aep}, and is used on a specific example in the proof of \citet[Proposition 4.4]{chuggPostHocLargeSampleStatistical2026}.
The latter reference calls this construction \emph{event partitioning}; we favor the terminology of domination by another $r$-\gls{aep} to emphasize the role played by the process serving as an upper bound.
\begin{restatable}[Asymptotically almost-sure domination]{theorem}{domination}\label{thm:domination}
    Let $E^\ast = (E^\ast_{m,n})_{m,n\in\N}$ and $E = (E_{m,n})_{m,n\in\N}$ be nonnegative processes adapted to a filtration sequence $\filtration$.
    Assume that there exists an extended integer sequence $r=(r_{m})_{m\in\N}$ such that $E^\ast$ is a uniformly $r$-\gls{aep}, and that \begin{equation*}
        \lim_{m\to\infty}\inf_{P\in\subsetProbaMeasures}P\left[A_m\right] = 1,\quad\text{where}\quad A_m := \left\{\forall n\in\integers{r_m},\quad E_{m,n}\leq E^\ast_{m,n}\right\}\in\sigAlg,\quad\forall m\in\N.
    \end{equation*}
    Then $E$ is a uniformly $r$-\gls{aep} for $\filtration$ and $\subsetProbaMeasures$.
    Furthermore, any truncation sequence\footnote{We remind the reader that this term is defined in \cref{prop:characterization uniform in stopping time}.} $t = (t_m)_{m\in\N}\subset\Rnn$ of $E^\ast$ such that $t_m\cdot\sup_{P\in\subsetProbaMeasures}P[A_m^\complement] \to 0$ as $m\to\infty$ is a truncation sequence of $E$, and there exists at least one such sequence.
\end{restatable}
As announced, it follows immediately that overestimating a parameter on which the process depends nonincreasingly yields an $r$-\gls{aep}.
The following result uses the notation $\bar\theta$ for the overestimation to distinguish it from the consistent estimator $\hat\theta$ mentioned in the introductory paragraph; clarifying the relationship between the two is the purpose of a subsequent result.
\begin{restatable}[Domination via monotonicity and overestimation]{corollary}{dominationViaMonotonicity}\label{clry:domination via monotonicity}
    Let $\filtration = (\filtration_{m,n})_{m,n\in\N}$ be a filtration sequence and $E = (E_{m,n}(\theta))_{(m,n,\theta)\in\N\times\N\times\R}$ be a nonnegative process such that the map $(\omega,\theta)\in\Omega\times\R\mapsto E_{m,n}(\theta)(\omega)$ is $\filtration_{m,n}\otimes\mathcal B(\R)$-measurable, for $m,n\in\N$, where $\mathcal B(\R)$ denotes the Borel $\sigma$-algebra on $\R$.
    Let $\bar\theta = (\bar\theta_{m,n})_{m,n\in\N}$ be a finite real process adapted to $\filtration$, let $\theta^\ast\in\R$, and let $r$ be an extended integer sequence.
    Assume that the following holds: \begin{enumerate}[label=(\roman*)]
        \item \label{item:true aep} $E(\theta^\ast)$ is a uniformly $r$-\gls{aep} for $\subsetProbaMeasures$ and $\filtration$,
        \item \label{item:domination} for any $m,n\in\N$ and $\omega\in\Omega$, the function $E_{m,n}(\cdot)(\omega)$ is nonincreasing on $[\theta^\ast,\infty)$,
        \item \label{item:overestimation} it holds that \begin{equation}\label{eq:overestimation}
            \lim_{m\to\infty}\inf_{P\in\subsetProbaMeasures} P[D_m] = 1, \quad\text{where}\quad D_m := \left\{\forall n\in\integers{r_m}, \bar\theta_{m,n}\geq\theta^\ast\right\}\in\sigAlg,\quad\forall m\in\N.
        \end{equation}
    \end{enumerate}
    Then, the process $E := (E_{m,n}(\bar\theta_{m,n}))_{m,n\in\N}$ is a uniformly $r$-\gls{aep}.
    Furthermore, any truncation sequence $t = (t_m)_{m\in\N}\subset\Rnn$ of $E(\theta^\ast)$ such that $t_m\cdot\sup_{P\in\subsetProbaMeasures}P[D_m^\complement] \to 0$ as $m\to\infty$ is a truncation sequence of $E$, and there exists at least one such sequence.
\end{restatable}
For practical purposes, there is a simple way to consistently overestimate scalar parameters if one has at hand a consistent estimator.
Depending on the strength of the convergence at hand, this results in a variety of possible overestimations.
\begin{restatable}[Overestimation via convergence]{lemma}{overestimationViaConvergence}\label{lemma:overestimation via convergence}
    Let $\hat\theta = (\hat\theta_m)_{m\in\N}$ be a finite real process and let $\theta^\ast\in\R$.
    If $\hat\theta_m\to\theta^\ast$ as $m\to\infty$ in probability uniformly in $\subsetProbaMeasures$; that is, if \begin{equation}\label{eq:uniform convergence in probability}
        \forall \epsilon>0,\quad\lim_{m\to\infty} a_m(\epsilon) = 0,\quad\text{where}\quad a_m(\epsilon) := \sup_{P\in\subsetProbaMeasures} P[\lvert \hat\theta_m-\theta^\ast\rvert>\epsilon],\quad\forall m\in\N,
    \end{equation}
    then there exist nonnegative sequences $(\epsilon_m)_{m\in\N}$ and $t=(t_m)_{m\in\N}$ such that $\epsilon_m\to0$, $t_m\to\infty$, and $t_m\cdot a_m(\epsilon_m)\to 0$ as $m\to\infty$.
    In particular, the process $\bar\theta$ defined as $\bar\theta_{m,n} = \hat\theta_m + \epsilon_m$, $m,n\in\N$, satisfies \cref{clry:domination via monotonicity}\labelcref{item:overestimation}, and the sequence $t$ satisfies $t_m\cdot\sup_{P\in\subsetProbaMeasures}P[D_m^\complement]\to0$ as $m\to\infty$, where $(D_m)_{m\in\N}$ is defined in \cref{eq:overestimation}.
    \par
    If $\hat\theta_m\to\theta^\ast$ as $m\to\infty$ \gls{as} uniformly in $\subsetProbaMeasures$; that is, if \begin{equation}\label{eq:uniform as convergence}
        \forall \epsilon>0,\quad\lim_{m\to\infty} a_m^\prime(\epsilon) = 0,\quad\text{where}\quad a_m^\prime(\epsilon) := \sup_{P\in\subsetProbaMeasures} P[\lvert \sup_{k\geq m}\hat\theta_k-\theta^\ast\rvert>\epsilon],\quad\forall m\in\N,
    \end{equation}
    then there exist nonnegative sequences $(\epsilon_m)_{m\in\N}$ and $t=(t_m)_{m\in\N}$ such that $\epsilon_m\to0$, $t_m\to\infty$, and $t_m\cdot a_m^\prime(\epsilon_m)\to 0$ as $m\to\infty$.
    In particular, for every integer-valued process $b = (b_{m,n})_{m,n\in\N}$, the process $\bar\theta$ defined as $\bar\theta_{m,n} = \hat\theta_{m+b_{m,n}} + \epsilon_m$, $m,n\in\N$, satisfies \cref{clry:domination via monotonicity}\labelcref{item:overestimation}, and the sequence $t$ satisfies $t_m\cdot\sup_{P\in\subsetProbaMeasures}P[D_m^\complement]\to0$ as $m\to\infty$, where $(D_m)_{m\in\N}$ is defined in \cref{eq:overestimation}.
\end{restatable}
The characterization of distribution-uniform \gls{as} convergence we provide in \cref{eq:uniform as convergence} comes from the characterization \citep[Corollary 20.8]{bauerMeasureIntegrationTheory2001} that a finite real process $(Z_m)_{m\in\N}$ converges $P$-\gls{as} to a variable $Z$ if, and only if, \begin{equation*}
    \forall\epsilon>0,\quad \lim_{m\to\infty}P\left[\lvert\sup_{k\geq m} Z_k - Z\rvert>\epsilon\right] = 0.
\end{equation*}
The condition \cref{eq:uniform as convergence} is a distribution-uniform version of this characterization, and is extensively discussed in \cite{waudby-smithDistributionuniformAnytimevalidSequential2026}.
\par
The combination of \cref{clry:domination via monotonicity} and of \cref{lemma:overestimation via convergence} enables creating a wide variety of $r$-\glspl{aep}.
The only conditions to verify to that end are finding an appropriate sequence $\epsilon$, and the fact that the resulting overestimation process $\bar\theta$ needs to be adapted to the filtration sequence at hand.
A case of particular interest is when the processes at hand all result from an underlying data stream represented by its filtration.
We state this explicitly below.
\begin{restatable}{theorem}{burnIn}\label{thm:r-aep from burn in of single data stream}
    Let $\mathcal G = (\mathcal G_n)_{n\in\N}$ be a filtration, and $E = (E_{n}(\theta))_{(n,\theta)\in\N\times\R}$ be a nonnegative process such that $E(\theta) := (E_n(\theta))_{n\in\N}$ is adapted to $\mathcal G$ for all $\theta\in\R$.
    Assume that $E_n(\cdot)$ is nonincreasing, for all $n\in\N$, and that $(\omega,\theta)\in\Omega\times\R\mapsto E_n(\theta)(\omega)$ is $\mathcal G_n\otimes\mathcal B(\R)$-measurable.
    Let $\theta^\ast\in\R$, and assume that $E(\theta^\ast)$ is an e-process.
    Let $\hat\theta = (\hat\theta_m)_{m\in\N}$ be a finite real process adapted to $\mathcal G$.
    \begin{itemize}
    \item  
    If $\hat\theta_m\to\theta^\ast$ as $m\to\infty$ in probability uniformly in $\subsetProbaMeasures$, then the process 
        \begin{equation*}
            E^\mathrm{prob} = (E_{n}(\hat\theta_m + \epsilon_m))_{m,n\in\N}
        \end{equation*}
    is an $\infty$-\gls{aep} for $\subsetProbaMeasures$ and the filtration array $(\mathcal G_{m\pwMax n})_{m,n\in\N}$, for any sequence $(\epsilon_m)_{m\in\N}$ such that $\epsilon_m\to 0$ and $a_m(\epsilon_m)\to 0$ as $m\to\infty$, where $a_m$ is introduced in \cref{eq:uniform convergence in probability}.
    Furthermore, any nonnegative sequence $(t_m)_{m\in\N}$ such that $t_m\to\infty$ and $t_m\cdot a_m(\epsilon_m)\to 0$ as $m\to\infty$ is a truncation sequence for $E^\mathrm{prob}$.
    \item
    If $\hat\theta_m\to\theta^\ast$ as $m\to\infty$ \gls{as} uniformly in $\subsetProbaMeasures$, then the process 
    \begin{equation*}
        E^\mathrm{a.s.} = (E_{n}(\hat\theta_{m\pwMax n} + \epsilon_m))_{m,n\in\N}
    \end{equation*}
    is an $\infty$-\gls{aep} for $\subsetProbaMeasures$ and the filtration array $(\mathcal G_{m\pwMax n})_{m,n\in\N}$, for any sequence $(\epsilon_m)_{m\in\N}$ such that $\epsilon_m\to 0$ and $a_m^\prime(\epsilon_m)\to 0$ as $m\to\infty$, where $a_m^\prime$ is introduced in \cref{eq:uniform as convergence}.
    Furthermore, any nonnegative sequence $(t_m)_{m\in\N}$ such that $t_m\to\infty$ and $t_m\cdot a_m^\prime(\epsilon_m)\to 0$ as $m\to\infty$ is a truncation sequence for $E^\mathrm{a.s.}$.
    \end{itemize}
\end{restatable}
We emphasize that the choice of using $\hat\theta_{m\pwMax n}$ in the definition of $E^\mathrm{a.s.}$ is arbitrary; any choice of the form $\hat\theta_{m+b_{m,n}}$ is valid, at the price of adapting the filtration to preserve adaptedness.
\par
The construction of \citet[Proposition 4.4]{chuggPostHocLargeSampleStatistical2026} can be seen as an application of \cref{thm:r-aep from burn in of single data stream}, where the sequences $\epsilon$ and $t$ are made explicit.
Specifically, introducing the notation of the reference and renaming their $\theta$ as $\mu$ and their $\sigma^2$ as $\theta$ for consistency with our notation, since it is the unknown parameter, define for $m,n\in\N$\begin{align*}
    E_{n}(\theta) &= \exp\left[\lambda S_n(\mu) - \frac{\lambda^2}2\left(\frac23n\theta + \frac13\sum_{i=1}^n(X_i - \mu)^2\right)\right],\quad\text{and}\\
    \hat\theta_m &= \frac32\hat s_m^2 - \frac12m^{-1}\sum_{i=1}^m(X_i-\mu)^2.
\end{align*}
The reference argues that, when $\theta$ is set to $\theta^\ast$ the variance of the underlying i.i.d.\ data, $E(\theta^\ast)$ is an e-process.
The estimator $\hat\theta$ above is chosen such that, for $n\geq m$, \begin{equation}\label{eq:aep chugg et al}
    E_n(\theta_{m\pwMax n} + \epsilon_m) = \exp\left[\lambda S_n(\mu) - n\left(\hat s_n^2 + \frac23\epsilon_m\right)\frac{\lambda^2}2\right],
\end{equation}
which is the process defined in their Proposition 4.4 if $\epsilon_m = \frac32\log(m)^{-1}$, before truncation and neglecting the mixture on the parameter $\lambda$.
Our \cref{thm:r-aep from burn in of single data stream} guarantees that there exists a vanishing sequence $(\epsilon_m)_{m\in\N}$ such that \cref{eq:aep chugg et al} defines an $\infty$-\gls{aep}, by relying on the results of \cite{waudby-smithDistributionuniformAnytimevalidSequential2026}, and of \cite{chuggPostHocLargeSampleStatistical2026} to show the uniform \gls{as} convergence of $\hat\theta_m$ to $\theta^\ast$ when $\subsetProbaMeasures$ satisfies the moment conditions imposed in the reference.
\citet{chuggPostHocLargeSampleStatistical2026} have the more precise result that $\epsilon_m = \frac32\log(m)^{-1}$ is a suitable such sequence, and compute a corresponding truncation sequence.
\begin{remark}\label{rmk:index shift}
    The construction of \cref{thm:r-aep from burn in of single data stream} with an \gls{as} consistent estimator results in an $\infty$-\gls{aep} that is an upper-triangular array; specifically, via the index change $k := m+n$.
    This structure has been taken as a defining feature in the previous works of \cite{waudby-smithDistributionuniformAnytimevalidSequential2026} and \cite{chuggPostHocLargeSampleStatistical2026}, but our results show that this needs not be the case.
\end{remark}
\begin{remark}
    \cref{lemma:overestimation via convergence} highlights that the \emph{mode of convergence} of the estimator $\hat\theta$ (in probability, or almost sure) results in \emph{different \glspl{aep}}.
    The \gls{aep} obtained with convergence in probability involves the estimator evaluated at the current approximation index, $\hat\theta_m$.
    In contrast, the one obtained with \gls{as} convergence allows more general choices, $\hat\theta_{m+b_{m,n}}$ for any nonnegative $b_{m,n}$.
    In particular, the estimator involved in the \gls{aep} is allowed to evolve with $n\in\N$.
    One could imagine intermediate modes of convergence that allow some, but not all dependency with $n$.
    For instance, concluding on an estimator involving $\hat\theta_{m+b_{m,n}}$ where $b = (b_{m,n})_{m,n\in\N}$ is a process that is \emph{bounded} in $n\in\N$; say, $b_{m,n}\leq h_m$ for some $h_m\in\N$, is achieved by imposing the following mode of convergence: \begin{equation*}
        \forall \epsilon>0, \lim_{m\to\infty}\sup_{P\in\subsetProbaMeasures}P\left[\left\lvert \sup_{k\in\integers{m,m+h_{m}}}\hat\theta_k - \theta^\ast\right\rvert>\epsilon\right] = 0.
    \end{equation*}
    This mode of convergence of $\hat\theta$ to $\theta^\ast$ appears to be intermediate between convergence in probability and almost sure convergence as soon as $h_m\to\infty$ as $m\to\infty$.
    We do not investigate this in more detail, but point it out as an interesting avenue for future exploration.
\end{remark}

\subsection{Calibration of anytime p-values and asymptotic p-processes}\label{subsec:anytime p-values}
The recent work of \cite{waudby-smithDistributionuniformAnytimevalidSequential2026} proposes \emph{distribution-uniform anytime p-values} as a relevant object for asymptotic \gls{savi}.
In short, they are triangular arrays of variables that satisfy the condition defining anytime-valid p-values over one index, 
asymptotically in the other index.
In this section, we show how anytime p-values can be calibrated into $\infty$-\glspl{aep} by leveraging ideas from \cite{ignatiadisAsymptoticCompoundEvalues2024}.
More generally, this motivates the introduction of \glspl{app} and their strong version, which respectively result in $r$-\glspl{aep} and $r$-\glspl{saep} after calibration; the distribution-uniform anytime-valid p-values of \cite{waudby-smithDistributionuniformAnytimevalidSequential2026} correspond to $\infty$-\glspl{app}, up to the index shift described in \cref{rmk:index shift}.

\paragraph*{Distribution-uniform anytime p-values and asymptotic p-processes\preprintdot} We begin by recalling the definitions of p-variables, p-processes, and distribution-uniform anytime p-values.
We then introduce $r$-\glspl{app} and \glspl{sapp} as generalizations that fit into our framework.
\begin{definition}[p-variable, p-process]
    A p-variable for $\subsetProbaMeasures$ is a nonnegative, finite
    random variable $p$ such that, for all $\alpha\in(0,1)$ and $P\in\subsetProbaMeasures$, $P[p\leq \alpha]\leq\alpha$. 
    A p-process for $\subsetProbaMeasures$ and a filtration $\filtration$ is a nonnegative finite process $p = (p_n)_{n\in\N}$ such that, for all $\filtration$-stopping times $\tau$, the variable $p_\tau$ is a p-variable.
\end{definition}
To introduce distribution-uniform anytime p-values, we use the notation $I = \{(m,k)\in\N^2\mid k\geq m\}\subset\N^2$ for indexing upper-triangular arrays.
\begin{definition}[$\subsetProbaMeasures$-uniform anytime p-values, \citealp{waudby-smithDistributionuniformAnytimevalidSequential2026}]\label{def:anytime p-value}
    A finite nonnegative process $p = (p_{m,k})_{m,k\in I}$ is a \emph{$\subsetProbaMeasures$-uniform anytime p-value} if \begin{equation*}
        \limsup_{m\to\infty}\sup_{P\in\subsetProbaMeasures} P[\exists k\geq m, p_{m,k}\leq \alpha]\leq\alpha.
    \end{equation*}
\end{definition}
This requirement can be connected to the asymptotic p-variables introduced in \citet[Definition 3.5]{ignatiadisAsymptoticCompoundEvalues2024}.
We emphasize that we do not rely directly on the definition of the reference, but rather on an equivalent characterization that we present below as a definition.
The proof that this is indeed an equivalent characterization can be found in the supplementary materials.
\begin{definition}\label{def:uniformly asymptotic p-variables}
    Let $p = (p_m)_{m\in\N}$ be a stochastic process taking values in $\Rnn$.
    Then, $p$ is called a \emph{uniformly asymptotic p-variable} (for $\subsetProbaMeasures$) if
    \begin{equation}\label{eq:uniformly asymptotic p-variable}
\forall\alpha\in(0,1),~\limsup_{m\to\infty}\sup_{P\in\subsetProbaMeasures}P[p_m\leq\alpha]\leq \alpha.
    \end{equation}
    It is called a \emph{uniformly strongly asymptotic p-variable} (for $\subsetProbaMeasures$) if  
    \begin{equation}\label{eq:uniformly strongly asymptotic p-variable}
        \limsup_{m\to\infty}\sup_{P\in\subsetProbaMeasures}\sup_{\alpha\in(0,1)}\alpha^{-1} P[p_m\leq\alpha]\leq 1.
    \end{equation}
\end{definition}
This definition enables the following characterization.
\begin{restatable}{theorem}{anytimePValuesAreInftyApps}\label{thm:anytime p-values are infty-APPs}
    Let $p = (p_{m,k})_{m,k\in I}$ be a nonnegative finite process adapted to a triangular filtration sequence $\filtration = (\filtration_{m,k})_{m,k\in I}$.
    Introduce the shifted filtration $\bar\filtration = (\filtration_{m,m+n})_{m,n\in\N}$
    Then, $p$ is a $\subsetProbaMeasures$-uniform anytime-valid p-value if, and only if, for any sequence of stopping times $\tau = (\tau_m)_{m\in\N}\in\rBoundedStoppingTimes_\frm(\infty,\bar\filtration,\subsetProbaMeasures)$, the process $(p_{m,m+\tau_m})_{m\in\N}$ is a uniformly asymptotic p-variable.
\end{restatable}
This characterization strongly motivates a definition mirroring \cref{def:asymptotic e-process}, but with p-variables.
In particular, we abandon the triangular structure, and allow general horizon sequences $r$.
\begin{definition}[Uniformly $r$-asymptotic p-process]
    \label{def:asymptotic p-process}
    Let $p = (p_{m,n})_{m,n\in\N}$ be a nonnegative process adapted to a filtration sequence $\filtration$ and $r$ be a sequence of extended integers.
    We say that $p$ is a \begin{enumerate}[label=(\roman*)]
        \item \emph{uniformly $r$-asymptotic p-process ($r$-\glsps{app}\glsunset{app})} (for $\subsetProbaMeasures$ and for $\filtration$) if for all $\tau=(\tau_m)_{m\in\N}\in\rBoundedStoppingTimes_\frm(r,\filtration,\subsetProbaMeasures)$, the process $(p_{m,\tau_m})_{m\in\N}$ is a uniformly asymptotic p-variable for $\subsetProbaMeasures$,
        \item \emph{uniformly strongly $r$-asymptotic p-process ($r$-\glsps{sapp}\glsunset{sapp})} (for $\subsetProbaMeasures$ and for $\filtration$) if for all $\tau=(\tau_m)_{m\in\N}\in\rBoundedStoppingTimes_\frm(r,\filtration,\subsetProbaMeasures)$, the process $(p_{m,\tau_m})_{m\in\N}$ is a uniformly strongly asymptotic p-variable for $\subsetProbaMeasures$.
    \end{enumerate}
\end{definition}
Similarly as with \glspl{aep} and \glspl{saep}, we omit the qualifier ``uniformly'' in the discussion.
It follows that \cref{thm:anytime p-values are infty-APPs} identifies distribution-uniform anytime p-values with $\infty$-\glspl{app}, up to an index shift.
We do not provide a systematic study of $r$-\glspl{app} and $r$-\glspl{sapp}, but we expect that they can be characterized equivalently with finite or possibly infinite stopping times, similarly as $r$-\glspl{aep} and $r$-\glspl{saep}, by an application of \cref{thm:characterizations r-SAEP} similarly as in the proof of \cref{thm:anytime p-values are infty-APPs}.

\paragraph*{Calibration of asymptotic p-processes\preprintdot}
We now address the construction of $r$-\glspl{aep} and $r$-\glspl{saep} from $r$-\glspl{app} and $r$-\glspl{sapp} via calibration.
We begin by recalling the definition of a p-to-e calibrator.
\begin{definition}[p-to-e calibrator]
    A \emph{p-to-e calibrator} is a nonincreasing function $f:[0,\infty)\to[0,\infty]$ satisfying $f=0$ on $(1,\infty)$ and such that $f(p)$ is an e-variable for $\subsetProbaMeasures$ for any p-variable $p$.
\end{definition}
There exist many p-to-e calibrators, and they are well-studied; we refer to \citet[Section 2.3]{ramdasHypothesisTestingEvalues2025} for an extensive discussion.
We now state the main result on calibration.
\begin{restatable}[Calibration of asymptotic p-processes]{theorem}{calibration}\label{thm:calibration}
    Let $p = (p_{m,n})_{m,n\in \N}$ be a finite nonnegative process adapted to a filtration sequence $\filtration$, and $r$ be an extended integer sequence.
    Let $f:[0,\infty)\to[0,\infty]$ be a p-to-e calibrator, and define $E_{m,n} = f(p_{m,n})$, for all $m,n\in\N$.
    If $p$ is a uniformly $r$-\gls{app} for $\subsetProbaMeasures$ and $\filtration$, then $E$ is a uniformly $r$-\gls{aep} for $\subsetProbaMeasures$ and $\filtration$.
    If, additionally, $f$ is bounded or $p$ is a uniformly $r$-\gls{sapp} for $\subsetProbaMeasures$ and $\filtration$, then $E$ is a uniformly $r$-\gls{saep} for $\subsetProbaMeasures$ and $\filtration$.
\end{restatable}
\section{Numerical Simulation}
\label{sec:numerical-experiments}

\begin{figure}
    \centering
    \includegraphics[width=0.99\linewidth]{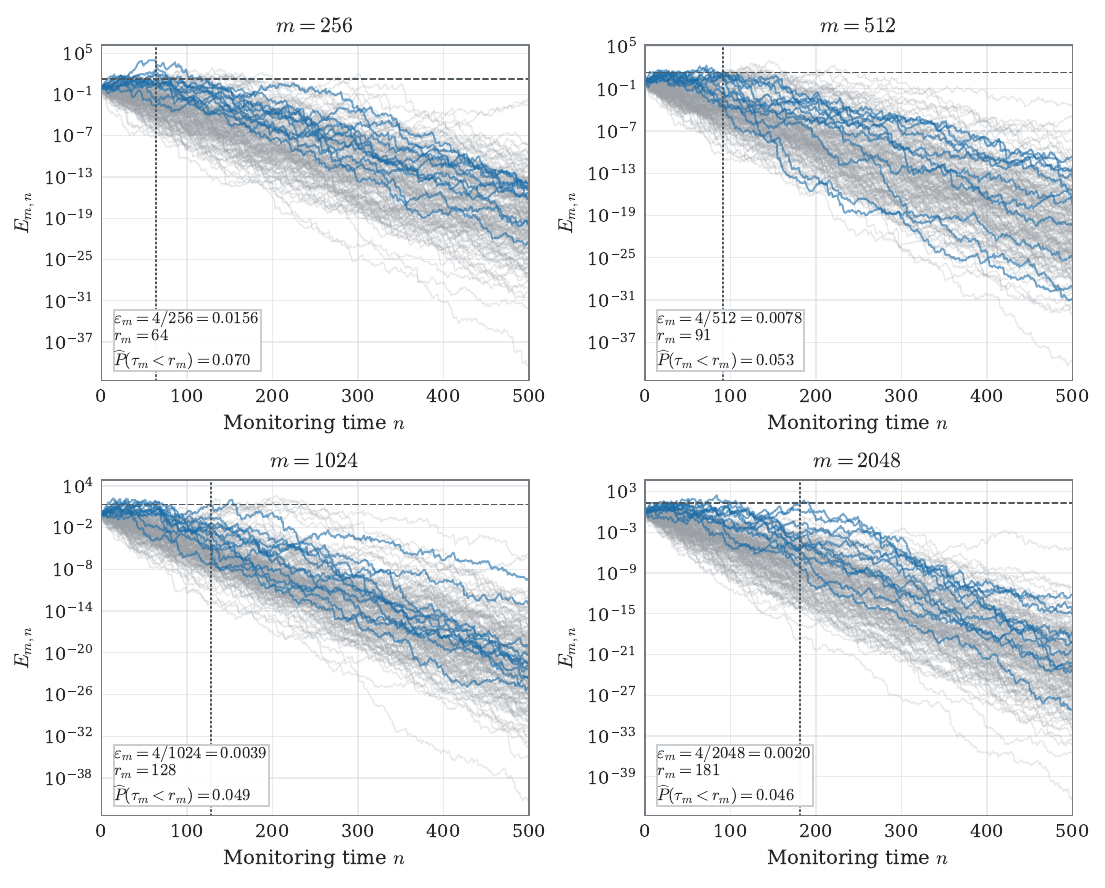}
    \caption{Trajectories of the
    simulated process described in \cref{sec:numerical-experiments}
    for a selection of different approximation indices $m \in \N$.
    The horizontal dashed line indicates the fixed threshold level
    $1/\alpha = 20$, the vertical dashed line
    indicates the time horizon sequence defined by 
    $r_m = 4 \lfloor m^{1/2} \rfloor$.
    The blue trajectories indicate a crossing of the threshold
    at a time $n \leq r_m$. 
    The estimated  probability $\widehat P [\tau_m  < r_m]$ is reported.
    }
    \label{fig:asymptotic-ville}
\end{figure}

\begin{figure}
    \centering
    \includegraphics[width=0.75\linewidth]{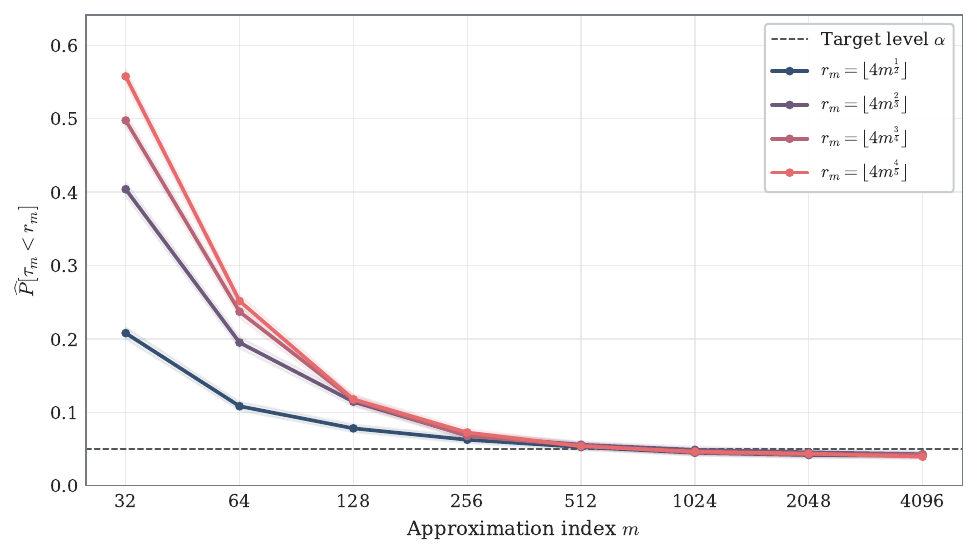}
    \caption{Empirical illustration of
    the asymptotic Ville's inequality given in 
    \cref{thm:asymptotic-ville}.
    The excursion probabilities $\widehat P [\tau_m  < r_m]$
    estimated from simulated trajectories are asymptotically
    bounded by $\alpha = 0.05$ for increasing $m$ and
    different choices of $r_m$ for which the process
    is an $r$-asymptotic e-process.}
    \label{fig:asymptotic-ville-convergence}
\end{figure}

We now simulate an $r$-\gls{saep} constructed from the
cumulative product as outlined in \cref{subsec:cumulative-product}.
In particular, we set $E_{m,0} = 1$ for all $m \in \N$
and
\begin{equation*}
    \forall n\in\Np,\quad E_{m,n} = \prod_{i = 1}^n e_{m,i}, 
\end{equation*}
where we define an array of 
asymptotic e-variables $(e_{m,i})_{(m,i) \in \N\times\Np}$ given by
\begin{equation*}
    \forall (m,i) \in \N\times\Np,\quad e_{m,i} := U_i + G_{m,i},
\end{equation*}
with independent random variables 
$
    U_i \sim \mathrm{Unif}\!\left[\tfrac12,\tfrac32\right],
$
and 
$
G_{m,i} \sim \mathcal{N}_{\geq b}(\varepsilon_m, \sigma^2),
$
that is, $G_{m,i}$ follow a normal distribution truncated from below
with mean $\varepsilon_m$, variance $\sigma^2$, and truncation threshold $b$.
The mean and variance parameters are those of the truncated distribution; see e.g. \citet{robert1995simulation} and \citet{barr1999truncated} for details regarding simulation and how the parameters of the original and truncated normal distributions are related.
We choose $\varepsilon_m = 4/m$, $\sigma^2 = 0.35^2$, and $b=-\frac12+10^{-6}$.
These choices guarantee that $e_{m,i} = U_i + G_{m,i} \geq 10^{-6}>0$ and $\E[e_{m,i}] = 1 + \varepsilon_m > 1$, making $e_{\bullet,i}$ a uniformly strongly asymptotic e-variable (since $\varepsilon_m\to 0$).
We now choose
$$
    r_m := \left\lfloor 4m^p \right\rfloor
$$
for some freely chosen $0<p<1$.
This choice matches the sufficient condition from
\cref{thm:cumulative product}. Indeed, for the filtration generated by the
process, independence gives
\begin{equation*}
    \E[e_{m,n+1}\mid\filtration_{m,n}]
    =
    \E[U_{n+1}] + \E[G_{m,n+1}]
    =
    1 + \varepsilon_m,
\end{equation*}
Moreover, $\E[e_{m,n}] = 1 + \varepsilon_m \to 1$,
so $(e_{m,n})_{m\in\N}$ is an asymptotic e-variable for each fixed $n$.
Finally, for the choices of $r_m$ above,
we have
\begin{equation*}
    r_m \cdot \varepsilon_m
    \to 0,
\end{equation*}
and therefore \cref{thm:cumulative product} implies that
$(E_{m,n})_{m,n\in\N}$ is an $r$-asymptotic e-process for this choice of
time horizon.

For the numerical simulations, 
we run $10,000$ independent trajectories each for 
approximation level
$m \in \{32,64,128,256,512,2048,4096\}$ for a 
total simulation
time of $n_{\text{end}} = 500$.
We fix the confidence level $\alpha = 0.05$ and
obtain empirical
estimates of the excursion probabilities
$\widehat P[\tau_m < r_m]$ with
$\tau_m = \inf\{ t \in \N \mid E_{m, t} \geq 1/\alpha\}$ by Monte Carlo integration over all simulated trajectories
for each scenario. 
We visualize a uniformly subsampled set of trajectories
for different choices of $m$ in \cref{fig:asymptotic-ville}.
In \cref{fig:asymptotic-ville-convergence}, we
empirically validate the behavior of
\cref{thm:asymptotic-ville} in the sense
that for increasing $m$, we confirm that
$\widehat P[\tau_m < r_m]$ is asymptotically
bounded by $\alpha$ for different choices of $p$.

\section{Conclusion}

We investigate two notions of asymptotic e-processes 
with respect to a fixed set of probability measures $\subsetProbaMeasures$
that naturally allow for constructions based on asymptotic e-variables. 
Contrary to existing notions, they rely on a horizon sequence $r$ characterizing the time-horizon of validity.
This sequence enables a genuinely asymptotic notion, where anytime-validity is obtained asymptotically as $r\to\infty$ but is not imposed at any finite approximation index.
We focus on properties and characterizations of asymptotic e-processes
that relate them to their non-asymptotic counterparts, 
as well as an explicit criterion based on asymptotic supermartingales.
We also provide several general methods for their construction by leveraging their abstract formulation, building on and generalizing existing ones.
In the context of hypothesis testing, our
version of Ville's inequality for asymptotic e-processes allows for
an asymptotic type-I error control when the set $\subsetProbaMeasures$
specifies the null hypothesis.
\par
Our work naturally opens up a range of related questions.
Most importantly, in order to obtain guarantees for sequential
testing with asymptotic e-processes,
their behavior under alternative hypotheses
and their corresponding type II errors need to be investigated.
Furthermore, from an application perspective, asymptotic 
\gls{savi} is still a recent framework and remains to be explored empirically.

\ifpreprint\section*{Acknowledgments}\else\begin{acks}[Acknowledgments]\fi
The authors thank Ben Chugg for pointing out
the recent preprint \citet{chuggPostHocLargeSampleStatistical2026}
which provided important context for the updated version of
this manuscript.
\ifpreprint\else\end{acks}\begin{funding}\fi
This work was partially funded by the German Federal Ministry of Research, Technology and Space as part of
the TRAICELL project (grant no. 03XP0636C).
\ifpreprint\else\end{funding}\fi

\ifpreprint\else
\begin{supplement}
    \stitle{Proofs}
    \sdescription{This section collects the proofs to the results presented in the main text.}
\end{supplement}
\begin{supplement}
    \stitle{Auxiliary results}
    \sdescription{This section contains auxiliary results used in the proofs}
\end{supplement}
\begin{supplement}
    \stitle{Characterization of uniformly asymptotic p-variables}
    \sdescription{This section shows that \cref{def:uniformly asymptotic p-variables} is indeed an equivalent characterization of uniformly (strongly) asymptotic p-variables.}
\end{supplement}
\fi

\ifpreprint
    \bibliographystyle{plainnat}
\else
    \bibliographystyle{imsart-nameyear}
\fi
\bibliography{references}

\clearpage
\renewcommand{\thesection}{S\arabic{section}}
\setcounter{section}{0}
\renewcommand{\theHsection}{S\arabic{section}} 
\crefalias{section}{appendix} % so that appendices are referred to as "Appendix A/B/..." instead of "Section A/B/..."
\section{Proofs}
This section collects the proofs to the results presented in the main text.
For the reader's convenience, we state again the formal result to be proven before stating each proof.
\ifpreprint\subsection{Proofs for Section 3}\fi
\characterizationUniformStoppingTime*
\begin{proof}[Proof of \cref{prop:characterization uniform in stopping time}]
    We begin with the characterization of $r$-\glspl{saep}.
    The converse implication is clear, and the direct implication follows from the fact that for all $m\in\N$, there exists $\tau_m\in\rBoundedStoppingTimes(r_m,\filtration_{m,\bullet},\subsetProbaMeasures)$ such that \begin{equation*}
        \sup_{P\in\subsetProbaMeasures}\E_P[E_{m,\tau_m}] + \frac{1}{m+1}\geq \sup_{\tau\in\rBoundedStoppingTimes(r_m,\filtration_{m,\bullet},\subsetProbaMeasures)}\sup_{P\in\subsetProbaMeasures}\E_P[E_{m,\tau}].
    \end{equation*}
    Taking the limit superior over $m$ concludes the proof of the direct implication.
    \par
    We move on to the characterization of $r$-\glspl{aep}.
    By definition, $E$ is an $r$-\gls{aep} if, and only if, $(E_{m,n}\pwMin t)_{m,n\in\N}$ is an $r$-\gls{saep} for all $t\in\Rnn$.
    Introduce then \begin{equation*}
        f_m:t\in\Rnn\mapsto \sup_{\tau\in\rBoundedStoppingTimes(r_m,\filtration_{m,\bullet},\subsetProbaMeasures)}\sup_{P\in\subsetProbaMeasures}\E_P[E_{m,\tau}\pwMin t].
    \end{equation*}
    Clearly, $f_m$ is nondecreasing for all $m\in\N$, and $E$ is an $r$-\gls{aep} if, and only if, $\limsup_{m\to\infty} f_m(t)\leq 1$ for all $t\in\Rnn$, by the result on $r$-\glspl{saep} that precedes.
    It follows from \cref{clry:rate of uniform boundedness 2} that this is equivalent to the existence of a sequence $(t_m)_{m\in\N}\subset\Rnn$ with $t_m\to\infty$ as $m\to\infty$ and such that $\limsup_{m\to\infty}f_m(t_m)\leq 1$, concluding the proof.
\end{proof}

\horizon*
\begin{proof}[Proof of \cref{prop:horizon of uniformly asymptotic e-processes}]
    We focus on the case of $r$-\glspl{aep}, as the case of $r$-\glspl{saep} follows similarly by \cref{prop:characterization uniform in stopping time}.
    The implication \labelcref{item:horizon of uniformly asymptotic e-processes:diverging horizon} $\Rightarrow$
    \labelcref{item:horizon of uniformly asymptotic e-processes:bounded horizon}
    follows immediately from \cref{prop:characterization uniform in stopping time}.
    Indeed, if $s = (s_m)_{m\in\N}$ is a bounded integer sequence, then there exists $m_0\in\N$ such that $r_m\geq s_m$ for all $m\geq m_0$, in which case it follows that \begin{equation*}
        \sup_{\tau\in\rBoundedStoppingTimes(s_m,\filtration_{m,\bullet},\subsetProbaMeasures)}\sup_{P\in\subsetProbaMeasures}\E_P[E_{m,\tau}\pwMin t_m] \leq \sup_{\tau\in\rBoundedStoppingTimes(r_m,\filtration_{m,\bullet},\subsetProbaMeasures)}\sup_{P\in\subsetProbaMeasures}\E_P[E_{m,\tau}\pwMin t_m],
    \end{equation*}
    where the sequence $t$ is given in \cref{prop:characterization uniform in stopping time}.
    The result follows by taking the limit superior over $m\geq m_0$.
    We now focus on the 
    implication \labelcref{item:horizon of uniformly asymptotic e-processes:bounded horizon} $\Rightarrow$
    \labelcref{item:horizon of uniformly asymptotic e-processes:diverging horizon}.
    For all $m,n\in\N$, and $t\in\Rnn$ define \begin{equation*}
        x_{m,n}(t) = \sup_{\tau\in\rBoundedStoppingTimes(n,\filtration_{m,\bullet},\subsetProbaMeasures)}\sup_{P\in\subsetProbaMeasures}\E_P[E_{m,\tau}\pwMin t].
    \end{equation*}
    It follows from \labelcref{item:horizon of uniformly asymptotic e-processes:bounded horizon} and the converse implication of \cref{prop:characterization uniform in stopping time} for the $r$-\gls{saep} $(E_{m,n}\pwMin t)_{m,n\in\N}$ that $\limsup_{m\to\infty} x_{m,n}(t)\leq 1$ for all fixed $n\in\N$ and $t\in\R$.
    This holds in particular with $t=n$.
    Consequently, \cref{lemma:rate of uniform boundedness} guarantees the existence of an integer sequence $(r_m)_{m\in\N}$ such that $r_m\to\infty$ and \begin{equation*}
        \limsup_{m\to\infty} x_{m,r_m}(r_m) = \limsup_{m\to\infty} \sup_{\tau_m\in\rBoundedStoppingTimes(r_m,\filtration_{m},\subsetProbaMeasures)} \sup_{P\in\subsetProbaMeasures}\E_P[E_{m,\tau_m}\pwMin r_m]\leq 1.
    \end{equation*}
    This shows that \cref{eq:characterization r-AEP} holds with $t_m := r_m$, which indeed goes to $\infty$ as $m\to\infty$, and concludes the proof by the converse implication of \cref{prop:characterization uniform in stopping time}.
\end{proof}

\characterizationSAEP*
The proof of \cref{thm:characterizations r-SAEP} relies on the following lemma to guarantee the integrability required in \labelcref{item:characterizations r-SAEP:npsms integrability}, whose proof repeats that of \cref{prop:characterization uniform in stopping time} by replacing formally $\rBoundedStoppingTimes$ with $\rBoundedStoppingTimes_\frm$.
\begin{lemma}\label{lemma:integrability of asymptotic e-process}
    Let $\filtration$ be a filtration sequence, $E = (E_{m,n})_{m,n\in\N}$ be a nonnegative process adapted to $\filtration$, and $r$ be an extended integer sequence.
    Then, \cref{thm:characterizations r-SAEP}\labelcref{item:characterizations r-SAEP:ii} holds if, and only if, \begin{equation*}
        \limsup_{m\to\infty}\sup_{\tau\in\rBoundedStoppingTimes_\frm(r_m,\filtration_{m,\bullet},\subsetProbaMeasures)}\sup_{P\in\subsetProbaMeasures}\E_P[E_{m,\tau}]\leq 1.
    \end{equation*}
    If one of these conditions holds (and therefore both hold), then there exists $m_0\in\N$ such that \begin{equation*}
        \sup_{\tau\in\rBoundedStoppingTimes_\frm(r_m,\filtration_{m,\bullet},\subsetProbaMeasures)}\sup_{P\in\subsetProbaMeasures}\E_P[E_{m,\tau}]<\infty, 
    \end{equation*}
    for all $m\geq m_0$.
    In particular, $E_{m,\tau}$ is $P$-integrable for all $\tau\in\rBoundedStoppingTimes_\frm(r_m,\filtration_{m,\bullet},\subsetProbaMeasures)$ and $P\in\subsetProbaMeasures$.
\end{lemma}
\begin{proof}[Proof of \cref{thm:characterizations r-SAEP}]
    The implication \labelcref{item:characterizations r-SAEP:i}$\implies$\labelcref{item:characterizations r-SAEP:ii} is trivial.
    We begin with \labelcref{item:characterizations r-SAEP:iii}$\implies$\labelcref{item:characterizations r-SAEP:i}.
    Let $\tau\in\rBoundedStoppingTimes(r,\filtration,\subsetProbaMeasures)$.
    For $P\in\subsetProbaMeasures$, introduce $L$
    as given by \labelcref{item:characterizations r-SAEP:iii}, and let $m_0\in\N$ be as in \labelcref{item:characterizations r-SAEP:npsms integrability}.
    An immediate consequence is that \begin{equation*}
        \forall m\geq m_0,~\E_{P}[E_{m,\tau_m}]\leq\E_P[L_{m,\tau_m}^P].
    \end{equation*}
    Now, it also holds that \begin{equation*}
        \forall m\geq m_0,~\E_P[L_{m,\tau_m}^P] \leq \E_P[L_{m,0}^P].
    \end{equation*}
    To see this, define the stopped process $U_{m,n}^P = L_{m,n\pwMin r_m}^P$ for all $m,n\in\N$ (note that $U^P$ is effectively stopped only if $r_m<\infty$; otherwise it coincides with $L^P$).
    Then, $U^P$ is clearly nonnegative, and $U_{m,\bullet}^P$ is a supermartingale for all $m\geq m_0$.
    Indeed, for all $n\in\N$ with $n\leq r_m-1$, $\E_P[U_{m,n+1}^P\mid\filtration_{m,n}]\leq U_{m,n}^P$ by \labelcref{item:characterizations r-SAEP:npsms supermartingale}.
    For $n\geq r_m$ (which is only possible when $r_m<\infty$), we have instead $\E_P[U_{m,n+1}^P\mid\filtration_{m,n}] = L_{m,r_m}^P = U_{m,n}^P$.
    Summarizing, $U^P_{m,\bullet}$ is a nonnegative supermartingale.
    We can thus apply \cref{thm:optional sampling nonnegative supermartingales} to $U_{m,\bullet}^P$ and obtain that \begin{equation*}
        \forall m\geq m_0,~\E_P[L_{m,\tau_m}^P] = \E_P[U_{m,\tau_m}^P] \leq \E_{P}[U_{m,0}^P] = \E_P[L_{m,0}^P],
    \end{equation*}
    where the first equality leverages the fact that $\tau_m\leq r_m$, $P$-\gls{as}, when $r_m$ is finite and $U_{m,n}^P = L_{m,n}^P$ for all $n\in\N$ when $r_m=\infty$, and thus $U_{m,\infty}^P = L_{m,\infty}^P$ as well.
    Putting it all together, we have shown that \begin{equation*}
        \E_{P}[E_{m,\tau_m}] \leq \E_P[L_{m,0}^P],
    \end{equation*}
    for all $m\geq m_0$ and $\tau=(\tau_m)_{m\in\N}\in\rBoundedStoppingTimes(r,\filtration,\subsetProbaMeasures)$.
    Taking the supremum over $P\in\subsetProbaMeasures$ and the limit superior in $m$ shows that $E$ is an $r$-asymptotic e-process by calibration of $L$ in \labelcref{item:characterizations r-SAEP:npsms calibration}.
    \par
    We now show \labelcref{item:characterizations r-SAEP:ii}$\implies$\labelcref{item:characterizations r-SAEP:iii}. Let $P\in\subsetProbaMeasures$ be fixed and $m_0\in\N$ be as given in \cref{lemma:integrability of asymptotic e-process}.
    For all $m\in\N$ and $n\in\N$, define \begin{equation*}
        \mathcal S_{m,n} = \{\tau\in\rBoundedStoppingTimes_\frm(r_m,\filtration_{m,\bullet},\subsetProbaMeasures)\mid n\leq\tau\}.
    \end{equation*}
    Here, the inequality $n\leq \tau$ is to be understood as holding pointwise on $\Omega$ for the stopping time $\tau$.
    It is immediate to check that $\mathcal S_{m,\bullet}$ is nonincreasing, and that $\mathcal S_{m,n}\neq\emptyset$ if, and only if, $n\leq r_m$.
    In addition, we introduce \begin{equation*}
        L_{m,n}^P = \begin{cases}
            \ess\sup_{\tau\in\mathcal S_{m,n}}\E_P[E_{m,\tau}\mid\filtration_{m,n}],&\text{if }m\geq m_0\text{ and }n\leq r_m,\\
            E_{m,n},&\text{otherwise},
        \end{cases}
    \end{equation*}
    where we emphasize that the conditional expectations are well-defined under the condition that $m\geq m_0$ and $n\leq r_m$ since then $E_{m,\tau}$ is $P$-integrable for all $\tau\in\mathcal S_{m,n}\subset\rBoundedStoppingTimes_\frm(r_m,\filtration_{m,\bullet},\subsetProbaMeasures)$.
    Here, $\ess\sup_{\lambda\in\Lambda}Y_\lambda$ denotes the essential supremum of a collection of real random variables $(Y_\lambda)_{\lambda\in\Lambda}$, where $\Lambda$ is an arbitrary index set.
    It is always a well-defined random variable, and is unique \gls{as}; we refer to \citet[Section A.2]{ramdasAdmissibleAnytimevalidSequential2022} and the references therein for more details.
    For all $m\in\N$, the process $L^P_{m,\bullet}$ is related to the
    so-called \emph{Snell envelope} 
    of $E_{m,\bullet}$; the Snell envelope is defined similarly without imposing that the stopping times in the essential supremum are less than $r_m$ \citep{ramdasAdmissibleAnytimevalidSequential2022}.
    We see that $L^P\geq E$, $P$-\gls{as}, since the constant stopping time equal to $n$ is always in $\mathcal S_{m,n}$ when the set is nonempty.
    Furthermore, $L^P$ is adapted to $\filtration$ by construction, since $E$ is adapted to $\filtration$.
    \par
    We now show \labelcref{item:characterizations r-SAEP:npsms integrability,item:characterizations r-SAEP:npsms supermartingale,item:characterizations r-SAEP:npsms calibration}, and let $m\geq m_0$ and $n\leq r_m$.
    To that end, we reason similarly as in the proof of \citet[Lemma 6]{ramdasAdmissibleAnytimevalidSequential2022}, and rely on Proposition 45 in the reference.
    It guarantees that one can find a sequence $(\tau_{k,n})_{k\in\N}\subset\mathcal S_{m,n}$ such that $(\E_P[E_{m,\tau_{k,n}}\mid\filtration_{m,n}])_{k\in\N}$ is nondecreasing and 
    \begin{equation*}
        \lim_{k\to\infty} \E_P[E_{m,\tau_{k,n}}\mid\filtration_{m,n}] = L^P_{m,n}, \quad P\text{-\gls{as}},
    \end{equation*}
    provided that the set $\{\E_P[E_{m,\tau}\mid\filtration_{m,n}]\mid \tau\in\mathcal S_{m,n}\}$ is closed under maxima.
    To show this last point, let $\tau$ and $\tau^\prime$ be in $\mathcal S_{m,n}$, and define the event \begin{equation*}
        A = \{\E_P[E_{m,\tau}\mid\filtration_{m,n}] > \E_P[E_{m,\tau^\prime}\mid\filtration_{m,n}]\}.
    \end{equation*}
    Define the variable $\sigma = \tau\indicator_A + \tau^\prime\indicator_{A^C}$.
    One immediately sees that $\sigma\geq n$ and $\sigma<\infty$ pointwise, by assumption on $\tau$ and $\tau^\prime$.
    Furthermore, $\sigma\leq r_m$, $Q$-\gls{as} for all $Q\in\subsetProbaMeasures$, also by assumption $\tau$ and $\tau^\prime$.
    Finally, $\sigma$ is an $\filtration_{m,\bullet}$-stopping time.
    Indeed, for any $i<n$, the event $\{\sigma = i\}$ is empty, and thus is in $\filtration_{m,i}$.
    For $i\geq n$, the event $A$ is in $\filtration_{m,n}$, and thus in $\filtration_{m,i}$.
    The conclusion follows from the fact that $\tau$ and $\tau^\prime$ are stopping times.
    Consequently, $\sigma\in\mathcal S_{m,n}$.
    But it also holds that \begin{align*}
        \E_P[E_{m,\sigma}\mid\filtration_{m,n}]
            &= \indicator_A\E_P[E_{m,\tau}\mid\filtration_{m,n}] + \indicator_{A^C}\E_P[E_{m,\tau^\prime}\mid\filtration_{m,n}] \\
            &= \max\{\E_P[E_{m,\tau}\mid\filtration_{m,n}],\E_P[E_{m,\tau^\prime}\mid\filtration_{m,n}]\},
    \end{align*}
    showing closedness under maxima.
    We can thus apply Proposition 45 in \citet{ramdasAdmissibleAnytimevalidSequential2022} and find a sequence $(\tau_{k,n})_{k\in\N}\subset\mathcal S_{m,n}$ as announced.
    A first consequence is that, by the monotone convergence theorem and the tower property, 
    \begin{align*}
        \E_P[L^P_{m,n}] 
            &= \E_P\left[\lim_{k\to\infty}\E_P[E_{m,\tau_{k,n}}\mid\filtration_{m,n}]\right]\\
            &= \lim_{k\to\infty}\E_P\left[\E_P[E_{m,\tau_{k,n}}\mid\filtration_{m,n}]\right]\\
            &= \lim_{k\to\infty}\E_P[E_{m,\tau_{k,n}}]\\
            &\leq \sup_{\sigma\in\rBoundedStoppingTimes_\frm(r_m,\filtration_{m,\bullet},\subsetProbaMeasures)}\sup_{P\in\subsetProbaMeasures}\E_P[E_{m,\sigma}].
    \end{align*}
    It follows from \cref{lemma:integrability of asymptotic e-process} that this last bound is finite, since $m\geq m_0$, and thus $L_{m,n}^P$ is indeed integrable, showing \labelcref{item:characterizations r-SAEP:npsms integrability}.
    Furthermore, this inequality holds for all $m\geq m_0$, $n\in\integers{r_m}$, and $P\in\subsetProbaMeasures$; in particular, it holds with $n=0$ for all $m\geq m_0$ and $P\in\subsetProbaMeasures$.
    Taking the supremum over $P\in\subsetProbaMeasures$ and the limit superior in $m$ and leveraging \cref{lemma:integrability of asymptotic e-process} together with the assumption \labelcref{item:characterizations r-SAEP:ii} shows that \labelcref{item:characterizations r-SAEP:npsms calibration} holds.
    A second consequence follows in the case $n+1\leq r_m$; then, $L_{m,n+1}^P$ is $P$-integrable, and the monotone convergence theorem for conditional expectations and the tower property thus yield
    \begin{align*}
        \E_P[L^P_{m,n+1}\mid\filtration_{m,n}] 
            &= \E_P\left[\lim_{k\to\infty}\E_P[E_{m,\tau_{k,n+1}}\mid\filtration_{m,n+1}]\,\middle|\,\filtration_{m,n}\right]\\
            &= \lim_{k\to\infty}\E_P\left[\E_P[E_{m,\tau_{k,n+1}}\mid\filtration_{m,n+1}]\,\middle|\,\filtration_{m,n}\right]\\
            &= \lim_{k\to\infty}\E_P[E_{m,\tau_{k,n+1}}\mid\filtration_{m,n}]\\
            &\leq L^P_{m,n},\quad P\text{-\gls{as}},
    \end{align*}
    where the last step relies on the fact that $L^P_{m,n}$ is defined via the essential supremum (since $n<n+1\leq r_m$) and that $\tau_{k,n+1}\in\mathcal{S}_{m,n+1}\subset\mathcal{S}_{m,n}$ for all $k\in\N$.
    This holds for all $m\geq m_0$ and $n\in\integers{r_m-1}$, showing that \labelcref{item:characterizations r-SAEP:npsms supermartingale} holds and concluding the proof.
\end{proof}

\characterizationAEP*
\begin{proof}[Proof of \cref{clry:characterizations r-AEP}]
    By definition, $E$ is an $r$-\gls{aep} if, and only if, the process $E\pwMin t := (E_{m,n}\pwMin t)_{m,n\in\N}$ is an $r$-\gls{saep} for all $t\in\Rnn$.
    It results from applying \cref{thm:characterizations r-SAEP} to $E\pwMin t$ for every fixed $t\in\Rnn$ that this last condition is equivalent to $(E_{m,\tau_m}\pwMin t)_{m\in\N}$ being a uniformly strongly asymptotic e-variable for all $\tau=(\tau_m)_{m\in\N}\in\rBoundedStoppingTimes_\frm(r,\filtration,\subsetProbaMeasures)$ and $t\in\Rnn$, which is precisely \labelcref{item:characterizations r-AEP:ii}.
\end{proof}

\aepWithConvergence*
\begin{proof}[Proof of \cref{thm:asymptotic e-processes converge to e-processes}]
    The implication \labelcref{item:asymptotic e-processes converge to e-processes:r-SAEP}$\implies$\labelcref{item:asymptotic e-processes converge to e-processes:r-AEP} is trivial.
    We begin with the implication \labelcref{item:asymptotic e-processes converge to e-processes:r-AEP}$\implies$\labelcref{item:asymptotic e-processes converge to e-processes:e-process}, and assume the existence of an extended integer sequence $r=(r_m)_{m\in\N}$ with $r_m\to\infty$ as $m\to\infty$ such that $E$ is an $r$-\gls{aep} for $\filtration$ and $\subsetProbaMeasures$.
    First, we can assume without loss of generality that \begin{equation}\label{eq:limit expectation of supremum}
        \lim_{m\to\infty}\sup_{P\in\subsetProbaMeasures} (r_m+1)\cdot \E_P\left[\sup_{n\leq r_m}\lvert E_{m,n} - F_n\rvert\right] = 0.
    \end{equation}
    Indeed, define for all $m,n\in\N$ the quantity \begin{equation*}
        x_{m,n} = \sup_{P\in\subsetProbaMeasures}(n+1)\cdot\E_P\left[\sup_{k\leq n}\lvert E_{m,k} - F_k\rvert\right].
    \end{equation*}
    It holds that \begin{align*}
        x_{m,n} &\leq \sup_{P\in\subsetProbaMeasures}(n+1)\cdot\E_P\left[\sum_{k=0}^n\lvert E_{m,k} - F_k\rvert\right]\\
            &\leq\sum_{k=0}^n\sup_{P\in\subsetProbaMeasures} (n+1)\cdot\E_P\left[\lvert E_{m,k} - F_k\rvert\right] =: y_{m,n}.
    \end{align*}
    By assumption, $y_{m,n}\to0$ as $m\to\infty$ for all fixed $n\in\N$ (as the sum of finitely many vanishing terms).
    As a result, there exists by \cref{lemma:rate of uniform boundedness} an integer sequence $s=(s_m)_{m\in\N}$ such that $y_{m,s_m}\to 0$ as $m\to\infty$ and $s_m\to\infty$ as $m\to\infty$.
    The above bound then guarantees $x_{m,s_m}\to 0$ as $m\to\infty$.
    Furthermore, defining $t_m = r_m\pwMin s_m$, one straightforwardly verifies that $x_{m,t_m}\to0$ as $m\to\infty$, as $x_{m,\bullet}$ is nondecreasing, and this last condition is precisely \cref{eq:limit expectation of supremum} where $r$ is formally replaced by $t$.
    Crucially, $t\leq r$, and thus $E$ is also a $t$-\gls{aep}.
    Also, $t_m\to\infty$ as $m\to\infty$.
    In other words, the sequence $t$ satisfies the assumption, and is such that \cref{eq:limit expectation of supremum} is satisfied, which guarantees that the condition can indeed be assumed without loss of generality.
    \par
    Let now $P\in\subsetProbaMeasures$ and $\tau\in\rBoundedStoppingTimes_\frm(\infty,\filtration_{\infty,\bullet},\subsetProbaMeasures)$.
    We construct a sequence $\theta = (\theta_m)_{m\in\N}\in\rBoundedStoppingTimes_\frm(r,\filtration,\subsetProbaMeasures)$ that converges to $\tau$, \gls{as}
    For this, define for all $m,n\in\N$ \begin{align*}
        Z_{m,n} &= \E_P[\indicator_{\{\tau\leq n\}}\mid\filtration_{m,n}],\\
        \text{and}\quad\bar\theta_m &= \inf\left\{k\in\N\mid Z_{m,k}>\frac12\right\}.
    \end{align*}
    The specific choice of $1/2$ for the threshold is unimportant here; the construction works with any threshold in $(0,1)$.
    It is clear that $\bar\theta_m$ is an $\filtration_{m,\bullet}$-stopping time for all $m\in\N$.
    Furthermore, by Lévy's upward theorem \citep[Theorem 14.2]{williams1991martingales}, it follows that for all $n\in\N$ \begin{equation*}
        \lim_{m\to\infty}Z_{m,n} = \lim_{m\to\infty}\E_P[\indicator_{\{\tau\leq n\}}\mid\filtration_{m,n}] = \E_P[\indicator_{\{\tau\leq n\}}\mid\filtration_{\infty,n}] = \indicator_{\{\tau\leq n\}},\quad P\text{-\gls{as}}
    \end{equation*}
    As a result, we have a set $A\in\sigAlg$ with $P[A]=1$ and such that the above convergence holds pointwise on $A$.
    Let $n_0\in\N$, and take $\omega\in A\cap\{\tau = n_0\}$.
    For any $n\geq n_0$, it holds that $Z_{m,n}(\omega)\to \indicator_{\{\tau\leq n\}}(\omega) = 1$ as $m\to\infty$.
    In particular, $Z_{m,n}>1/2$ for $m$ sufficiently large, and thus $\bar\theta_m(\omega)\leq n$ eventually.
    Furthermore, for all $n<n_0$, it holds that $Z_{m,n}(\omega)\to \indicator_{\{\tau\leq n\}}(\omega) = 0$ as $m\to\infty$, and thus  $Z_{m,n}(\omega)<1/2$ for $m$ sufficiently large, showing $\bar\theta_m(\omega)> n$ eventually.
    Combining these two results shows that $\bar\theta_m(\omega)= n_0$ for $m$ sufficiently large.
    Summarizing, we have shown that for all $\omega\in A$, $\bar\theta_m(\omega) \to\tau(\omega)$ as $m\to\infty$, which implies in turn that $\bar\theta_m\to\tau$ as $m\to\infty$, $P$-\gls{as}
    We now define $\theta_m = \bar\theta_m\pwMin r_m$.
    Since $r_m\to\infty$, one easily concludes that $\theta_m\to\tau$ as $m\to\infty$, $P$-\gls{as}
    Furthermore, $\theta\in\rBoundedStoppingTimes_\frm(r,\filtration,\subsetProbaMeasures)$ by construction.
    \par
    Now, observe that, for all $m,n\in\N$, and $t\in\Rnn$, \begin{equation*}
        F_n\pwMin t \leq (F_n - E_{m,n})^+ + E_{m,n}\pwMin t.
    \end{equation*}
    Indeed, this results from the general inequalities $a\leq b + (a-b)^+$, $(c+d)\pwMin t\leq c\pwMin t + d\pwMin t$, and $c\pwMin t\leq c$, where $a$, $b$, $c$ , $d$, and $t$ are in $\Rnn$, and which we apply with $a = F_n$, $b = c = E_{m,n}$, and $d = (F_n - E_{m,n})^+$.
    As a result, for all $P\in\subsetProbaMeasures$, $m\in\N$, and $t\in\Rnn$, it holds that\begin{align*}
        \E_P[F_{\theta_m}\pwMin t] \leq \E_P[(F_{\theta_m} - E_{m,\theta_m})^+] + \E_P[E_{m,\theta_m}\pwMin t].
    \end{align*}
    We bound the first term as \begin{align*}
        \E_P[(F_{\theta_m} - E_{m,\theta_m})^+]
            &\leq \E_P[\lvert F_{\theta_m} - E_{m,\theta_m}\rvert]\\
            &= \E_P\left[\sum_{n=0}^{r_m}\indicator_{\{\theta_m=n\}}\lvert F_{n} - E_{m,n}\rvert\right]\\
            &\leq \E_P\left[\sup_{n\leq r_m}\lvert F_{n} - E_{m,n}\rvert\sum_{n=0}^{r_m}\indicator_{\{\theta_m=n\}}\right]\\
            &\leq (r_m+1)\E_P\left[\sup_{n\leq r_m}\lvert F_{n} - E_{m,n}\rvert\right].
    \end{align*}
    By \cref{eq:limit expectation of supremum}, this quantity vanishes as $m\to\infty$.
    Consequently, leveraging the property that $E$ is an $r$-\gls{aep} and the fact that $\theta\in\rBoundedStoppingTimes_\frm(r,\filtration,\subsetProbaMeasures)$,\begin{equation*}
        \limsup_{m\to\infty}\E_P[F_{\theta_m}\pwMin t] \leq \limsup_{m\to\infty} \E_P[(F_{\theta_m} - E_{m,\theta_m})^+] + \limsup_{m\to\infty}\E_P[E_{m,\theta_m}\pwMin t] \leq 1.
    \end{equation*}
    Next, it follows from the $P$-\gls{as} convergence $\theta_m\to\tau$ as $m\to\infty$ that we also have $F_{\theta_m}\to F_\tau$ as $m\to\infty$, $P$-\gls{as} (since $\tau$ is finite), and thus $F_{\theta_m}\pwMin t\to F_\tau\pwMin t$ as $m\to\infty$, $P$-\gls{as}.
    Consequently, by Fatou's lemma \begin{equation*}
        \E_P[F_\tau\pwMin t] \leq\liminf_{m\to\infty}\E_P[F_{\theta_m}\pwMin t] \leq \limsup_{m\to\infty}\E_P[F_{\theta_m}\pwMin t]\leq 1.
    \end{equation*}
    This holds for all $t\in\Rnn$. 
    By the monotone convergence theorem, $\E_P[F_\tau] = \lim_{n\to\infty}\E_P[F_\tau\pwMin n]$, and thus $\E_P[F_\tau]\leq 1$.
    We can now apply \citet[Lemma 6]{ramdasAdmissibleAnytimevalidSequential2022} to show that the property extends to arbitrary stopping times (not necessarily \gls{as} finite), showing that $F$ is indeed an e-process.
    \par
    We now show \labelcref{item:asymptotic e-processes converge to e-processes:e-process}$\implies$\labelcref{item:asymptotic e-processes converge to e-processes:r-SAEP}, and assume that $F$ is an e-process.
    For all extended integer sequences $r$, all $P\in\subsetProbaMeasures$, all $\tau\in\rBoundedStoppingTimes_\frm(r,\filtration,\subsetProbaMeasures)$, and all $m\in\N$, we have, postponing shortly showing integrability: \begin{equation*}
        \E_P[E_{m,\tau_m}] = \E_P[E_{m,\tau_m} - F_{\tau_m}] + \E_P[F_{\tau_m}] \leq 1 + \E_P[E_{m,\tau_m} - F_{\tau_m}].
    \end{equation*}
    As a result, it suffices to choose $r$ such that $\limsup_{m\to\infty}\sup_{P\in\subsetProbaMeasures} \E_P[E_{m,\tau_m} - F_{\tau_m}] = 0$ for $E$ to be an $r$-\gls{saep}, implicitly leveraging \cref{thm:characterizations r-SAEP} to conclude based on finite stopping times.
    Yet, for any extended integer sequence $r$, \begin{align*}
        \sup_{P\in\subsetProbaMeasures}\E_P[E_{m,\tau_m} - F_{\tau_m}] 
            &\leq \sup_{P\in\subsetProbaMeasures}\E_P[\lvert E_{m,\tau_m} - F_{\tau_m}\rvert]\\
            &= \sup_{P\in\subsetProbaMeasures}\E_P\left[\left\lvert\sum_{n=0}^{r_m} (E_{m,n} - F_n)\cdot\indicator_{\{\tau_m = n\}}\right\rvert\right]\\
            &\leq \sum_{n=0}^{r_m}\sup_{P\in\subsetProbaMeasures}\E_P[\lvert E_{m,n} - F_n\rvert],
    \end{align*}
    where the last step leverages the triangle inequality.
    This shows that any extended integer sequence satisfying the condition of the theorem implies that $E$ is an $r$-\gls{saep}, and guarantees that $E_{m,\tau_m}$ is integrable for $m\geq m_0$ where $m_0\in\N$ is some integer independent of $\tau_m$.
    Furthermore, we can show that such a diverging integer sequence necessarily exists by applying \cref{lemma:rate of uniform boundedness} to the sequence $x_{m,n} = \sum_{i=0}^n\sup_{P\in\subsetProbaMeasures}\E_P[\lvert E_{m,i} - F_i\rvert]$, $m,n\in\N$, as it satisfies $x_{m,n}\to 0$ as $m\to\infty$ for all fixed $n\in\N$ by assumption, showing the implication and the claims involving \cref{eq:r of converging asymptotic e-process} and concluding the proof.
\end{proof}
\ifpreprint\subsection{Proofs for Section 4}\fi
\asmImpliesAsp*
\begin{proof}[Proof of \cref{thm:asm implies asp}]
    Let $S:= (S_n)_{n\in\N}$ be a supermartingale for $\subsetProbaMeasures$ with respect to $\filtration_{\infty,\bullet}$ such that $E$ converges to $S$ in $L_1$ uniformly in $\subsetProbaMeasures$.
    Fix $n\in\N$.
    We decompose $\eta_{m,n}$ as follows, for all $m,n\in\N$ and $P\in\subsetProbaMeasures$, recalling that $E_{m,n+1}$ is $P$-integrable by \cref{def:l1 convergence of bi-indexed process}:\begin{align*}
        \eta_{m,n} &= \E_P[E_{m,n+1} - S_{n+1} \mid\filtration_{m,n}] \\
            &+ \E_P[S_{n+1} - S_n\mid \filtration_{m,n}] \\
            &+ \E_P[S_n\mid \filtration_{m,n}]  - S_n\\
            &+ S_n - E_{m,n},\quad P\text{-\gls{as}}
    \end{align*}
    Define each of these terms as $A_m = \E_P[E_{m,n+1} - S_{n+1} \mid\filtration_{m,n}]$, $B_m = \E_P[S_{n+1} - S_{n}\mid \filtration_{m,n}]$, $C_m = \E_P[S_{n}\mid \filtration_{m,n}]  - S_{n}$, and $D_m = S_{n} - E_{m,n}$.
    We handle $B_m$ immediately as follows:\begin{align*}
        B_m &= \E_P\left[\E_P[S_{n+1} - S_{n}\mid \filtration_{\infty,n}]~\middle|~\filtration_{m,n}\right]\\
            &= \E_P\left[\E_P[S_{n+1}\mid \filtration_{\infty,n}] - S_n~\middle|~\filtration_{m,n}\right]\\
            &\leq \E_P\left[0\mid \filtration_{m,n}\right]\\
            &\leq 0,\quad P\text{-\gls{as}}
    \end{align*}
    The first three steps come respectively from the facts that $\filtration_{m,n}\subset\filtration_{\infty,n}$, that $S_n$ is $\filtration_{\infty,n}$-measurable, and that $\E_P[S_{n+1}\mid \filtration_{\infty,n}] - S_n\leq 0$, $P$-\gls{as}, since $S$ is an $\filtration_{\infty,\bullet}$-supermartingale.
    As a result, $\eta_{m,n}\leq A_m + C_m + D_m$, $P$-\gls{as}, and thus\begin{equation*}
        \sup_{P\in\subsetProbaMeasures}\E_P[\eta_{m,n}^+] \leq \sup_{P\in\subsetProbaMeasures}\norm{A_m}_P + \sup_{P\in\subsetProbaMeasures}\norm{C_m}_P + \sup_{P\in\subsetProbaMeasures}\norm{D_m}_P.
    \end{equation*}
    We handle each term separately.
    First, it follows from the tower property and the contraction property of conditional expectations that, for all $P\in\subsetProbaMeasures$, \begin{align*}
        \norm{A_m}_P = \E_P\left[\left\lvert\,\E_P[E_{m,n+1} - S_{n+1}\mid\filtration_{m,n}]\right\rvert\right]
        \leq\E_P[\lvert E_{m,n+1} - S_{n+1}\rvert],
    \end{align*}
    which goes to $0$ uniformly in $P\in\subsetProbaMeasures$ as $m\to\infty$ by \cref{def:l1 convergence of bi-indexed process}.
    The term $\sup_{P\in\subsetProbaMeasures}\norm{D_m}_P$ also vanishes asymptotically for the same reason.
    For the term $C_m$, we have, for all $m\in\N$,
    \begin{align*}
        \sup_{P\in\subsetProbaMeasures}\norm{C_m}_P
            &= \sup_{P\in\subsetProbaMeasures}\norm{\E_P[S_n\mid \filtration_{m,n}] - S_n}_P\\
            &= \sup_{P\in\subsetProbaMeasures}\norm{\E_P[S_n\mid \filtration_{m,n}] - \E_P[S_n\mid\filtration_{\infty,n}]}_P\\
            &\leq \sup_{P\in\subsetProbaMeasures}\norm{\E_P[S_n\mid \filtration_{m,n}] - \E_P[E_{m,n}\mid \filtration_{m,n}]}_P \\
            &\quad+ \sup_{P\in\subsetProbaMeasures}\norm{\E_P[E_{m,n}\mid \filtration_{m,n}] - \E_P[E_{m,n}\mid \filtration_{\infty,n}]}_P \\
            &\quad+ \sup_{P\in\subsetProbaMeasures}\norm{\E_P[E_{m,n}\mid \filtration_{\infty,n}] - \E_P[S_n\mid\filtration_{\infty,n}]}_P\\
            &\leq 2\sup_{P\in\subsetProbaMeasures}\norm{S_n - E_{m,n}}_P + \sup_{P\in\subsetProbaMeasures}\norm{\E_P[E_{m,n}\mid \filtration_{m,n}] - \E_P[E_{m,n}\mid \filtration_{\infty,n}]}_P.
    \end{align*}
    Indeed, the above relations follow successively from the definition of $C_m$, the $\filtration_{\infty,n}$-measurability of $S_n$, the triangle inequality for $\sup_{P\in\subsetProbaMeasures}\norm{\cdot}_P$, and the contraction property of conditional expectations on $L_1(\Omega,\sigAlg,P)$ for all $P\in\subsetProbaMeasures$.
    Now, crucially, $E_{m,n}$ is $\filtration_{m,n}$-measurable, and thus $\filtration_{\infty,n}$-measurable as well. Therefore, $\E_P[E_{m,n}\mid \filtration_{m,n}] = \E_P[E_{m,n}\mid \filtration_{\infty,n}] = E_{m,n}$, $P$-\gls{as}, and the second term is $0$.
    In addition, the first term vanishes as $m\to\infty$ by \cref{def:l1 convergence of bi-indexed process}.
    This shows that $\sup_{P\in\subsetProbaMeasures}\E_P[\eta_{m,n}^+]$ goes to $0$ as $m\to\infty$ for all $n\in\N$, and concludes the proof.
\end{proof}

\rSummableIncrements*
\begin{proof}[Proof of \cref{lemma:r-summable increments}]
    For all $m,n\in\N$, define 
    \begin{equation*}
        x_{m,n} = \sup_{P\in\subsetProbaMeasures}\E_P\left[\sum_{k=0}^{n-1}\eta_{m,k}^+\right].
    \end{equation*}
    By assumption, $\lim_{m\to\infty}x_{m,n} = 0$ for all $n\in\N$.
    The existence of $r$ such that \cref{eq:r-summable increments} holds follows from \cref{lemma:rate of uniform boundedness}.
\end{proof}

\aspImpliesSaep*
\begin{proof}[Proof of \cref{thm:asp implies asymptotic e-process}]
    For all $P\in\subsetProbaMeasures$ and $m\in\N$, the Doob decomposition of $E_{m,\bullet}$ exists on $(\Omega,\sigAlg,P)$, since the process is $P$-integrable \citep[Theorem 12.11]{williams1991martingales}.
    Fix then $P\in\subsetProbaMeasures$, and write this decomposition as $E_{m,n} = M_{m,n} + A_{m,n}$, where $M_{m,\bullet}$ is a martingale and \begin{align*}
        M_{m,n} &= E_{m,0} + \sum_{k=1}^{n} E_{m,k} - \E_P[E_{m,k}\mid\filtration_{m,k-1}],\quad\text{and}\\ 
        A_{m,n} &= \sum_{k=1}^n\E_P[E_{m,k}\mid \filtration_{m,k-1}] - E_{m,k-1},
    \end{align*}
    for all $m,n\in\N$.
    Here, we omitted again the dependence of the variables on $P$ for readability.
    Furthermore, the following upper bound holds: 
    \begin{equation*}
        A_{m,n}\leq \sum_{k=1}^n (\E_P[E_{m,k}\mid \filtration_{m,k-1}] - E_{m,k-1})^+ = \sum_{k=1}^n \eta_{m,k-1}^+ 
        =: \Delta_{m,n},\quad P\text{-\gls{as}},
    \end{equation*}
    and thus $E_{m,n} \leq M_{m,n} + \Delta_{m,n}$, $P$-\gls{as} and for all $m,n\in\N$.
    \par
    Let then $r$ be an extended integer sequence that satisfies \cref{eq:r-summable increments} and $\tau=(\tau_m)_{m\in\N}\in\rBoundedStoppingTimes_\frm(r,\filtration,\subsetProbaMeasures)$.
    It follows from what precedes that \begin{equation*}
        \forall m\in\N,\quad\E_P[E_{m,\tau_m}]\leq \E_P[M_{m,\tau_m}] + \E_P[\Delta_{m,\tau_m}].
    \end{equation*}
    We bound each term in the \gls{rhs} separately, and begin with the second one.
    Specifically, by monotonicity of $\Delta_{m,\bullet}$ and the fact that $\tau_m\leq r_m$, for $m\in\N$ and $P$-\gls{as}, it holds that \begin{equation*}
        \forall m\in\N,\quad\E_P[\Delta_{m,\tau_m}]\leq \E_P\left[\Delta_{m,r_m}\right]\leq\sup_{Q\in\subsetProbaMeasures}\E_Q\left[\Delta_{m,r_m}\right].
    \end{equation*}
    It follows from \cref{eq:r-summable increments} that there exists $m_0\in\N$ independent of $P\in\subsetProbaMeasures$ such that this last term is finite for all $m\geq m_0$. We thus move on to the first term, to which we apply
   the optional stopping theorem for supermartingales with integrable lower bound (\cref{clry:optional sampling lower bounded supermartingales}).
    Specifically, notice that for all $m\in\N$, nonnegativity of $E$ entails that the martingale $M_{m,\bullet}$ satisfies \begin{equation*}
        \forall n\in\integers{r_m},\quad M_{m,n} \geq -\Delta_{m,n} \geq -\Delta_{m,r_m}.
    \end{equation*}
    Furthermore, the \gls{rhs} is $P$-integrable for all $m\geq m_0$.
    Since a martingale is also a supermartingale, we can apply \cref{clry:optional sampling lower bounded supermartingales}, which yields \begin{equation*}
        \forall m\geq m_0,\quad\E_P[M_{m,\tau_m}]\leq \E_P[M_{m,0}] = \E_P[E_{m,0}].
    \end{equation*}
    Putting everything together, we have shown that \begin{equation*}
        \forall m\geq m_0,\quad\E_P[E_{m,\tau_m}]\leq \E_P[E_{m,0}] + \E_{P}[\Delta_{m,r_m}].
    \end{equation*}
    Since $m_0$ is independent of $P$, taking the supremum over $P\in\subsetProbaMeasures$ and the limit superior over $m\geq m_0$ shows that \begin{align*}
        \limsup_{m\to\infty} \sup_{P\in\subsetProbaMeasures}\E_P[E_{m,\tau_m}]
            \leq \limsup_{m\to\infty}\sup_{P\in\subsetProbaMeasures}\E_P[E_{m,0}] + \limsup_{m\to\infty}\sup_{P\in\subsetProbaMeasures}\E_{P}[\Delta_{m,r_m}]
            \leq 1 + 0
            = 1,
    \end{align*}
    where the second inequality comes from the calibration condition \cref{eq:calibration} and the assumption \cref{eq:r-summable increments}.
    This is true for all $\tau\in\rBoundedStoppingTimes_\frm(r,\filtration,\subsetProbaMeasures)$; we then deduce from \cref{thm:characterizations r-SAEP} that $E$ is indeed an $r$-\gls{saep}.
\end{proof}

\aspAndConvergenceImplyAsm*
\begin{proof}[Proof of \cref{thm:asp and convergence implies asm}]
    Let $n\in\N$ and $P\in\subsetProbaMeasures$; we show that $\E_P[S_{n+1}\mid\filtration_{\infty,n}]\leq S_n$, $P$-\gls{as}
    For all $m\in\N$, \begin{align*}
        \eta_{m,n} =~& \E_P[E_{m,n+1}-S_{n+1}\mid\filtration_{m,n}] \\
            &+ \E_P[S_{n+1}\mid\filtration_{m,n}] - \E_P[S_{n+1}\mid\filtration_{\infty,n}]\\
            &+ \E_P[S_{n+1}\mid\filtration_{\infty,n}] - S_{n}\\
            &+ S_n - E_{m,n},\quad P\text{-\gls{as}}
    \end{align*}
    Introducing $\sigma_{n} = \E_P[S_{n+1}\mid\filtration_{\infty,n}] - S_n$, it follows that \begin{align*}
        \eta_{m,n} - \sigma_{n} =~&\E_P[E_{m,n+1}-S_{n+1}\mid\filtration_{m,n}] \\
            &+ \E_P[S_{n+1}\mid\filtration_{m,n}] - \E_P[S_{n+1}\mid\filtration_{\infty,n}] \\
            &+ S_n - E_{m,n},\quad P\text{-\gls{as}}
    \end{align*}
    We show that 
    the \gls{rhs} goes to $0$ in $L_1(\Omega,\sigAlg,P)$ as $m\to\infty$.
    Indeed, the first term vanishes by the contraction property of conditional expectations and $L_1$ convergence of $E_{\bullet,n+1}$ to $S_{n+1}$.
    The last term also vanishes by $L_1$ convergence of $E_{\bullet,n}$ to $S_{n}$.
    Finally, the second term 
    is handled by Lévy's upward theorem \citep[Theorem 14.2]{williams1991martingales}
    , which precisely guarantees that \begin{equation*}
        \lim_{m\to\infty}\E_P[S_{n+1}\mid\filtration_{m,n}] = \E_P[S_{n+1}\mid\filtration_{\infty,n}],
    \end{equation*}
    both in $L_1(\Omega,\sigAlg,P)$ and $P$-\gls{as}
    As a result, $\eta_{m,n}\to\sigma_n$ as $m\to\infty$ in $L_1(\Omega,\sigAlg,P)$.
    It follows immediately that \begin{equation*}
        \norm{\sigma_n^+}_{P} = \lim_{m\to\infty}\norm{\eta_{m,n}^+}_P = 0,
    \end{equation*}
    and thus $\sigma_n^+ = 0$, $P$-\gls{as}, concluding the proof.
\end{proof}
\ifpreprint\subsection{Proofs for Section 5}\fi
\ville*
\begin{proof}[Proof of \cref{thm:asymptotic-ville}]
    We show the result for $r$-\glspl{saep}, as the result for $r$-\glspl{aep} follows by applying the stronger result to the $r$-\gls{saep} $(E_{m,n}\pwMin t_m)_{m,n\in\N}$.
    For all $\alpha\in(0,1)$ and $m\in\N$, define \begin{equation*}
        \tau_m := \inf\left\{n\in\integers{r_m}~\middle|~E_{m,n}\geq\frac1\alpha\right\},
    \end{equation*}
    with the convention that $\tau_m = r_m$ if the set is empty.
    Since $(E_{m,n})_{n\in\N}$ is adapted to $\filtration_{m,\bullet}$, $\tau_m$ is an $\filtration_{m,\bullet}$-stopping time.
    Furthermore, it is certainly less than $r_m$, $P$-\gls{as} for all $P\in\subsetProbaMeasures$, and thus $\tau\in\rBoundedStoppingTimes(r,\filtration,\subsetProbaMeasures)$.
    By assumption on $E$, it follows that $\limsup_{m\to\infty}\sup_{P\in\subsetProbaMeasures}\E_P[E_{m,\tau_m}]\leq 1$.
    Yet, for all $P\in\subsetProbaMeasures$ and $m\in\N$, \begin{equation*}
        \left\{\sup_{n\in\integers{r_m}} E_{m,n}\geq\frac1\alpha\right\}=\left\{E_{m,\tau_m}\geq\frac1\alpha\right\},
    \end{equation*}
    by definition of $\tau_m$.
    By Markov's inequality applied to $E_{m,\tau_m}$, it follows that \begin{equation*}
        P\left[\sup_{n\in\integers{r_m}}E_{m,n}\geq\frac1\alpha\right] = P\left[E_{m,\tau_m}\geq\frac1\alpha\right]\leq \frac1\alpha\E_P[E_{m,\tau_m}].
    \end{equation*}
    Taking the supremum over $P\in\subsetProbaMeasures$ and the limit superior over $m\in\N$ concludes the proof.
\end{proof}
\ifpreprint\subsection{Proofs for Section 6}\fi
\cumulativeProduct*
\begin{proof}[Proof of \cref{thm:cumulative product}]
    We reintroduce the notation \begin{equation*}
        \eta_{m,n} = \E_P[E_{m,n+1}\mid\filtration_{m,n}] - E_{m,n},
    \end{equation*}
    first used in \cref{def:asp} with $m,n\in\N$ and $P\in\subsetProbaMeasures$ and define \begin{equation*}
        \epsilon_{m,n} = \E_P[e_{m,n+1}\mid\filtration_{m,n}] - 1.
    \end{equation*}
    We have $\eta_{m,n} = E_{m,n}\cdot\epsilon_{m,n}$, and thus by nonnegativity of $E$, \begin{equation*}
        \eta_{m,n}^+ = E_{m,n}\cdot\epsilon_{m,n}^+.
    \end{equation*}
    The criterion of the \gls{asp} for $E$ is that $\eta_{m,n}^+\to 0$ in $L_1$ as $m\to\infty$ for all $n\in\N$.
    This can be established under \cref{eq:conditionally uniformly SAEV}.
    Specifically, it follows from \cref{eq:conditionally uniformly SAEV} that \begin{equation*}
        \E_P[\eta_{m,n}^+] = \E_P[E_{m,n}\cdot\epsilon_{m,n}^+] = \E_P[E_{m,n}\E_{P}[\epsilon_{m,n}^+\mid\filtration_{m,n}]]\leq \E_P[E_{m,n}]\varepsilon_m.
    \end{equation*}
    Here, the second equality follows from the tower property and $\filtration_{m,n}$-measurability of $E_{m,n}$ and the last one from the fact that $\epsilon_{m,n}\leq \varepsilon_m$, \gls{as}, by \cref{eq:conditionally uniformly SAEV}.
    Furthermore, for all $n\geq 1$,\begin{align*}
        \E_P[E_{m,n}] &= \E_P\left[E_{m,n-1}\cdot e_{m,n}\right]\\
            &= \E_P[\E_P[E_{m,n-1}\cdot e_{m,n}\mid\filtration_{m,n-1}]]\\
            &= \E_P[E_{m,n-1}\cdot \E_P[e_{m,n}\mid\filtration_{m,n-1}]]\\
            &\leq\E_P[E_{m,n-1}]\cdot(1+\varepsilon_m),
    \end{align*}
    and thus by induction \begin{equation}\label{eq:induction outcome}
        \E_P[E_{m,n}]\leq \E_P[E_{m,0}]\cdot(1+\varepsilon_m)^n = \E_P[e_{m,0}]\cdot(1+\varepsilon_m)^n,\quad\forall n\in\N.
    \end{equation}
    Putting it all together, for all $m,n\in\N$ and $P\in\subsetProbaMeasures$,\begin{equation*}
        \E_P[\eta_{m,n}^+]\leq \E_P[e_{m,0}]\cdot \varepsilon_m^P\cdot(1+\varepsilon_m^P)^n.
    \end{equation*}
    Under the assumptions that $\sup_{P\in\subsetProbaMeasures}\E_P[e_{\bullet,0}]$ is bounded and that $\varepsilon_m\to 0$ as $m\to\infty$, it follows that $E$ has the \gls{asp}.
    Assume additionally that $e_{\bullet,0}$ is a uniformly strongly asymptotic e-variable.
    This is equivalent to $E$ being asymptotically calibrated since $\E_P[e_{m,0}] = \E_P[E_{m,0}]$, for all $m\in\N$, and $e_{\bullet,n}$ is then a uniformly strongly asymptotic e-variable for all $n\in\N$ by \cref{eq:induction outcome}.
    We find sequences $r$ that make $E$ an $r$-\gls{saep} by leveraging \cref{thm:asp implies asymptotic e-process}.
    For any sequence $r = (r_m)_{m\in\N}\subset\N$ and $m\in\N$, \begin{align*}
        \sup_{P\in\subsetProbaMeasures}\E_P\left[\sum_{n=0}^{r_m-1}\eta_{m,n}^+\right] 
            &\leq \sup_{P\in\subsetProbaMeasures}\E_P[E_{m,0}]\cdot \sum_{n=0}^{r_m-1}\varepsilon_m\cdot(1+\varepsilon_m)^n\\
            &=\sup_{P\in\subsetProbaMeasures}\E_P[E_{m,0}]\cdot \left[(1+\varepsilon_m)^{r_m} - 1\right]
    \end{align*}
    It follows that a sufficient condition for $E$ to be an $r$-asymptotic e-process is that $(1+\varepsilon_m)^{r_m}\to 1$ as $m\to\infty$.
    Simple manipulations show that this holds when $r_m\cdot \varepsilon_m \to 0$ as $m\to\infty$, concluding the proof.
\end{proof}

\domination*
\begin{proof}[Proof of \cref{thm:domination}]
    Let $t\in\Rnn$ and $\tau = (\tau_m)_{m\in\N}\in\rBoundedStoppingTimes(r,\filtration,\subsetProbaMeasures)$.
    For all $m\in\N$ and $P\in\subsetProbaMeasures$, \begin{equation*}
        \E_P[E_{m,\tau_m}\pwMin t] 
            = \E_P[(E_{m,\tau_m}\pwMin t)\indicator_{A_m}] + \E_P[(E_{m,\tau_m}\pwMin t)\indicator_{A_m^\complement}]
            \leq \E_P[E_{m,\tau_m}^\ast\pwMin t] + t P[A_m^\complement].
    \end{equation*}
    The claim that $E$ is an $r$-\gls{aep} follows by taking the supremum over $P\in\subsetProbaMeasures$ and the limit superior over $m\in\N$.
    Furthermore, for a sequence $(t_m)_{m\in\N}$ as described, applying what precedes with $t = t_m$ immediately shows that $(E_{m,n}\pwMin t_m)_{m,n\in\N}$ is indeed an $r$-\gls{saep}.
    Finally, the existence of such a sequence $(t_m)_{m\in\N}$ follows a from diagonal argument similar to that of \cref{clry:rate of uniform boundedness 2}.
\end{proof}

\dominationViaMonotonicity*
\begin{proof}[Proof of \cref{clry:domination via monotonicity}]
    The conclusion follows from an application of \cref{thm:domination}, noting that $D_m\subset A_m$, where $A_m$ is introduced in \cref{thm:domination}, $m\in\N$.
    The fact that $E$ is adapted to $\filtration$ results from the measurability assumption on $(\omega,\theta)\mapsto E_{m,n}(\theta)(\omega)$, $m,n\in\N$, and the fact that $\bar\theta$ is adapted to $\filtration$.
\end{proof}

\overestimationViaConvergence*
\begin{proof}[Proof of \cref{lemma:overestimation via convergence}]
    We focus on the case with \gls{as} uniform in $\subsetProbaMeasures$ convergence; the other case follows similarly.
    The existence of a sequence $\epsilon$ as announced follows from a diagonal argument; e.g., applying \cref{clry:rate of uniform boundedness 2} to the family $f_m:\epsilon\in\Rp\mapsto a_m^\prime(\epsilon)+1$.
    This yields a sequence $(\epsilon_m)_{m\in\N}$ such that $\limsup_{m\to\infty}a_m^\prime(\epsilon_m)+1\leq 1$, from where it follows immediately that $a_m^\prime(\epsilon_m)\to 0$ as $m\to\infty$.
    Then, the existence of $t = (t_m)_{m\in\N}$ as announced follows again from a diagonal argument, applying for instance \cref{clry:rate of uniform boundedness 2} again to $f_m^\prime:t\in\Rnn\mapsto t\cdot a_m^\prime(\epsilon_m)+1$.
    Let $b=(b_{m,n})_{m,n\in\N}$ be an integer-valued process, and introduce $\hat\theta_{m,m+b_{m,n}}$ for $m,n\in\N$, as well as the events \begin{align*}
        B_m &= \left\{\forall n\in\N,~\lvert\hat\theta_{m+b_{m,n}}-\theta^\ast\rvert \leq \epsilon_m\right\},\\
        C_m &= \left\{\lvert\sup_{k\geq m}\hat\theta_{k}-\theta^\ast\rvert \leq \epsilon_m\right\}.
    \end{align*}
    It follows from the fact that $m+b_{m,n}\geq m$ for all $n\in\N$ that $C_m\subset B_m$.
    But it holds that, on $B_m$, $\hat\theta_{m+b_{m,n}} - \theta^\ast \geq -\epsilon_{m}$, and thus $\bar\theta_{m,n}\geq\theta^\ast$, for all $n\in\N$.
    As a result, $C_m\subset B_m\subset D_m$, where $D_m$ is defined in \cref{clry:domination via monotonicity}.
    It follows that \begin{equation*}
        \lim_{m\to\infty}\inf_{P\in\subsetProbaMeasures} P[D_m] \geq \lim_{m\to\infty}\inf_{P\in\subsetProbaMeasures} P[C_m] = 1,
    \end{equation*}
    and the conclusion follows from \cref{clry:domination via monotonicity}.
\end{proof}

\burnIn*
\begin{proof}[Proof of \cref{thm:r-aep from burn in of single data stream}]
    We verify that $\bar\theta$ is adapted to the filtration sequence $\filtration$.
    We focus on the case of \gls{as} convergence uniformly in $\subsetProbaMeasures$, as it is the more challenging one.
    For every $m,n\in\N$, the variable $\hat\theta_{m\pwMax n}$ is $\mathcal G_{m\pwMax n}$-measurable by assumption on $\hat \theta$.
    Consequently, $\bar\theta_{m,n}:=\hat\theta_{m\pwMax n}+\epsilon_m$ is as well.
    Furthermore, $(\omega,\theta)\mapsto E_n(\theta)(\omega)$ is $\mathcal G_n\otimes\mathcal B(\R)$-measurable by assumption, and thus is $\mathcal G_{m\pwMax n}\otimes\mathcal B(\R)$-measurable.
    Therefore, the composition $E_n(\bar\theta_{m,n})$ is $\mathcal G_{m\pwMax n}$-measurable, showing that $(E_n(\bar\theta_{m,n}))_{m,n\in\N}$ is adapted to $\filtration$.
    The rest follows immediately from \cref{lemma:overestimation via convergence} and \cref{clry:domination via monotonicity}.
\end{proof}

\anytimePValuesAreInftyApps*
\begin{proof}[Proof of \cref{thm:anytime p-values are infty-APPs}]
    We begin with the direct implication and take $\tau$ such a sequence of stopping times.
    One verifies that, for all $m\in\N$ and $\alpha\in(0,1)$,
    \begin{equation*}
        \{p_{m,m+\tau_m}\leq\alpha\}\subset\{\exists k\geq m, p_{m,k}\leq \alpha\}.
    \end{equation*}
    Taking the supremum over $P$ and the limit superior over $m$ shows the result.
    \par
    For the converse implication, define for all $m,n\in\N$,\begin{equation*}
        E_{m,n} = \alpha^{-1}\indicator_{\{p_{m,m+n}\leq \alpha\}}.
    \end{equation*}
    It follows from the assumption that $E$ is an $\infty$-\gls{saep} for $\bar\filtration$ and $\subsetProbaMeasures$.
    Indeed, $E$ is adapted to $\bar \filtration$, and for any $\sigma\in\rBoundedStoppingTimes_\frm(\infty,\bar\filtration_{m,\bullet},\subsetProbaMeasures)$ and $P\in\subsetProbaMeasures$, \begin{equation*}
        \E_P[E_{m,\sigma}] = \alpha^{-1}P[p_{m,m+\sigma}\leq\alpha].
    \end{equation*}
    Therefore, $E$ satisfies \cref{thm:characterizations r-SAEP}\labelcref{item:characterizations r-SAEP:ii}.
    It follows from \cref{thm:characterizations r-SAEP} that $E$ is an $\infty$-\gls{saep}, and thus \begin{equation*}
        \limsup_{m\to\infty}\sup_{P\in\subsetProbaMeasures} \E_P[E_{m,\tau_m}]\leq 1,
    \end{equation*}
    for all $(\tau_m)_{m\in\N}\in\rBoundedStoppingTimes(\infty,\bar\filtration,\subsetProbaMeasures)$, and not only in $\rBoundedStoppingTimes_\frm(\infty,\bar\filtration,\subsetProbaMeasures)$.
    In particular, this is the case for the sequence of stopping times defined as \begin{equation*}
        \tau_m = \inf\{n\in\N\mid p_{m,m+n}\leq\alpha\}.
    \end{equation*}
    Deviating exceptionally from the convention \cref{eq:convention evaluation at infinity}, define $p_{m,\infty} = \liminf_{n\in\N} p_{m,m+n}$.
    It holds that, for all $m\in\N$ \begin{equation*}
        E_{m,\tau_m} = \alpha^{-1}\indicator_{\{p_{m,m+\tau_m}\leq\alpha\}},
    \end{equation*}
    pointwise.
    Indeed, this is clear by definition of $E$ on the event $\{\tau_m<\infty\}$, and we have the equivalences \begin{align*}
        E_{m,\infty} = \alpha^{-1} 
        &\iff \limsup_{n\to\infty}\indicator_{\{p_{m,m+n}\leq\alpha\}} = 1\\
        &\iff \forall N\in\N, \exists n\geq N, p_{m,m+n}\leq\alpha\\
        &\iff \forall N\in\N, \inf_{n\geq N} p_{m,m+n}\leq\alpha\\
        &\iff \lim_{N\to\infty} \inf_{n\geq N} p_{m,m+n}\leq\alpha\\
        &\iff \liminf_{n\to\infty} p_{m,m+n}\leq\alpha.
    \end{align*}
    Consequently, for all $P\in\subsetProbaMeasures$ and $m\in\N$,\begin{align*}
        \E_P[E_{m,\tau_m}] 
            = \alpha^{-1}P[p_{m,m+\tau_m}\leq\alpha]
            = \alpha^{-1}P[\exists n\in\N, p_{m,m+n}\leq \alpha],
    \end{align*}
    where the second equality results from the definition of $\tau_m$.
    Taking the supremum over $P\in\subsetProbaMeasures$ and the limit superior as $m\to\infty$ concludes the proof.
\end{proof}

\calibration*
\begin{proof}[Proof of \cref{thm:calibration}]
    Let $\tau = (\tau_m)_{m\in\N}\in\rBoundedStoppingTimes_\frm(r,\filtration,\subsetProbaMeasures)$.
    If $p$ is an $r$-\gls{app}, then it follows from \citet[Proposition 3.12]{ignatiadisAsymptoticCompoundEvalues2024} that $(E_{m,\tau_m})_{m\in\N}$ is a uniformly asymptotic e-variable.
    Since this holds for all $\tau$, this shows from \cref{clry:characterizations r-AEP} that $E$ is an $r$-\gls{aep}.
    If $p$ is an $r$-\gls{sapp}, the same reasoning but invoking \cref{thm:characterizations r-SAEP} shows that it is an $r$-\gls{saep}.
    The last case is when $p$ is an $r$-\gls{app} but not an $r$-\gls{sapp}, and $f$ is bounded.
    Then, it follows that $(E_{m,n}\pwMin t)_{m,n\in\N}$ is an $r$-\gls{saep} for all $t\in\Rnn$, by what precedes.
    This holds in particular for any $t > \lVert f\rVert_\infty$, for which $E_{m,n} \pwMin t = E_{m,n}$ for all $m,n\in\N$, showing the result.
\end{proof}
\section{Auxiliary results}
We collect in this section auxiliary technical results that are useful in the proofs of this paper.
\paragraph*{Rate of uniform boundedness\preprintdot}
The first two lemmas are diagonal arguments on double-indexed families that are bounded asymptotically along one axis.
\begin{lemma}\label{lemma:rate of uniform boundedness}
Let $(x_{m,n})_{m,n\in\mathbb{N}}$ be a doubly indexed family of real numbers such that
\begin{equation*}
    \limsup_{m\to\infty} x_{m,n}\le 1
    \text{ for every } n\in\mathbb{N}.
\end{equation*}
Then there exists a nondecreasing sequence 
$(r_m)_{m\in\mathbb{N}} \subset \N$ 
with $r_m\to\infty$ such that
\begin{equation*}
    \limsup_{m\to\infty} x_{m,r_m}\le 1.
\end{equation*}
\end{lemma}
\begin{proof}
For each $n \in \mathbb{N}$, the assumption
\begin{equation*}
    \limsup_{m\to\infty} x_{m,n} \le 1
\end{equation*}
means that there exists $N_n \in \mathbb{N}$ such that
\begin{equation*}
x_{m,n} \le 1 + \frac{1}{n}
\text{ for every }
m \ge N_n.
\end{equation*}
By increasing the numbers $N_n$ recursively if necessary, we may additionally assume that the sequence $(N_n)$ is strictly increasing and $N_n \to \infty$.

Now define $(r_m)_{m\in\mathbb{N}}$ by setting $r_m = 1$ for $m < N_1$, and, for $m \ge N_1$, letting
\begin{equation*}
r_m = n \qquad \text{whenever } N_n \le m < N_{n+1}.
\end{equation*}
Since $(N_n)$ is strictly increasing, the sequence $(r_m)_{m\in\N}$ is well defined and nondecreasing. Moreover, because $N_n \to \infty$, we also have $r_m \to \infty$ as $m \to \infty$.

With this construction, we have
\begin{equation*}
x_{m,r_m} = x_{m,n} \le 1 + \frac{1}{n} = 1 + \frac{1}{r_m}.
\end{equation*}
Since $r_m \to \infty$, we have $1 + \frac{1}{r_m} \to 1$. Therefore
\begin{equation*}
\limsup_{m\to\infty} x_{m,r_m} \le 1.
\end{equation*}
This concludes the proof.
\end{proof}
\begin{corollary}\label{clry:rate of uniform boundedness 2}
    For all $m\in\N$, let $f_m:\Rnn\to\Rnn$ be a family of functions.
    If for all $t\in\Rnn$, $\limsup_{m\to\infty}f_m(t)\leq 1$, then there exists $(t_m)_{m\in\N}\subset\Rnn$ with $t_m\to\infty$ as $m\to\infty$ such that $\limsup_{m\to\infty}f_m(t_m)\leq 1$.
    If $f_m$ is nondecreasing for all $m\in\N$, then the converse implication also holds.
\end{corollary}
\begin{proof}
    We begin with the direct implication.
    Define $x_{m,n} = f_m(n)$ for all $m,n\in\N$.
    By assumption, $x$ satisfies the assumptions of \cref{lemma:rate of uniform boundedness}, yielding the existence of a sequence $(t_m)_{m\in\N}$ such that $t_m\to\infty$ as $m\to\infty$ and $\limsup_{m\to\infty} f_m(t_m) = \limsup_{m\to\infty}x_{m,t_m}\leq 1$.
    Assume now that $f_m$ is nondecreasing for all $m\in\N$, and take $(t_m)_{m\in\N}$ as given in the assumption.
    Let $t\in\R$; there exists $m_0\in\N$ such that $t_m\geq t$ for all $m\geq m_0$ since $t_m\to\infty$ as $m\to\infty$.
    In particular, $f_m(t)\leq f_m(t_m)$ for all $m\geq m_0$, and the converse implication follows by taking the limit superior over $m\geq m_0$.
\end{proof}

\paragraph*{Optional sampling theorem for nonnegative supermartingales}
The second result is a strengthening of the optional sampling theorem in the case of nonnegative supermartingales.
The classical result requires bounded stopping times, and is commonly extended to \gls{as} finite ones; see for instance \citet[Theorem 10.11]{klenke2020probability}.
We require allowing infinite values.
It is known that the result also extends to that case under an additional uniform integrability requirement; see for instance the discussion after (6) in \cite{ramdasAdmissibleAnytimevalidSequential2022}.
A formal statement can be adapted from \citet[Theorem 9.30]{kallenbergFoundations2021}, which involves submartingales on $\R^+$.
In the interest of simplicity, however, we now provide a self-contained statement and proof of the special case that we use in this work.
\begin{theorem}\label{thm:optional sampling nonnegative supermartingales}
    Let $P\in\probaMeasures(\Omega)$ and $X = (X_n)_{n\in\N}$ be a nonnegative $\filtration$-supermartingale on $(\Omega,\sigAlg,P)$, where $\filtration$ is a filtration.
    Let $\tau$ be an $\filtration$-stopping time, possibly infinite (meaning that $P[\tau=\infty]$ may be positive).
    Then, $X_\tau$ is integrable and \begin{equation*}
        \E_P[X_\tau]\leq\E_P[X_0],
    \end{equation*}
    where we recall the convention $X_\infty := \limsup_{n\to\infty}X_n$.
\end{theorem}
\begin{proof}
    First, by Doob's supermartingale convergence theorem \cite[see e.g.][Theorem 11.4]{klenke2020probability}, the limit superior in the definition of $X_\infty$ is actually a true limit and convergence holds \gls{as}
    As a result, it holds \gls{as} that \begin{equation*}
        \lim_{n\to\infty}X_{\tau\pwMin n} = X_\tau.
    \end{equation*}
    Next, $\tau_n := \tau\pwMin n$ is a bounded stopping time.
    It follows from the optional sampling theorem \cite[Theorem 10.11]{klenke2020probability} that \begin{equation*}
        \E_P[X_{\tau_n}]\leq \E_P[X_0].
    \end{equation*}
    We can thus conclude by Fatou's lemma: \begin{equation*}
        \E_P[X_\tau] = \E_P\left[\lim_{n\to\infty}X_{\tau_n}\right]\leq\liminf_{n\to\infty}\E_P[X_{\tau_n}]\leq\E_{P}[X_0].
    \end{equation*}
    In particular, $X_\tau$ is integrable by nonnegativity.
\end{proof}
To extend this result to supermartingales that are lower-bounded by integrable variables, we leverage the following result that shows equality in the case of a martingale of conditional expectations.
\begin{theorem}\label{thm:optional sampling levy martingale}
    Let $P\in\probaMeasures(\Omega)$, $X$ be a $P$-integrable real variable, $\filtration = (\filtration_{n})_{n\in\N}$ be a filtration, and define the martingale \begin{equation*}
        \forall n\in\N,\quad X_n = \E_{P}[X\mid\filtration_n].
    \end{equation*}
    Then, for all $\filtration$-stopping time $\tau$, possibly infinite, $X_\tau$ is integrable and $\E_P[X_\tau] = \E_P[X]$.
\end{theorem}
\begin{proof}
    Let $\tau$ be an $\filtration$-stopping time.
    Introduce $U_n = \E_P[\lvert X\rvert\mid\filtration_n]$, as well as the stopped processes $S_n = X_{\tau\pwMin n}$ and $T_n = U_{\tau\pwMin n}$.
    Since $X_\bullet$ and $U_\bullet$ are martingales by the tower property, $S$ and $T$ are both martingales as well by \citet[Theorem 10.15]{klenke2020probability}.
    We show that $S$ converges \gls{as} to $X_\tau$ and is uniformly integrable, which will show that it converges to $X_\tau$ in $L_1(\Omega,\sigAlg,P)$ by \citet[Theorem 4.6.3]{durrett2019probability}.
    The \gls{as} convergence of $S$ to $X_\tau$ is clear after noticing that $X_n\to \E_P[X\mid\filtration_\infty] = X_\infty$ as $n\to\infty$ both in $L_1$ and \gls{as}, by Lévy's upward theorem \citep[Theorem 14.2]{williams1991martingales}.
    Turning to the uniform integrability, it follows from Jensen's inequality for conditional expectations that for all $n\in\N$, \begin{equation*}
        \lvert S_n\rvert \leq \sum_{k=0}^n \indicator_{\{\tau_n = k\}} \lvert X_k\rvert \leq \sum_{k=0}^n \indicator_{\{\tau_n = k\}} \E_P[\lvert X\rvert\mid\filtration_{k}] = \E_P[\lvert X\rvert\mid\filtration_{\tau_n}]= T_n,\quad \text{\gls{as}}
    \end{equation*}
    Consequently, for all $K>0$, $\{\lvert S_n\rvert>K\}\subset\{T_n>K\}$, and thus \begin{align*}
        \E_P[\lvert S_n\rvert \indicator_{\{\lvert S_n\rvert >K\}}]
            &\leq \E_P[T_n\indicator_{\{T_n>K\}}]\\
            &= \E_P[\E_P[\lvert X\rvert\mid\filtration_{\tau_n}]\indicator_{\{T_n>K\}}]\\
            &= \E_P[\E_P[\lvert X\rvert\indicator_{\{T_n>K\}}\mid\filtration_{\tau_n}]]\\
            &= \E_P[\lvert X\rvert\indicator_{\{T_n>K\}}],
    \end{align*}
    where the second equality comes from the fact that $T_n$ is $\filtration_n$-measurable.
    Next, by Markov's inequality applied to $T_n$, \begin{equation*}
        P[T_n>K]\leq K^{-1}\E_P[T_n] = K^{-1} \E_P[T_0] = K^{-1}\E_P[\E_P[\lvert X\rvert\mid\filtration_0]] = K^{-1}\E_P[\lvert X\rvert],
    \end{equation*}
    where we applied the fact that $T$ is a martingale and the tower property.
    We conclude by using the fact that, for every $\epsilon>0$, there exists $\delta>0$ such that for all $A\in\sigAlg$ such that $P[A]\leq\delta$, then $\E_P[\lvert X\rvert \indicator_A]\leq \epsilon$.
    Indeed, let $\epsilon>0$, and pick $t>0$ such that $\E_P[\lvert X\rvert\indicator_{\{\lvert X\rvert >t\}}]\leq\epsilon/2$; such a $t$ necessarily exists by the \gls{as} monotone convergence of $\lvert X\rvert\indicator_{\{\lvert X\rvert >s\}}$ to $0$ as $s\to\infty$.
    Then, \begin{equation*}
        \E_P[\lvert X\rvert \indicator_A] 
            = \E_P[\lvert X\rvert \indicator_{A\cap\{\lvert X\rvert\leq t\}}] + \E_P[\lvert X\rvert \indicator_{A\cap\{\lvert X\rvert>t\}}].
            \leq t P[A] + \frac\epsilon2
    \end{equation*}
    Choosing $\delta = t^{-1}\epsilon/2$ guarantees that $\E_P[\lvert X\rvert \indicator_A] \leq \epsilon$ for all $A$ such that $P[A]\leq\delta$.
    Let then $\epsilon>0$, and introduce a suitable corresponding $\delta$ as described.
    By what precedes, there exists $K_0>0$ such that for all $K\geq K_0$ and $n\in\N$, $P[T_n>K]\leq \delta$, and thus  \begin{equation*}
        \sup_{n\in\N}\E_P[\lvert S_n\rvert \indicator_{\{\lvert S_n\rvert >K\}}] \leq \sup_{n\in\N}\E_P[\lvert X\rvert \indicator_{\{\lvert T_n\rvert >K\}}] \leq \epsilon.
    \end{equation*}
    Since this is true for all $\epsilon>0$, this shows that \begin{equation*}
        \lim_{K\to\infty}\sup_{n\in\N}\E_P[\lvert S_n\rvert \indicator_{\{\lvert S_n\rvert >K\}}] = 0,
    \end{equation*}
    which is precisely uniform integrability of $S$.
    Consequently, $S$ converges to $X_\tau$ in $L_1(\Omega,\sigAlg,P)$ by \citet[Theorem 4.6.4]{durrett2019probability}, as announced.
    This shows that $X_\tau$ is integrable, and that \begin{equation*}
        \lim_{n\to\infty} \E_P[S_n] = \E_P[X_\tau].
    \end{equation*}
    Yet, since $S$ is a martingale, $\E_P[S_n] = \E_P[S_0] = \E_P[X_{\tau\pwMin0}] = \E_P[X_0]$ for all $n\in\N$, and the result follows.
\end{proof}
\begin{corollary}\label{clry:optional sampling lower bounded supermartingales}
    Let $X = (X_n)_{n\in\N}$ be an $\filtration$-supermartingale on $(\Omega,\sigAlg,P)$, where $\filtration$ is a filtration.
    Let $\tau$ be an $\filtration$-stopping time, possibly infinite, and $Y$ be a $P$-integrable real variable.
    If $X_n\geq Y$, $P$-\gls{as}, for all $n\in\N$, then $X_\tau$ is integrable and \begin{equation*}
        \E_P[X_{\tau}]\leq\E_P[X_0].
    \end{equation*}
    If, instead, $X_n\geq Y$, $P$-\gls{as}, for all $n\in\integers{N}$ for some $N\in\N$, then $X_{\tau\pwMin N}$ is integrable and \begin{equation*}
        \E_P[X_{\tau\pwMin N}]\leq\E_P[X_0].
    \end{equation*}
\end{corollary}
\begin{proof}
    We begin with the case where $X_n\geq Y$ for all $n\in\N$.
    For all $n\in\N$, define $Y_n = \E_P[Y\mid \filtration_n]$; it constitutes a martingale.
    Indeed, \begin{equation*}
        \forall n\in\N, \quad\E_P[Y_{n+1}\mid\filtration_{n}] = \E_P[\E_P[Y\mid\filtration_{n+1}]\mid\filtration_n] = \E_P[Y\mid\filtration_n] = Y_n.
    \end{equation*}
    Consequently, $Z_n := X_n - Y_n$ defines a nonnegative supermartingale, $n\in\N$.
    It follows from \cref{thm:optional sampling nonnegative supermartingales} that \begin{equation*}
        \E_{P}[X_{\tau} - Y_{\tau}] \leq \E_P[X_0 - Y_0].
    \end{equation*}
    Now, it follows from \cref{thm:optional sampling levy martingale} that $Y_\tau$ is integrable, and that $\E_P[Y_\tau] = \E_P[Y_0]$.
    We deduce that $X_{\tau} = (X_{\tau} - Y_\tau) + Y_\tau$ is integrable as the sum of integrable variables, and \begin{equation*}
        \E_P[X_\tau]\leq \E_P[X_0 - Y_0] + \E_P[Y_\tau] = \E_P[X_0],
    \end{equation*}
    which is the result.
    The case where $X_n\geq Y$ \gls{as} for all $n\in\integers N$ with $N\in\N$ follows immediately by applying what precedes to the stopped process $\bar X_n = X_{n\pwMin N}$, concluding the proof.
\end{proof}
\section{Characterization of uniformly asymptotic p-variables}\label{apdx:characterization uniformly asymptotic p-variables}
We prove in that section that \cref{def:uniformly asymptotic p-variables} is indeed an equivalent characterization of uniformly asymptotic p-variables and of uniformly strongly asymptotic p-variables as they are defined in the original work of \cite{ignatiadisAsymptoticCompoundEvalues2024}.
We begin by recalling the definition in this reference, while sticking to our convention of using the index $m\in\N$ for the approximation index.
\begin{definition}[Approximate p-variable]\label{def:approximate p-variable}
    Let $\epsilon:\subsetProbaMeasures\to\Rnn$ and $\delta:\subsetProbaMeasures\to[0,1]$ be functions.
    A nonnegative, finite random variable $p$ is an \emph{$(\epsilon,\delta)$-approximate p-variable} for $\subsetProbaMeasures$ if \begin{equation*}
        \forall P\in\subsetProbaMeasures, \forall \alpha\in(0,1),\quad P[p\leq \alpha]\leq(1+\epsilon(P))\alpha + \delta(P).
    \end{equation*}
\end{definition}
We introduce the notations $\mathcal E$ and $\mathcal D$ to denote the sets of functions from $\subsetProbaMeasures$ to $\Rnn$ and from $\subsetProbaMeasures$ to $[0,1]$, respectively.
\begin{theorem}\label{thm:characterization uniformly asymptotic p-variables}
    A nonnegative, finite process $p=(p_m)_{m\in\N}$ is 
    \begin{enumerate}[label=(\roman*)]
        \item \label{item:uniformly asymptotic p-variable} a \emph{uniformly asymptotic p-variable} if and only if for $\subsetProbaMeasures$ there exists $((\epsilon_m,\delta_m))_{m\in\N}\subset\mathcal E\times\mathcal D$ such that $p_m$ is $(\epsilon_m,\delta_m)$-approximate for $\subsetProbaMeasures$ for all $m\in\N$ and
        \begin{equation}
        \label{eq:uniform strongly asymp equivalent}
            \lim_{m\to\infty}\sup_{P\in\subsetProbaMeasures}\epsilon_m(P) = \lim_{m\to\infty}\sup_{P\in\subsetProbaMeasures}\delta_m(P) = 0;
        \end{equation}
        \item \label{item:uniformly strongly asymptotic p-variable} a \emph{uniformly strongly asymptotic p-variable} if
        and only if condition \cref{eq:uniform strongly asymp equivalent} is replaced by \begin{equation}\label{eq:definition uniformly strongly asymptotic p-variable}
            \lim_{m\to\infty}\sup_{P\in\subsetProbaMeasures}\epsilon_m(P) = 0 \quad\text{and}\quad \sup_{P\in\subsetProbaMeasures}\delta_m(P) = 0\quad\text{for $m$ large enough}.
        \end{equation}
    \end{enumerate}
\end{theorem}
\begin{proof}
    The fact that \cref{eq:uniformly asymptotic p-variable} is equivalent to \cref{thm:characterization uniformly asymptotic p-variables}\labelcref{item:uniformly asymptotic p-variable} is shown in \citet[Proposition 3.10]{ignatiadisAsymptoticCompoundEvalues2024}.
    We thus focus on showing that \cref{eq:uniformly strongly asymptotic p-variable} is equivalent to \cref{thm:characterization uniformly asymptotic p-variables}\labelcref{item:uniformly strongly asymptotic p-variable}, and begin with the converse implication.
    Let then $p = (p_m)_{m\in\N}$ be a nonnegative process satisfying \cref{thm:characterization uniformly asymptotic p-variables}\labelcref{item:uniformly strongly asymptotic p-variable}; we show that it also satisfies\begin{equation}\label{eq:characterization uniformly strongly asymptotic p-variable}
        \limsup_{m\to\infty}\sup_{P\in\subsetProbaMeasures}\sup_{\alpha\in(0,1)}\alpha^{-1}P[p_m\leq\alpha]\leq 1.
    \end{equation}
    Let $((\epsilon_m,\delta_m))_{m\in\N}\subset\mathcal E\times\mathcal D$ be such that $p_m$ is an $(\epsilon_m,\delta_m)$-approximate p-variable for $\subsetProbaMeasures$ for all $m\in\N$, and \cref{eq:definition uniformly strongly asymptotic p-variable} holds.
    There exists $m_0\in\N$ such that for all $m\geq m_0$, $\sup_{P\in\subsetProbaMeasures}\delta_m(P) = 0$.
    Consequently, by definition, for all $m\geq m_0$, $\alpha\in(0,1)$, and $P\in\subsetProbaMeasures$, \begin{align*}
        P[p_m\leq \alpha]\leq(1+\epsilon_m(P))\alpha + \delta_{m}(P) = (1+\epsilon_m(P))\alpha,
    \end{align*}
    showing that \begin{equation*}
        \alpha^{-1}P[p_m\leq\alpha]\leq 1+\epsilon_m(P).
    \end{equation*}
    Taking first the suprema over $\alpha\in(0,1)$ and $P\in\subsetProbaMeasures$ and then the limit superior over $m\geq m_0$ shows that \cref{eq:characterization uniformly strongly asymptotic p-variable} holds, where we leverage the fact that $\lim_{m\to\infty}\sup_{P\in\subsetProbaMeasures}\epsilon_m(P) = 0$.
    \par
    We now consider the direct implication, and assume that \cref{eq:characterization uniformly strongly asymptotic p-variable} holds.
    It follows immediately that there exists $m_0\in\N$ such that \begin{equation*}
        \forall m\geq m_0,\quad\sup_{P\in\subsetProbaMeasures}\sup_{\alpha\in(0,1)}\alpha^{-1} P[p_m\leq \alpha]<\infty
    \end{equation*}
    Define then for all $m\in\N$ and $P\in\subsetProbaMeasures$ \begin{align*}
        \delta_m(P) &= \begin{cases}
            1,&\text{if }m<m_0,\\
            0,&\text{otherwise,}
        \end{cases}\\
        \text{and}\quad
        \epsilon_m(P) &= \begin{cases}
            0,&\text{if }m< m_0,\\
            \max\{0,\sup_{\alpha\in(0,1)}\alpha^{-1}P[p_m\leq\alpha]-1\},&\text{otherwise}.
        \end{cases}
    \end{align*}
    One immediately verifies that $P[p_m\leq\alpha]\leq(1+\epsilon_m(P))\alpha + \delta_m(P)$, for all $\alpha\in(0,1)$, and that $\epsilon_m(P)$ is finite for all $m\in\N$ and $P\in\subsetProbaMeasures$, by construction.
    Furthermore, $\sup_{P\in\subsetProbaMeasures}\delta_m(P) = 0$ for all $m\geq m_0$, also by construction.
    All that is left is to verify that $\lim_{m\to\infty}\sup_{P\in\subsetProbaMeasures}\epsilon_m(P) = 0$.
    But this follows immediately from \cref{eq:characterization uniformly strongly asymptotic p-variable}; indeed, \begin{equation*}
        \limsup_{m\to\infty} \sup_{P\in\subsetProbaMeasures}\sup_{\alpha\in(0,1)}\alpha^{-1}P[p_m\leq\alpha] - 1\leq 0,
    \end{equation*}
    and therefore \begin{equation*}
        0\leq\limsup_{m\to\infty}\sup_{P\in\subsetProbaMeasures}\epsilon_m(P) = \max\{0,\limsup_{m\to\infty} \sup_{P\in\subsetProbaMeasures}\sup_{\alpha\in(0,1)}\alpha^{-1}P[p_m\leq\alpha] - 1\} = 0,
    \end{equation*}
    showing that $\lim_{m\to\infty}\sup_{P\in\subsetProbaMeasures}\epsilon_m(P)$ exists and is indeed $0$, concluding the proof.
\end{proof}
\end{document}